\documentclass[a4paper]{amsart}
\usepackage[utf8]{inputenc}
\usepackage[T1]{fontenc}
\usepackage{lmodern}
\usepackage{amssymb}
\usepackage[all]{xy}
\usepackage{nicefrac,mathtools,enumitem,amsmath}
\usepackage{microtype}
\usepackage{graphicx}

\usepackage{xcolor}

\usepackage{tikz}
\usetikzlibrary{matrix,arrows}
\tikzset{cd/.style=matrix of math nodes,row sep=2em,column sep=2em, text height=1.5ex, text depth=0.5ex}
\tikzset{cdar/.style=->,auto}

\setlist[enumerate,1]{label=\textup{(\arabic*)}}
\setlist[enumerate,2]{label=\textup{(\alph*)}}

\usepackage[pdftitle={Product systems over Ore monoids},
pdfauthor={Suliman Albandik and Ralf Meyer},
pdfsubject={Mathematics}
]{hyperref}
\usepackage[lite]{amsrefs}
\newcommand*{\MRref}[2]{ \href{http://www.ams.org/mathscinet-getitem?mr=#1}{MR #1}}
\newcommand*{\arxiv}[1]{ \href{http://www.arxiv.org/abs/#1}{arXiv:#1}}
\renewcommand{\PrintDOI}[1]{\href{http://dx.doi.org/#1}{DOI #1}}

\BibSpec{book}{%
  +{}  {\PrintPrimary}                {transition}
  +{,} { \textit}                     {title}
  +{.} { }                            {part}
  +{:} { \textit}                     {subtitle}
  +{,} { \PrintEdition}               {edition}
  +{}  { \PrintEditorsB}              {editor}
  +{,} { \PrintTranslatorsC}          {translator}
  +{,} { \PrintContributions}         {contribution}
  +{,} { }                            {series}
  +{,} { \voltext}                    {volume}
  +{,} { }                            {publisher}
  +{,} { }                            {organization}
  +{,} { }                            {address}
  +{,} { \PrintDateB}                 {date}
  +{,} { }                            {status}
  +{}  { \parenthesize}               {language}
  +{}  { \PrintTranslation}           {translation}
  +{;} { \PrintReprint}               {reprint}
  +{.} { }                            {note}
  +{.} {}                             {transition}
  +{,} { \PrintDOI}                   {doi}
  +{}  {\SentenceSpace \PrintReviews} {review}
}

\numberwithin{equation}{section}

\theoremstyle{plain}
\newtheorem{theorem}[equation]{Theorem}
\newtheorem{lemma}[equation]{Lemma}
\newtheorem{proposition}[equation]{Proposition}

\newtheorem{deflem}[equation]{Definition and Lemma}
\newtheorem{corollary}[equation]{Corollary}

\theoremstyle{definition}
\newtheorem{definition}[equation]{Definition}

\theoremstyle{remark}
\newtheorem{remark}[equation]{Remark}

\newtheorem{example}[equation]{Example}

\newcommand*{\braket}[2]{\langle#1, #2\rangle}
\newcommand*{\congto}{\xrightarrow\sim}
\newcommand*{\alb}{\hspace{0pt}} 

\newcommand*{\Mult}{\mathcal M}
\newcommand*{\U}{\mathcal U}

\newcommand*{\s}{s} 
\newcommand*{\rg}{r}
\newcommand*{\midpart}{\operatorname{mid}}

\newcommand{\C}{\mathbb{C}}
\newcommand{\N}{\mathbb{N}}
\newcommand{\Z}{\mathbb{Z}}
\newcommand{\R}{\mathbb{R}}
\newcommand{\T}{\mathbb{T}}

\newcommand*{\Cat}[1][C]{\mathcal #1}
\newcommand*{\Hilm}[1][E]{\mathcal #1}
\newcommand*{\Hils}[1][H]{\mathcal #1}
\newcommand*{\Gr}[1][G]{\mathcal #1}

\newcommand*{\CP}{\mathcal{O}}

\newcommand*{\diff}{\mathrm d}
\newcommand*{\Cont}{\mathrm C}
\newcommand*{\Contc}{\mathrm{C_c}}
\newcommand*{\op}{\mathrm{op}}

\newcommand*{\Star}{$^*$\nobreakdash-}
\newcommand*{\nb}{\nobreakdash}
\newcommand*{\Cst}{\mathrm C^*}
\newcommand*{\Cred}{\mathrm C^*_\mathrm r}

\newcommand*{\defeq}{\mathrel{\vcentcolon=}}
\newcommand{\ket}[1]{{\lvert#1\rangle}}
\newcommand{\bra}[1]{{\langle#1\rvert}}

\newcommand{\Comp}{\mathbb K}
\newcommand{\Bound}{\mathbb B}
\newcommand{\Mat}{\mathbb M}

\newcommand*{\abs}[1]{\lvert#1\rvert}
\newcommand*{\norm}[1]{\lVert#1\rVert}
\newcommand*{\conj}[1]{\overline{#1}}

\newcommand{\idealin}{\mathrel{\triangleleft}} 

\newcommand*{\Id}{\mathrm{id}}
\newcommand*{\K}{\mathrm{K}}
\newcommand*{\KK}{\mathrm{KK}}

\DeclareMathOperator{\pr}{pr}

\begin{document}
\title{Product systems over Ore monoids}
\author{Suliman Albandik}
\email{albandik@uni-goettingen.de}
\author{Ralf Meyer}
\email{rmeyer2@uni-goettingen.de}
\address{Mathematisches Institut\\
  Georg-August Universität Göttingen\\
  Bunsenstraße 3--5\\
  37073 Göttingen\\
  Germany}


\begin{abstract}
  We interpret the Cuntz--Pimsner covariance condition as a
  nondegeneracy condition for representations of product systems.  We
  show that Cuntz--Pimsner algebras over Ore monoids are constructed
  through inductive limits and section algebras of Fell bundles over
  groups.  We construct a groupoid model for the Cuntz--Pimsner
  algebra coming from an action of an Ore monoid on a space by
  topological correspondences.  We characterise when this groupoid is
  effective or locally contracting and describe its invariant subsets
  and invariant measures.
\end{abstract}
\maketitle

\section{Introduction}
\label{sec:intro}

Let \(A\) and~\(B\) be \(\Cst\)\nb-algebras.  A
\emph{correspondence} from~\(A\) to~\(B\) is a Hilbert
\(B\)\nb-module~\(\Hilm\) with a nondegenerate \Star{}homomorphism
from~\(A\) to the \(\Cst\)\nb-algebra of adjointable operators
on~\(\Hilm\).  It is called \emph{proper} if the left
\(A\)\nb-action is by compact operators, \(A\to \Comp(\Hilm)\).  If
\(\Hilm_{AB}\) and~\(\Hilm_{BC}\) are correspondences from~\(A\)
to~\(B\) and from~\(B\) to~\(C\), respectively, then \(\Hilm_{AB}
\otimes_B \Hilm_{BC}\) is a correspondence from~\(A\) to~\(C\).

A triangle of correspondences consists of three \(\Cst\)\nb-algebras
\(A\), \(B\), \(C\), correspondences \(\Hilm_{AB}\), \(\Hilm_{AC}\)
and~\(\Hilm_{BC}\) between them, and an isomorphism of
correspondences
\[
u\colon \Hilm_{AB} \otimes_{B} \Hilm_{BC} \to \Hilm_{AC};
\]
that is, \(u\) is a unitary operator of Hilbert
\(C\)\nb-modules that also intertwines the left \(A\)\nb-module
structures.  Such triangles appear naturally if we study the
correspondence bicategory of \(\Cst\)\nb-algebras introduced
in~\cite{Buss-Meyer-Zhu:Higher_twisted}.

This article started with the observation that a correspondence
triangle with \(A=B\) and \(\Hilm_{BC}=\Hilm_{AC}\) is the same as a
\emph{Cuntz--Pimsner covariant} representation of the correspondence
\(\Hilm\defeq \Hilm_{AB}\) by adjointable operators on
\(\Hilm[F]\defeq \Hilm_{BC}=\Hilm_{AC}\), provided~\(\Hilm_{AB}\) is
proper.  Thus we get to the Cuntz--Pimsner algebra directly, without
going through the Cuntz--Toeplitz algebra.

This is limited, however, to proper correspondences and the absolute
Cuntz--Pimsner algebra; that is, we cannot treat the relative
Cuntz--Pimsner algebras introduced by Muhly and
Solel~\cite{Muhly-Solel:Tensor} and
Katsura~\cite{Katsura:Cstar_correspondences}.  The relative versions
are most relevant if the left action map \(A\to\Comp(\Hilm)\) is not
faithful.  Then the map from~\(A\) to the Cuntz--Pimsner algebra is
not faithful, and the latter may even be zero.

Our observation about the Cuntz--Pimsner algebra of a single proper
correspondence has great conceptional value because it exhibits these
(absolute) Cuntz--Pimsner algebras as a special case of a general
construction, namely, colimits in the correspondence bicategory,
see~\cite{Albandik-Meyer:Colimits}.  Other examples of such colimits
are crossed products for group and crossed module actions, inductive
limits for chains of \Star{}homomorphisms, and Cuntz--Pimsner algebras
for proper essential product systems.

In this article, we apply our observation on the Cuntz--Pimsner
covariance condition to the case of Cuntz--Pimsner algebras for proper
essential product systems over monoids.  Much less is known about
their structure.  Following Fowler~\cite{Fowler:Product_systems}, they
are always defined and treated through the corresponding
Cuntz--Toeplitz algebra.  Many articles never get farther than the
Nica--Toeplitz algebra.  We shall prove strong results about the
structure of Cuntz--Pimsner algebras of proper essential product
systems over Ore monoids.  Commutative monoids and groups are Ore.  So
are extensions of commutative monoids by groups such as the monoid
\(\Mat_n(\Z)^\times\)
of integer matrices with non-zero determinant, with multiplication as
group structure: this is an extension of the group
\(\textup{Gl}_n(\Z)\)
by the commutative monoid~\((\N_{\ge1},\cdot)\).
Thus most of the semigroups currently being treated in the operator
algebras literature are Ore monoids.  The main exception are free
monoids, which are not Ore.

Let~\(P\)
be a cancellative Ore monoid and let~\(G\)
be its group completion.  Let~\(A\)
be a \(\Cst\)\nb-algebra
and let~\((\Hilm_g)_{p\in P}\)
be a proper, essential product system over~\(P\)
with unit fibre \(\Hilm_1=A\);
that is, the left action of~\(A\)
on each~\(\Hilm_p\)
for \(p\in P\)
is by a nondegenerate \Star{}homomorphism \(A\to\Comp(\Hilm_p)\).
The Ore conditions for~\(P\)
ensure that the diagram formed by the
\(\Cst\)\nb-algebras~\(\Comp(\Hilm_p)\) for \(p\in P\) with the maps
\[
\Comp(\Hilm_p) \to \Comp(\Hilm_p\otimes_A \Hilm_q) \cong
\Comp(\Hilm_{p q})
\]
for \(q,p\in P\) is indexed by a directed set.
Hence the colimit for this diagram behaves like an inductive limit; it
may indeed be rewritten as an inductive limit of a chain of maps
\(\Comp(\Hilm_{p_i})\to \Comp(\Hilm_{p_{i+1}})\) for a suitable map
\(\N\to P\) if~\(P\) is countable.  Let~\(\CP_1\) be the inductive
limit of this diagram of \(\Cst\)\nb-algebras.  We construct a Fell
bundle~\((\CP_g)_{g\in G}\) over~\(G\) with~\(\CP_1\) as its unit
fibre, such that its section algebra is the Cuntz--Pimsner
algebra~\(\CP\) of the given product system.  Thus the construction of
the Cuntz--Pimsner algebra of a product system over~\(P\) has two
steps: inductive limits and Fell bundle section algebras.

After putting this article on the arxiv, we learnt of the
preprint~\cite{Kwasniewski-Szymanski:Ore} by Kwa\'sniewski and
Szyma\'nski, which proves essentially the same result about the
structure of Cuntz--Pimsner algebras over Ore monoids.

Assume now that the correspondences~\(\Hilm_p\) are full as Hilbert
\(A\)\nb-modules.  Then the \(\Cst\)\nb-algebras~\(\Comp(\Hilm_p)\)
for \(p\in P\) are all Morita--Rieffel equivalent to~\(A\) and the
Fell bundle~\((\CP_g)_{g\in G}\) is saturated.  Let~\(\Comp\) be the
\(\Cst\)\nb-algebra of compact operators.  By the
Brown--Green--Rieffel Theorem, \(\Hilm_p\otimes\Comp \cong
A\otimes\Comp\) as a Hilbert \(A\otimes\Comp\)\nb-module, so we may
replace the proper correspondence~\(\Hilm_p\) by an endomorphism
\(\varphi_p\colon A\otimes\Comp \to A\otimes\Comp\).  Choose a cofinal
sequence~\((p_i)_{i\ge0}\) in~\(P\) as above with \(p_0\defeq 1\), and
let \(q_i\in P\) be such that \(p_i=p_{i-1}q_i\).  Then
\(\CP_1\otimes\Comp\) is the inductive limit of the inductive system
\[
A\otimes\Comp \xrightarrow{\varphi_{q_1}}
A\otimes\Comp \xrightarrow{\varphi_{q_2}}
A\otimes\Comp \xrightarrow{\varphi_{q_3}}
A\otimes\Comp \to \dotsb;
\]
this inductive limit carries a natural \(G\)\nb-action with \(G\ltimes
\CP_1\cong \CP\otimes\Comp\).

Thus the \(\K\)\nb-theory of~\(\CP_1\) is an inductive limit of
copies of the \(\K\)\nb-theory of~\(A\); the maps are induced by the
proper correspondences~\(\Hilm_q\) or, equivalently, the
endomorphisms~\(\varphi_q\) of~\(A\otimes\Comp\).  Roughly speaking,
we have reduced the problem of computing the \(\K\)\nb-theory for
Cuntz--Pimsner algebras of proper product systems over an Ore
monoid~\(P\) to the problem of computing the \(\K\)\nb-theory for
crossed products with the group~\(G\).  This latter problem may be
difficult, but is much studied.  We cannot hope for more because
crossed products for \(G\)\nb-actions are special cases of
Cuntz--Pimsner algebras over~\(P\).

Lots of \(\Cst\)\nb-algebras
are or could be defined as Cuntz--Pimsner algebras of product systems
over Ore monoids.  Thus our structure theory for them has lots of
potential applications.  As a sample of how K\nb-theory computations
in this context might work, we consider certain higher-rank analogues
of the Doplicher--Roberts algebras that motivated the introduction of
graph algebras.  (Our higher-rank analogues, however, need not be
higher-rank graph algebras.)

Many Cuntz--Pimsner algebras are constructed from generalised
dynamical systems, such as higher-rank topological graphs.  The
appropriate topological analogue of a product system over~\(P\) is
given by locally compact spaces \(X\) and~\(M_p\) for \(p\in P\) with
continuous maps \(\rg_p,\s_p\colon M_p \to X\) and
\(\sigma_{p,q}\colon M_{p q} \congto M_p\times_{\s_p,X,\rg_q} M_q\).
We assume~\(\rg_p\) to be proper and~\(\s_p\) to be local
homeomorphisms to turn~\((M_p,\s_p,\rg_p)\) into proper
correspondences over~\(\Cont_0(X)\).  These form a product system
over~\(P\) with unit fibre~\(\Cont_0(X)\).  The data above may be
called a topological higher-rank graph over~\(P\); we prefer to call
it an action of~\(P\) on~\(X\) by topological correspondences.

In the above situation, we construct a groupoid model for the
Cuntz--Pimsner algebra of our product system.  This model is a
Hausdorff, locally compact, étale groupoid.  We translate what it
means for this groupoid to be effective, locally contracting, or
minimal into the original data \((X,M_p,\s_p,\rg_p,\sigma_{p,q})\).
We also describe invariant subsets and invariant measures for the
object space of our groupoid model.  This gives criteria when the
Cuntz--Pimsner algebra of an action by topological correspondences is
simple or purely infinite and often describes its traces and
KMS-states for certain one-parameter groups of automorphisms.

Our results are interesting already for the commutative Ore
monoids~\((\N^k,+)\).  Several authors have considered examples of
product systems over these and other commutative cancellative monoids
(\cites{Brownlowe-Raeburn:Exel-Larsen,
  Farthing-Patani-Willis:Crossed-product,
  Exel-Renault:Semigroups_interaction, Larsen:Crossed_abelian,
  Yeend:Groupoid_models}).  Commutativity seems to be a red herring:
what is relevant are Ore conditions.  Commutativity is hidden also in
Exel's idea in~\cite{Exel:New_look} to extend a semigroup action to an
``interaction semigroup.''  Examples
in~\cite{Exel-Renault:Semigroups_interaction} show that interaction
semigroups for product systems over~\(\N^2\) only exist under some
commutativity assumptions about certain conditional expectations.  Our
approach shows that a deep study of these examples is possible without
such technical commutativity assumptions (see
Section~\ref{sec:partial_local_homeo}).

Topological higher-rank graphs are already very close to our
situation, so we compare our constructions to the existing ones for
this class as we go along.  As an example involving noncommutative Ore
monoids, we discuss how the semigroup \(\Cst\)\nb-algebras
of Xin Li~\cite{Li:Semigroup_amenability} fit into our approach.

\section{The Cuntz--Pimsner covariance condition}
\label{sec:endos_CP}

We first reinterpret the Cuntz--Pimsner covariance condition for a
single correspondence as a nondegeneracy condition.

\begin{definition}
  \label{def:correspondence}
  A \emph{correspondence} from~\(A\) to~\(B\) is a Hilbert
  \(B\)\nb-module~\(\Hilm[F]\) with a \emph{nondegenerate} left
  action of~\(A\) by adjointable operators.  A correspondence is
  \emph{proper} if~\(A\) acts by compact operators.  We often
  write the left action multiplicatively as \(a\cdot\xi\) for \(a\in
  A\), \(\xi\in \Hilm[F]\).  An \emph{isomorphism} between two
  correspondences from~\(A\) to~\(B\) is an \(A,B\)-bimodule map
  that is unitary for the \(B\)\nb-valued inner products.
\end{definition}

\begin{definition}
  \label{def:transformation_correspondence}
  Let \(A\) and~\(B\) be \(\Cst\)\nb-algebras and let~\(\Hilm\) be a
  correspondence from~\(A\) to itself.  A \emph{transformation}
  from~\((A,\Hilm)\) to~\(B\) is a correspondence~\(\Hilm[F]\)
  from~\(A\) to~\(B\) with an isomorphism of correspondences \(u\colon
  \Hilm\otimes_B \Hilm[F] \congto \Hilm[F]\).
\end{definition}

This definition is a special case of the standard notion of a
``transformation'' between two ``morphisms'' between two
``bicategories'' (see~\cite{Leinster:Basic_Bicategories}).  This
point of view is developed further
in~\cite{Albandik-Meyer:Colimits}.  Here it will not play any role
besides guiding our choice of notation.

We want to relate transformations to certain Toeplitz
representations of correspondences.  The next proposition is already
implicit in~\cite{Meyer:Generalized_Fixed}*{\S5} and has also been
used by other authors before.

\begin{proposition}
  \label{pro:unitary_to_correspondence_representation}
  Let \(A\), \(B_1\), \(B_2\) be \(\Cst\)\nb-algebras.  Let
  \(\Hilm\colon A\to B_1\), \(\Hilm[F]_1\colon B_1\to B_2\) and
  \(\Hilm[F]_2\colon A\to B_2\) be correspondences.
  Isomorphisms \(\Hilm\otimes_{B_1} \Hilm[F]_1\to\Hilm[F]_2\) of
  correspondences are in
  natural bijection with linear maps \(S\colon
  \Hilm\to\Bound(\Hilm[F]_1,\Hilm[F]_2)\) that satisfy
  \begin{enumerate}
  \item\label{en:representation1} \(S(a\xi) = a S(\xi)\) for all
    \(a\in A\), \(\xi\in\Hilm\);
  \item\label{en:representation2} \(S(\xi_1)^* S(\xi_2) = \langle
    \xi_1,\xi_2\rangle_{B_1}\) for all \(\xi_1,\xi_2\in\Hilm\);
  \item\label{en:representation3} \(S(\Hilm)\cdot\Hilm[F]_1\) spans a
    dense subspace of~\(\Hilm[F]_2\).
  \end{enumerate}
  Furthermore, \ref{en:representation2} implies
  \begin{enumerate}[resume]
  \item\label{en:representation4} \(S(\xi b)=S(\xi)b\) for all \(b\in
    B_1\).
  \end{enumerate}
\end{proposition}

\begin{proof}
  First let \(u\colon \Hilm\otimes_{B_1} \Hilm[F]_1\to\Hilm[F]_2\) be
  an isomorphism of correspondences.  Define \(S(\xi)(\eta)\defeq
  u(\xi\otimes\eta)\) for \(\xi\in\Hilm\), \(\eta\in\Hilm[F]_1\).  For
  fixed~\(\xi\), this is an adjointable operator \(S(\xi)\colon
  \Hilm[F]_1\to\Hilm[F]_2\) because~\(u\) and the operator
  \(\Hilm[F]_1\to\Hilm\otimes_{B_1} \Hilm[F]_1\),
  \(\eta\mapsto\xi\otimes\eta\), are adjointable.  This map~\(S\) clearly
  satisfies~\ref{en:representation1}.  Since~\(u\) is isometric,
  \[
  \langle \eta_1,S(\xi_1)^*S(\xi_2)\eta_2\rangle
  = \langle S(\xi_1)\eta_1,S(\xi_2)\eta_2\rangle
  = \langle \xi_1\otimes\eta_1,\xi_2\otimes\eta_2\rangle
  = \langle \eta_1,\langle\xi_1,\xi_2\rangle\eta_2\rangle
  \]
  for all \(\xi_1,\xi_2\in\Hilm\), \(\eta_1,\eta_2\in\Hilm[F]_1\).
  This is equivalent to~\ref{en:representation2}.  Since~\(u\) is
  unitary, it has dense range, which gives~\ref{en:representation3}.

  Conversely, let \(S\colon \Hilm\to\Bound(\Hilm[F]_1,\Hilm[F]_2)\) be
  given.  Define~\(u\) on the algebraic tensor product of \(\Hilm\)
  and~\(\Hilm[F]_1\) by linear extension of \(u(\xi\otimes\eta)\defeq
  S(\xi)(\eta)\).  Condition~\ref{en:representation2} ensures that
  this is an isometry and hence extends to the completion
  \(\Hilm\otimes_{B_1} \Hilm[F]_1\).  Hence~\(S\) satisfies \(S(\xi
  b)(\eta) = S(\xi)(b\eta)\) for all \(\xi\in\Hilm\), \(b\in B_1\),
  \(\eta\in\Hilm[F]_1\), which is equivalent
  to~\ref{en:representation4}.  Condition~\ref{en:representation1}
  says that~\(u\) is \(A\)\nb-linear, and~\ref{en:representation3}
  says that it has dense range.  Being isometric, this means
  that~\(u\) is unitary.  The two constructions \(u\leftrightarrow S\)
  are inverse to each other because~\(u\) is determined by its values
  on the monomials \(\xi\otimes\eta\).
\end{proof}

A (Toeplitz) \emph{representation} of~\(\Hilm\)
is usually defined as a map~\(S\)
satisfying \ref{en:representation1} and~\ref{en:representation2} in
the case \(\Hilm[F]_1=\Hilm[F]_2\);
we also allow the case \(\Hilm[F]_1\neq\Hilm[F]_2\)
for a while because this is used in~\cite{Albandik-Meyer:Colimits} and
in Proposition~\ref{pro:relative_CP}.

\begin{definition}
  \label{def:representation_correspondence}
  A representationn is \emph{nondegenerate} if it
  satisfies~\ref{en:representation3}.
\end{definition}

By Proposition~\ref{pro:unitary_to_correspondence_representation}, a
transformation from~\((A,\Hilm)\) to~\(B\) is equivalent to a
correspondence \(\Hilm[F]\colon A\to B\) with a nondegenerate
representation of~\(\Hilm\) by operators on~\(\Hilm[F]\).  We now
relate nondegeneracy to the Cuntz--Pimsner covariance condition:

\begin{proposition}
  \label{pro:CP_vs_nondegenerate}
  Nondegenerate representations of a correspondence~\(\Hilm\) are
  Cuntz--Pimsner covariant.  The converse holds if~\(\Hilm\) is
  proper.
\end{proposition}

\begin{proof}
  Let \(A\), \(B_1\) and~\(B_2\) be \(\Cst\)\nb-algebras and let
  \(\Hilm\colon A\to B_1\), \(\Hilm[F]_1\colon B_1\to B_2\) and
  \(\Hilm[F]_2\colon A\to B_2\) be correspondences as in
  Proposition~\ref{pro:unitary_to_correspondence_representation}.
  The left actions of~\(A\) in our correspondences are nondegenerate
  \Star{}homomorphisms
  \[
  \varphi_{\Hilm}\colon A\to\Bound(\Hilm),\qquad
  \varphi_{\Hilm[F]_2}\colon A\to\Bound(\Hilm[F]_2).
  \]

  Let \(u\colon \Hilm\otimes_{B_1}\Hilm[F]_1 \to \Hilm[F]_2\) be an
  isomorphism of correspondences.  The map
  \[
  \vartheta\colon \Bound(\Hilm) \to \Bound(\Hilm[F]_2),\qquad
  T\mapsto u(T\otimes 1)u^*,
  \]
  is a strictly continuous, unital \Star{}homomorphism.  It satisfies
  \(\vartheta\circ\varphi_{\Hilm}=\varphi_{\Hilm[F]_2}\) because~\(u\)
  intertwines the left actions of~\(A\).  If \(\xi_1,\xi_2\in\Hilm\)
  and \(\ket{\xi_1}\bra{\xi_2}\) is the corresponding rank-one compact
  operator on~\(\Hilm\), then
  \[
  \vartheta(\ket{\xi_1}\bra{\xi_2}) = S(\xi_1)S(\xi_2)^*.
  \]
  This formula still defines a (possibly degenerate) \Star{}homomorphism
  \(\vartheta\colon \Comp(\Hilm)\to \Bound(\Hilm[F]_2)\) for any
  representation \(S\colon \Hilm\to\Bound(\Hilm[F]_1,\Hilm[F]_2)\),
  see \cite{Pimsner:Generalizing_Cuntz-Krieger}*{p.~202}.

  \begin{definition}
    \label{def:CP_covariance}
    A representation~\(S\) is \emph{Cuntz--Pimsner covariant} if
    \(\vartheta(\varphi_{\Hilm}(a)) = \varphi_{\Hilm[F]_2}(a)\) for
    all \(a\in A\) with \(\varphi_{\Hilm}(a)\in \Comp(\Hilm)\).
  \end{definition}

  By Proposition~\ref{pro:unitary_to_correspondence_representation},
  a nondegenerate representation comes from an
  isomorphism of correspondences.  We have already seen
  \(\vartheta(\varphi_{\Hilm}(a)) = \varphi_{\Hilm[F]_2}(a)\) for all
  \(a\in A\) in that case.  So nondegenerate representations are
  Cuntz--Pimsner covariant.

  Conversely, let~\(S\) be Cuntz--Pimsner covariant and assume
  that~\(\Hilm\) is proper, that is, \(\varphi_{\Hilm}(A)\subseteq
  \Comp(\Hilm)\).  Let \(\langle X\rangle\) denote the closed linear
  span of~\(X\).  We have
  \begin{multline*}
    \langle S(\Hilm)\Hilm[F]_1 \rangle
    \supseteq   \langle S(\Hilm)S(\Hilm)^*\Hilm[F]_2 \rangle
    =   \langle \vartheta(\Comp(\Hilm))\Hilm[F]_2 \rangle
    \\\supseteq   \langle \vartheta(\varphi_{\Hilm}(A))\Hilm[F]_2 \rangle
    =   \langle \varphi_{\Hilm[F]_2}(A)\Hilm[F]_2\rangle
    =   \langle \Hilm[F]_2  \rangle
  \end{multline*}
  because~\(\varphi_{\Hilm[F]_2}\) is nondegenerate.  Thus~\(S\) is
  nondegenerate.
\end{proof}

For a proper correspondence~\(\Hilm\), we may now reformulate the
universal property that defines its Cuntz--Pimsner
algebra~\(\CP_{\Hilm}\): it is the universal \(\Cst\)\nb-algebra for
\emph{nondegenerate} representations of~\(\Hilm\).
Equivalently, \(\CP_{\Hilm}\) is the universal target for
transformations from~\((A,\Hilm)\) to \(\Cst\)\nb-algebras.  The
Cuntz--Pimsner algebra comes with a nondegenerate
\Star{}\alb{}homomorphism \(\varphi_0\colon A\to\CP_{\Hilm}\) and a
representation \(S_0\colon \Hilm\to\CP_{\Hilm}\), which is
Cuntz--Pimsner covariant and thus nondegenerate.  This is equivalent
to a transformation from~\((A,\Hilm)\) to~\(\CP_{\Hilm}\); the
underlying correspondence is~\(\CP_{\Hilm}\) itself as a Hilbert
\(\CP_{\Hilm}\)-module, with~\(A\) acting via~\(\varphi_0\).  The
isomorphism \(u_0\colon \Hilm\otimes_A \CP_{\Hilm} \cong \CP_{\Hilm}\)
is the unitary that corresponds to~\(S_0\) by
Proposition~\ref{pro:unitary_to_correspondence_representation}.

The transformation~\((\CP_{\Hilm},u_0)\) has the following universal
property: if \((\Hilm[F],u)\) is another transformation
from~\((A,\Hilm)\) to a \(\Cst\)\nb-algebra~\(B\), then there is a
unique representation \(\psi\colon \CP_{\Hilm}\to\Bound(\Hilm[F])\)
for which \(u= u_0 \otimes_{\psi} \Id_{\Hilm[F]}\).  Conversely, a
representation \(\psi\colon \CP_{\Hilm}\to\Bound(\Hilm[F])\)
provides a unitary \(u= u_0 \otimes_{\psi} \Id_{\Hilm[F]}\) from
\(\Hilm\otimes_A \Hilm[F] \cong \Hilm \otimes_A \CP_{\Hilm}
\otimes_{\CP_{\Hilm}} \Hilm[F]\) to \(\Hilm[F]\cong
\CP_{\Hilm}\otimes_{\CP_{\Hilm}} \Hilm[F]\).  The
pair~\((\Hilm[F],\psi)\) is the same as a correspondence
from~\(\CP_{\Hilm}\) to~\(B\).  Thus transformations
from~\((A,\Hilm)\) to~\(B\) are the same as correspondences
from~\(\CP_{\Hilm}\) to~\(B\).

What happens for a representation \(S\colon \Hilm\to\Bound(\Hilm[F])\)
that does not satisfy the Cuntz--Pimsner covariance condition?  The
construction in the proof of
Proposition~\ref{pro:unitary_to_correspondence_representation} still
gives a map \(u\colon \Hilm\otimes_A \Hilm[F]\to \Hilm[F]\), which is
an \(A,B\)-bimodule map and isometric for the \(B\)\nb-valued inner
product.  But this isometry~\(u\) need not be unitary, not even
adjointable.  Thus allowing all Toeplitz representations replaces the
unitary in the definition of a transformation by a possibly
non-adjointable isometry.

\begin{example}
  \label{exa:Cuntz_infinity}
  What goes wrong if~\(\Hilm\) is not proper?  Let us consider the
  simplest case, \(A=\C\) and \(\Hilm=\ell^2(\N)\).  In this case,
  no non-zero element of~\(A\) acts by a compact operator, so there
  is no difference between the Cuntz--Pimsner and the
  Cuntz--Toeplitz algebra.  A correspondence from~\(A\) to~\(B\) is
  the same as a Hilbert \(B\)\nb-module.  The Cuntz--Pimsner
  algebra~\(\CP_{\Hilm}\) is the famous Cuntz
  algebra~\(\CP_\infty\).  The identity map on~\(\CP_\infty\)
  corresponds to a Cuntz--Pimsner covariant representation
  \(S_0\colon\ell^2(\N)\to \CP_\infty\), which maps the basis
  vector~\(\delta_i\) to the generating isometry~\(S_i\).  The
  induced \Star{}homomorphism \(\Comp(\ell^2\N)\to\CP_\infty\) is
  degenerate, however, because~\(\CP_\infty\) is unital.  It
  corresponds to the isometry of Hilbert \(\CP_\infty\)\nb-modules
  \(\ell^2(\N)\otimes\CP_\infty \hookrightarrow \CP_\infty\),
  \(E_{ij}\otimes x\mapsto S_i x S_j^*\).  If this were adjointable,
  its range would be of the form \(p\CP_\infty\) for a projection
  \(p\in \CP_\infty\) because~\(\CP_\infty\) is unital.  Then
  \([1]+p=p\) in \(\K_0(\CP_\infty)\) because \(\CP_\infty\oplus
  (\ell^2(\N)\otimes\CP_\infty) \cong \ell^2(\N)\otimes\CP_\infty\),
  giving \([1]=0\) in \(\K_0(\CP_\infty)\), which is false.
\end{example}

Katsura's definition of a relative Cuntz--Pimsner algebra only
requires the Cuntz--Pimsner covariance condition on a certain ideal
\(K\idealin A\) that acts on~\(\Hilm\) by compact operators
(see \cite{Katsura:from_correspondences} or
\cite{Katsura:Cstar_correspondences}*{Definition 3.4}).  We may
reformulate this as a \emph{partial nondegeneracy} condition:

\begin{proposition}
  \label{pro:relative_CP}
  Let \(A\) and~\(B\) be \(\Cst\)\nb-algebras, let \(\Hilm\)
  and~\(\Hilm[F]\) be correspondences from~\(A\) to~\(A\) and
  from~\(A\) to~\(B\), respectively.  Let~\(K\) be an ideal in~\(A\)
  that acts on~\(\Hilm\) by compact operators.  A representation
  \(S\colon \Hilm\to\Bound(\Hilm[F])\) satisfies the Cuntz--Pimsner
  covariance condition on~\(K\) if and only if
  \(K\cdot S(\Hilm)\Hilm[F] = K\cdot \Hilm[F]\).  Equivalently, the
  isometry \(\Hilm\otimes_A\Hilm[F] \to \Hilm[F]\) induced by~\(S\)
  restricts to an isomorphism of correspondences
  \(K\Hilm\otimes_A\Hilm[F] \to K\Hilm[F]\).
\end{proposition}

\begin{proof}
  Proposition~\ref{pro:unitary_to_correspondence_representation} says
  that an isomorphism \(K\Hilm\otimes_A\Hilm[F] \to K\Hilm[F]\) is
  equivalent to a nondegenerate representation
  \(K\Hilm\to\Bound(\Hilm[F],K\Hilm[F])\).  Now apply
  Proposition~\ref{pro:CP_vs_nondegenerate} to the correspondences
  \(K\Hilm\colon K\to A\), \(\Hilm[F]\colon A\to B\), and
  \(K\Hilm[F]\colon K\to B\), so substitute
  \(K,A,B,K\Hilm,\Hilm[F],K\Hilm[F]\) for
  \(A,B_1,B_2,\Hilm,\Hilm[F]_1,\Hilm[F]_2\).  Since we assume~\(K\) to
  act by compact operators on~\(\Hilm\), the correspondence
  \(K\Hilm\colon K\to A\) is always proper.  So the nondegeneracy
  condition \(K\Hilm\cdot\Hilm[F] = K\Hilm[F]\) is equivalent to the
  Cuntz--Pimsner covariance condition for the restriction of the left
  action of~\(K\) to~\(K\Hilm[F]\).  That is,
  \(\vartheta(\varphi_{\Hilm}(k))\xi = \varphi_{\Hilm[F]}(k)\xi\) for
  all \(k\in K\) and \(\xi\in K\Hilm[F]\), with \(\vartheta\colon
  \Comp(\Hilm)\to\Bound(\Hilm[F])\) as in the proof of
  Proposition~\ref{pro:CP_vs_nondegenerate}.  It remains to show that
  this equality for all \(\xi\in K\Hilm[F]\) implies the same equality
  for all \(\xi\in \Hilm[F]\): the latter is the usual coisometry
  condition for the ideal~\(K\).  Let \(T_k\defeq
  \vartheta(\varphi_{\Hilm}(k)) - \varphi_{\Hilm[F]}(k)\) for \(k\in
  K\).  Both \(T_k\) and~\(T_k^* = T_{k^*}\) map~\(\Hilm[F]\) to
  \(K\Hilm[F] = K\Hilm\cdot \Hilm[F]\), and they vanish on~\(K\Hilm[F]\)
  by the above computation.  Therefore, \(\braket{T_k\xi}{T_k\xi} =
  \braket{\xi}{T_k^*T_k\xi} = 0\) for all \(\xi\in \Hilm[F]\).
\end{proof}

\section{Cuntz--Pimsner algebras of product systems over Ore monoids}
\label{sec:Ore_monoids}

Product systems over a monoid~\(P\) were introduced by
Fowler~\cite{Fowler:Product_systems}, inspired by previous definitions
by Arveson~\cite{Arveson:Continuous_Fock} and
Dinh~\cite{Dinh:Discrete_product}.  The following data is equivalent
to a product system in Fowler's sense with the mild extra condition
that each fibre be an essential left module over the unit fibre:
\begin{itemize}
\item a \(\Cst\)\nb-algebra~\(A\);
\item correspondences~\(\Hilm_p\) from~\(A\) to itself for all
  \(p\in P\setminus\{1\}\);
\item isomorphisms of correspondences
  \(\mu_{p,q}\colon \Hilm_p\otimes_A \Hilm_q \to \Hilm_{p q}\)
  for all \(p,q\in P\setminus\{1\}\),
  which are \emph{associative}, that is, the following diagram commutes for
  all \(p,q,t\in P\):
  \[
  \begin{tikzpicture}[baseline=(current bounding box.west)]
    \matrix (m) [cd,column sep=6em] {
      \Hilm_p\otimes_A \Hilm_q\otimes_A \Hilm_t&
      \Hilm_p\otimes_A \Hilm_{q t}\\
      \Hilm_{p q}\otimes_A \Hilm_t&
      \Hilm_{p q t}\\
    };
    \draw[cdar] (m-1-1) -- node {\(\Id_{\Hilm_p}\otimes_A \mu_{q,t}\)} (m-1-2);
    \draw[cdar] (m-1-1) -- node[swap] {\(\mu_{p,q} \otimes_A \Id_{\Hilm_t}\)} (m-2-1);
    \draw[cdar] (m-1-2) -- node {\(\mu_{p,q t}\)} (m-2-2);
    \draw[cdar] (m-2-1) -- node[swap] {\(\mu_{p q,t}\)} (m-2-2);
  \end{tikzpicture}
  \]
\end{itemize}
here we let \(\Hilm_1=A\), and we let \(\mu_{1,q}\)
and~\(\mu_{p,1}\) be the isomorphisms \(A\otimes_A \Hilm_q\cong
\Hilm_q\) and \(\Hilm_p\otimes_A A\cong \Hilm_p\) from the left and
right \(A\)\nb-module structures, respectively; this is needed to
write down~\(\mu_{p,q}\) if \(p\cdot q=1\) and to formulate the
associativity condition for \(\Hilm_p\otimes_A \Hilm_q\otimes_A
\Hilm_t \to\Hilm_{p q t}\) if \(p\cdot q=1\) or \(q\cdot t=1\).

Our main theorems will only hold if all correspondences~\(\Hilm_p\)
are proper.  Then we speak of a \emph{proper product system
  over~\(P\)}.

\begin{remark}
  \label{rem:left_right}
  A nondegenerate \Star{}homomorphism \(f\colon A\to B\) gives a
  proper correspondence~\(\Hilm_f\) from~\(A\) to~\(B\): take
  \(\Hilm_f=B\) with~\(A\) acting through~\(f\).  For two composable
  nondegenerate \Star{}homomorphisms, we have a natural isomorphism
  \(\Hilm_f\otimes_B \Hilm_g \cong \Hilm_{gf}\).  Due to this change
  in the order of products, an action of the opposite
  monoid~\(P^\op\) by ordinary nondegenerate \Star{}homomorphisms
  gives a product system over~\(P\).  The Cuntz--Pimsner algebra of
  this product system is the same as the crossed product for the
  original action by endomorphisms, see
  \cite{Fowler:Product_systems}*{Section~3}.

  The change from~\(P\) to~\(P^\op\) also explains why left Ore
  conditions are needed in
  \cites{Laca:Endomorphisms_back,Cuntz-Echterhoff-Li:K-theory} to
  study actions of~\(P\) by endomorphisms, while we will need right Ore
  conditions to study product systems over~\(P\).
\end{remark}

\begin{definition}
  \label{def:transformation_product_system}
  Let~\((A,\Hilm_p,\mu_{p,q})\) be a product system over~\(P\).  A
  \emph{transformation} from it to a \(\Cst\)\nb-algebra~\(B\)
  consists of a correspondence~\(\Hilm[F]\) from~\(A\) to~\(B\) and
  isomorphisms of correspondences \(V_p\colon \Hilm_p \otimes_A
  \Hilm[F]\to\Hilm[F]\) for \(p\in P\setminus\{1\}\), such that for
  all \(p,q\in P\setminus\{1\}\), the following diagram of
  isomorphisms commutes:
  \begin{equation}
    \label{eq:transformation_product_system}
    \begin{tikzpicture}[baseline=(current bounding box.west)]
      \matrix (m) [cd,column sep=6em] {
        \Hilm_p\otimes_A \Hilm_q\otimes_A \Hilm[F]&
        \Hilm_p\otimes_A \Hilm[F]\\
        \Hilm_{p q}\otimes_A \Hilm[F]&
        \Hilm[F]\\
      };
      \draw[cdar] (m-1-2) -- node {\(V_p\)} (m-2-2);
      \draw[cdar] (m-2-1) -- node[swap] {\(V_{p q}\)} (m-2-2);
      \draw[cdar] (m-1-1) -- node {\(\Id_{\Hilm_p}\otimes_A V_q\)} (m-1-2);
      \draw[cdar] (m-1-1) -- node[swap] {\(\mu_{p,q} \otimes_A \Id_{\Hilm[F]}\)} (m-2-1);
    \end{tikzpicture}
  \end{equation}
  We let~\(V_1\) be the canonical isomorphism \(A\otimes_A
  \Hilm[F]\cong\Hilm[F]\) and use this
  in~\eqref{eq:transformation_product_system} if \(p\cdot q=1\).
\end{definition}

By Proposition~\ref{pro:unitary_to_correspondence_representation},
each isomorphism~\(V_p\) corresponds to a nondegenerate
representation \(S_p\colon \Hilm_p\to \Bound(\Hilm[F])\) of the
correspondence~\(\Hilm_p\).  By convention, \(S_1\colon
\Hilm_1=A\to\Bound(\Hilm[F])\) is the representation of~\(A\) that
is part of the correspondence~\(\Hilm[F]\).
Equation~\eqref{eq:transformation_product_system} means that both
maps around the square agree on all monomials
\(\xi_p\otimes\xi_q\otimes\eta\in \Hilm_p\otimes_A \Hilm_q\otimes_A
\Hilm[F]\).  This amounts to the condition
\[
S_p(\xi_p)\cdot S_q(\xi_q) = S_{p q}(\mu_{p,q}(\xi_p\otimes\xi_q))
\qquad
\text{for all }\xi_p\in\Hilm_p, \xi_q\in\Hilm_q,
\]
which is standard for representations of product systems.

Example~\ref{exa:Cuntz_infinity} shows that we cannot expect enough
transformations to exist unless our product system is proper.  We
assume this from now on.  By
Proposition~\ref{pro:CP_vs_nondegenerate}, the nondegeneracy of the
representations~\(S_p\) is equivalent to the Cuntz--Pimsner covariance
condition for all of them.  Hence the universal property that defines
the Cuntz--Pimsner algebra gives a natural bijection between
correspondences from it to a \(\Cst\)\nb-algebra~\(B\) and
transformations from the product system to~\(B\); this bijection
leaves the underlying Hilbert module~\(\Hilm[F]\) unchanged.

A transformation~\((\Hilm[F],V_p)\) gives unital, strictly
continuous \Star{}homomorphisms
\[
\vartheta_p\colon \Bound(\Hilm_p)\to \Bound(\Hilm_p\otimes_A\Hilm[F])
\xrightarrow{\cong} \Bound(\Hilm[F]),\qquad
T\mapsto V_p(T\otimes_A \Id_{\Hilm[F]})V_p^*,
\]
for all \(p\in P\).  Similarly, the isomorphisms \(\mu_{p,q}\colon
\Hilm_p\otimes_A \Hilm_q\to\Hilm_{p q}\) induce nondegenerate
\Star{}\alb{}homomorphisms
\begin{align}
  \label{eq:Fixed_Point_algebra_m}
  \varphi_{p,q}\colon \Comp(\Hilm_p)\to \Comp(\Hilm_{p q}),\qquad
  T\mapsto \mu_{p,q}(T\otimes_A \Id_{\Hilm_q})\mu_{p,q}^*.
\end{align}
Since~\(\Hilm_q\) is proper, \(\varphi_{p,q}(\Comp(\Hilm_p))\) is
contained in~\(\Comp(\Hilm_{p q})\).  The commuting
diagram~\eqref{eq:transformation_product_system} gives
\(\vartheta_{p q}\circ\varphi_{p,q}=\vartheta_p\) for all \(p,q\in
P\).

This situation invites us to take a colimit (or inductive limit) of
the \(\Cst\)\nb-algebras \(\Comp(\Hilm_p)\) along the
maps~\(\varphi_{p,q}\).  More precisely, let~\(\Cat_P\) be the
category with object set~\(P\) and arrow set \(P\times P\),
where~\((p,q)\) is an arrow from~\(p\) to~\(p q\), and where
\((p q,t)\cdot (p,q)\defeq (p,q t)\) for all \(p,q,t\in P\).

\begin{lemma}
  \label{lem:varphi_functor}
  The maps \(p\mapsto\Comp(\Hilm_p)\) and
  \((p,q)\mapsto\varphi_{p,q}\) form a functor from~\(\Cat_P\) to
  the category of\/ \(\Cst\)\nb-algebras and nondegenerate
  \Star{}homomorphisms.
\end{lemma}

\begin{proof}
  Functoriality means that \(\varphi_{p q,t}\circ \varphi_{p,q} =
  \varphi_{p,q t}\) for all \(p,q,t\in P\).  This is equivalent to
  the associativity of~\(\mu\).
\end{proof}

The diagram in Lemma~\ref{lem:varphi_functor} has a colimit in the
category of \(\Cst\)\nb-algebras and nondegenerate
\Star{}homomorphisms.  This colimit will act nondegenerately
on~\(\Hilm[F]\) by its universal property.  Therefore, it is part of
the Cuntz--Pimsner algebra of the product system.  In general, the
colimit involves amalgamated free products, which make it rather
intractable.  To get a well-behaved Cuntz--Pimsner algebra, we
assume that~\(\Cat_P\) is a \emph{filtered} category in the
following sense:

\begin{definition}[\cite{MacLane:Categories}*{Section IX.1}]
  \label{def:filtered}
  A category~\(\Cat\) is \emph{filtered} if it is nonempty and
  \begin{enumerate}[label=\textup{(F\arabic*)}]
  \item\label{item:filtered1} for any two objects \(x,y\in\Cat_0\),
    there are an object \(z\in\Cat_0\) and arrows \(g\in\Cat(x,z)\)
    and \(h\in\Cat(y,z)\);
  \item\label{item:filtered2} for any two parallel arrows \(g,h\in
    \Cat(x,y)\), there are \(z\in \Cat_0\) and \(k\in \Cat(y,z)\) with
    \(kg=kh\).
  \end{enumerate}
\end{definition}

These conditions for~\(\Cat_P\) are equivalent to the following
\emph{Ore conditions} for~\(P\):
\begin{enumerate}[label=\textup{(O\arabic*)}]
\item \label{enum:Ore1} for all \(x_1,x_2\in P\), there are
  \(y_1,y_2\in P\) with \(x_1y_1=x_2y_2\);
\item \label{enum:Ore2} if \(xy_1=xy_2\) for \(y_1,y_2,x\in P\), then
  there is \(z\in P\) with \(y_1z=y_2z\).
\end{enumerate}

\begin{definition}
  \label{def:Ore_monoid}
  We call~\(P\) a \emph{right Ore monoid} if it has these two
  properties or, equivalently, \(\Cat_P\) is filtered.  We
  call~\(P\) a \emph{left Ore monoid} if~\(P^\op\) is a \emph{right
    Ore monoid}.
\end{definition}

Condition~\ref{enum:Ore2} follows if~\(P\)
has cancellation.  Both hold if \(P\subseteq G\)
for a group~\(G\)
with \(PP^{-1}=G\).
Cancellative Ore monoids have already been considered by
\(\Cst\)\nb-algebraists;
see, for instance,
\cites{Laca:Endomorphisms_back,Cuntz-Echterhoff-Li:K-theory}.  We do
not expect product systems over non-cancellative monoids to be very
interesting (Lemma~\ref{lem:unique_truncation} hints strongly towards
this), but it costs little extra effort to work in this greater
generality.  The monoid~\(P\)
is cancellative if and only if there is at most one arrow between any
two objects in~\(\Cat_P\).
Then~\(\Cat_P\)
is the category associated to a directed set, and colimits
over~\(\Cat_P\)
are the same as colimits over this directed set.  Directed sets are
familiar to analysts as the indexing sets for nets.

Let~\(P\) be a right Ore monoid.  We may construct a group out
of~\(P\) by taking equivalence classes of formal quotients
\(p q^{-1}\defeq (p,q)\) for \(p,q\in P\), where \((p_1,q_1)\sim
(p_2,q_2)\) if there are \(t_1,t_2\in P\) with \((p_1 t_1,q_1 t_1) =
(p_2 t_2,q_2 t_2)\) (see also~\cite{Clifford-Preston:Semigroups_I}).
Condition~\ref{enum:Ore1} implies that this relation is transitive and
that products \(p_1q_1^{-1}\cdot p_2q_2^{-1}\) may be rewritten as
\(p q^{-1}\) by finding a common multiple of \(q_1\) and~\(p_2\): if
\(q_1t_1=p_2t_2\), then
\[
p_1q_1^{-1}\cdot p_2q_2^{-1}
= (p_1t_1)(q_1t_1)^{-1}\cdot (p_2t_2)(q_2t_2)^{-1}
= (p_1t_1)(q_2t_2)^{-1}.
\]
Hence we define the multiplication by \([p_1,q_1]\cdot [p_2,q_2]
\defeq [p_1t_1,q_2t_2]\) for \(t_1,t_2\in P\) with \(q_1t_1=p_2t_2\).
The conditions \ref{enum:Ore1} and~\ref{enum:Ore2} imply that this is
a well-defined group structure on \(G\defeq P/{\sim}\).

\begin{example}
  \label{exa:Ore}
  All commutative monoids are Ore: we may take \(y_1=x_2\) and
  \(y_2=x_1\) in~\ref{enum:Ore1} and \(z=x\) in~\ref{enum:Ore2}.
  Groups are also clearly Ore monoids.  We may combine both classes
  as follows.

  Assume that \(G\subseteq P\) is a group such that \(x g_1 = x
  g_2\) for \(x\in P\), \(g_1,g_2\in G\) implies \(g_1=g_2\); assume
  further that \(pG=Gp\) for all \(p\in P\) and \(p_1 p_2 G = p_2
  p_1 G\) for all \(p_1,p_2\in P\); roughly speaking, \(G\) is a
  normal subgroup of~\(P\) such that \(P/G\) is commutative.  We
  claim that such a monoid~\(P\) is Ore.  In~\ref{enum:Ore1}, given
  \(x_1,x_2\in P\), we first take \(y_1 = x_2\) and \(y_2^0 = x_1\)
  as in the commutative case; then \(x_1 y_1 G = x_2 y_2^0 G\), so
  there is \(g\in G\) with \(x_1 y_1 = x_2 y_2^0 g\), so \(y_2=
  y_2^0 g\) will do.  In~\ref{enum:Ore2}, assume \(x y_1 = x y_2\)
  for some \(x,y_1,y_2\in P\).  Then \(y_1 x G = x y_1 G = x y_2 G =
  y_2 x G\), so there is \(g\in G\) with \(y_1 x = y_2 x g\).  Then
  \(x y_1 x = x y_2 x g = x y_1 x g\), which implies \(g=1\) by one
  of our assumptions.  Thus \(y_1 x = y_2 x\), so \(z=x\) works
  in~\ref{enum:Ore2}.
\end{example}

\begin{example}
  \label{exa:axb}
  Let~\(R\)
  be a commutative, unital ring.  Let \(R^\times\subseteq R\)
  be the subset of all elements that are not zero divisors; this is
  the largest submonoid of~\((R,\cdot)\)
  that does not contain~\(0\).
  The ``\(ax+b\)-monoid''
  of~\(R\)
  is the monoid \(P=R^\times\ltimes R\),
  consisting of pairs \((a,b)\in R^\times\ltimes R\)
  with the multiplication
  \((a_1,b_1)\cdot (a_2,b_2) = (a_1 a_2, a_1 b_2 + b_1)\).
  This monoid is cancellative because we have taken away the zero
  divisors.  It acts on~\(R\)
  by \((a,b)\cdot x\defeq ax+b\).
  It contains the additive group~\((R,+)\)
  as a normal subgroup, and the quotient~\(P/R\)
  is the commutative monoid~\((R^\times,\cdot)\).
  Hence it is a special case of Example~\ref{exa:Ore}.
  \(\Cst\)\nb-algebras
  associated to semigroups of this form have recently been studied by
  several authors, following Cuntz~\cite{Cuntz:axb_N} and
  Li~\cite{Li:Ring_Cstar}.
\end{example}

\begin{example}
  \label{exa:Gl_number_field}
  Assume that~\(P\) is cancellative and has a ``norm'' homomorphism
  \(N\colon P\to (\N^\times,\cdot)\) such that \(G\defeq \ker N\) is
  a subgroup.  This is a special case of the situation in
  Example~\ref{exa:Ore}, with \(P/G\subseteq (\N^\times,\cdot)\).
  One example of this type is the monoid \(\Mat_n(\Z)^\times\), the
  ring of integer matrices with non-zero determinant, with the
  determinant as norm: the kernel of the norm is the
  group~\(\mathrm{Gl}_n(\Z)\).  Another example is the monoid
  \(\Mat_n(\Z)^\times\ltimes \Z^n\), with the determinant of the
  matrix part as norm and the semidirect product group
  \(\mathrm{Gl}_n(\Z)\ltimes\Z^n\) as the kernel of the norm.  More
  generally, we may replace~\(\Mat_n(\Z)\) by the ring of integers
  in a simple algebra over~\(\mathbb{Q}\); this also includes
  integer quaternion algebras.  Semigroups of this form appear in
  generalisations of the Bost--Connes dynamical system, see
  \cites{Bost-Connes:Hecke, Connes-Marcolli-Ramachadran:KMS}.
\end{example}

\begin{example}
  \label{exa:Ore_monoid_H}
  The matrices of the form
  \[
  h(a,b,c) \defeq
  \begin{pmatrix}
    1&a&c\\0&1&b\\0&0&1
  \end{pmatrix}
  \]
  for \(a,b,c\in \N\) form a noncommutative, cancellative
  monoid~\(H_\N\) under matrix multiplication:
  \[
  h(a_1,b_1,c_1)\cdot h(a_2,b_2,c_2) =
  h(a_1+a_2,b_1+b_2,c_1+c_2+a_1b_2).
  \]
  This is an Ore monoid.  To check the Ore
  condition~\ref{enum:Ore1}, pick \(h(a_1,b_1,c_1)\) and
  \(h(a_2,b_2,c_2)\) in~\(H_\N\).  Let
  \[
  a\defeq \max(a_1,a_2),\quad
  b\defeq \max(b_1,b_2),\quad
  c\defeq \max(c_1+a_1(b-b_1),c_2+a_2(b-b_2)).
  \]
  For \(i=1,2\), let \(a_i^\bot \defeq a-a_i\), \(b_i^\bot \defeq
  b-b_i\), and \(c_i^\bot \defeq c-c_i- a_i b_i^\bot\); then
  \(h(a_i,b_i,c_i) \cdot h(a_i^\bot,b_i^\bot,c_i^\bot) = h(a,b,c)\)
  for \(i=1,2\), so we have found the desired common multiple.  A
  similar formula works for the opposite monoid, so~\(H_\N\) is both
  left and right Ore.
\end{example}

An inductive limit in the usual sense is the same as a colimit over
the category associated to the poset~\((\N,\le)\), which is easily
seen to be filtered.  Colimits over general filtered categories
behave very much like inductive limits.  This is well-known to
category theorists.  For the operator algebraists, we now assume
that~\(P\) is a countable Ore monoid, so that~\(\Cat_P\) is a
countable filtered category.  Then we may replace a colimit
over~\(\Cat_P\) by an inductive limit over~\((\N,\le)\):

\begin{lemma}
  \label{lem:countable_filtered_cofinal}
  Let~\(\Cat\) be a countable filtered category.  Then there is a
  sequence of objects \((x_n)_{n\in\N}\) and maps \(f_n\in
  \Cat(x_{n-1},x_n)\) such that for any object~\(y\) of~\(\Cat\) there
  is \(n\in\N\) and an arrow \(y\to x_n\).  Furthermore, if \(y\to
  x_n\) and \(y\to x_m\) are two such arrows, they become equal by
  composing with \(f_{N-1}\circ \dotsb\circ f_n\colon x_n\to
  x_{n+1}\to \dotsb \to x_N\) and \(f_{N-1}\circ \dotsb\circ f_m\colon
  x_m\to x_{m+1} \to \dotsb \to x_N\) for sufficiently large~\(N\).
\end{lemma}

Such a sequence of objects and maps is called \emph{cofinal} or
final.  More precisely, the functor \((\N,\le)\to\Cat\) given by the
objects~\(x_n\) and the maps~\(f_n\) is called final
in~\cite{MacLane:Categories}.

\begin{proof}
  It is shown in~\cite{Andreka-Nemethi:Limits} that any filtered
  category receives a cofinal functor from a directed (partially
  ordered) set.  A partially ordered set is viewed as a category by
  putting a unique arrow \(x\to y\) if \(x\le y\), and no arrow
  otherwise.  A category is of this form if and only if for any two
  objects there is at most one arrow between them.  To simplify the
  proof, we first use~\cite{Andreka-Nemethi:Limits} to reduce to a
  countable, directed set.  The category~\(\Cat_P\) comes from a
  directed set if and only if~\(P\) has cancellation.

  Let~\((y_n)_{n\in\N}\) be an enumeration of the objects of~\(\Cat\).
  We construct~\(x_n\) for \(n\in\N\) inductively so that it receives
  maps from \(y_1,\dotsc,y_n\).  We start with \(x_0=y_0\).  Assume
  \(x_i\) and~\(f_i\) have been constructed for \(i<n\).
  Since~\(\Cat\) is filtered, there is an object~\(x_n\) that receives
  maps from \(y_n\) and~\(x_{n-1}\).  Let~\(f_n\) be the arrow
  \(x_{n-1}\to x_n\).  Since already~\(x_{n-1}\) receives maps
  from~\(y_i\) for \(i<n\), so does~\(x_n\) by composing with~\(f_n\).
  Thus every object~\(y\) has a map to some~\(x_n\).  Our simplifying
  assumption makes the second part of the lemma trivial.
\end{proof}

We now describe the colimit of the inductive system on~\(\Cat_P\)
given by the \(\Cst\)\nb-algebras \(\Comp(\Hilm_p)\) for \(p\in P\)
and the maps \(\varphi_{p,q}\) for \(p,q\in P\) defined
in~\eqref{eq:Fixed_Point_algebra_m}.

We first do this quickly in the countable case.  Then
Lemma~\ref{lem:countable_filtered_cofinal} allows us to choose a
cofinal functor~\((\N,\le)\) to~\(\Cat_P\), that is, we get a pair
of sequences
\((p_n)_{n\in\N}\) and \((q_n)_{n\in\N}\) in~\(P\) with \(p_{n+1} =
p_n q_n\) for all \(n\in\N\) that is ``cofinal'' in~\(\Cat_P\).  The
\(\Cst\)\nb-algebras \(\Comp(\Hilm_{p_n})\) and the nondegenerate
\Star{}homomorphisms \(\varphi_{p_n,q_n}\colon \Comp(\Hilm_{p_n}) \to
\Comp(\Hilm_{p_n q_n}) = \Comp(\Hilm_{p_{n+1}})\) form an inductive
system in the usual sense.  Let~\(\CP_1\) be its inductive limit
\(\Cst\)\nb-algebra.  Cofinality implies that this inductive limit is
also a colimit of the whole diagram on~\(\Cat_P\).

Now we give the more complicated construction without using
Lemma~\ref{lem:countable_filtered_cofinal}, which also works in the
uncountable case.  Let
\[
\CP_\sqcup\defeq \bigsqcup_{p\in P} \Comp(\Hilm_p).
\]
Let~\(\CP_\sim\) be the set of equivalence classes for the equivalence
relation on~\(\CP_\sqcup\) generated by the relations \((x,p)\sim
(\varphi_{p,q}(x),p q)\) for all \(p,q\in P\), \(x\in\Comp(\Hilm_p)\).

\begin{lemma}
  \label{lem:Ore_limit_star}
  There is a unique \Star{}algebra structure for which all maps
  \(\Comp(\Hilm_p)\to \CP_\sim\) are \Star{}homomorphisms, and the
  maximal \(\Cst\)\nb-seminorm on~\(\CP_\sim\) exists.
  Let~\(\CP_1\) be the resulting \(\Cst\)\nb-completion
  of~\(\CP_\sim\).  If \((\Hilm[F],V_p)\) is a transformation from
  \((A,\Hilm_p,\mu_{p,q})\) to a \(\Cst\)\nb-algebra~\(B\), then the
  resulting maps \(\vartheta_p\colon \Comp(\Hilm_p) \to
  \Bound(\Hilm[F])\) factor through a unique nondegenerate
  \Star{}\alb{}homomorphism \(\Theta\colon \CP_1\to
  \Bound(\Hilm[F])\).
\end{lemma}

\begin{proof}
  Let \(x\in\Comp(\Hilm_p)\), \(y\in\Comp(\Hilm_q)\).  There are
  \(t_1,t_2\in P\) with \(p t_1=q t_2\).  Then \((x,p)\sim
  (\varphi_{p,t_1}(x),p t_1)\) and \((y,q)\sim (\varphi_{q,t_2}(y),q
  t_2)\) both belong to the \(\Cst\)\nb-algebra \(\Comp(\Hilm_{p
    t_1}) = \Comp(\Hilm_{q t_2})\); this dictates what their sum or
  product should be in~\(\CP_\sim\).  If we choose \(t'_1,t'_2\in
  P\) with \(p t'_1 = q t'_2\) instead, then we may find
  \(m_1,m_2\in P\) with \(p t'_1 m_1 = p t_1 m_2\) and hence \(q
  t'_2 m_1 = q t_2 m_2\).  If not yet \(t'_2 m_1 = t_2 m_2\), then
  we find \(n\in P\) with \(t'_2 m_1 n = t_2 m_2 n\) and replace
  \(m_1,m_2\) by \(m_1 n, m_2 n\).  Similarly, we achieve \(t'_1 m_1
  = t_1 m_2\).  Then multiplication with \(m_2\) and~\(m_1\) will
  map our two choices of the sum or product to the same sum or
  product in \(\Comp(\Hilm_{p t_1 m_2})\), respectively.  Thus the
  multiplication and addition on~\(\CP_\sim\) are well-defined.

  A similar argument shows that any finite subset of~\(\CP_\sim\)
  belongs to the image of~\(\Comp(\Hilm_p)\) in~\(\CP_\sim\) for
  some \(p\in P\).  Since the algebraic operations are defined using
  those in~\(\Comp(\Hilm_p)\), \(\CP_\sim\) is a \Star{}algebra.  By
  construction, this is the only \Star{}algebra structure for which
  all maps \(\Comp(\Hilm_p)\to \CP_\sim\) are \Star{}homomorphisms.

  The kernel of the map \(\Comp(\Hilm_p)\to \CP_\sim\) is the union
  of the kernels of the \Star{}\alb{}homomorphisms
  \(\varphi_{q,p}\colon \Comp(\Hilm_p)\to \Comp(\Hilm_{p q})\).
  Thus the image of~\(\Comp(\Hilm_p)\) in~\(\CP_\sim\) is the
  quotient by a union of closed \Star{}ideals.  We equip it with the
  quotient seminorm, which is a \(\Cst\)\nb-seminorm (there may be a
  nullspace because the union of ideals need not be closed).  All
  these \(\Cst\)\nb-seminorms on subalgebras of~\(\CP_\sim\)
  together are compatible with each other and thus define a
  \(\Cst\)\nb-seminorm on~\(\CP_\sim\).  Since any
  \Star{}homomorphism between \(\Cst\)\nb-algebras is contractive,
  this is the maximal \(\Cst\)\nb-seminorm on the
  \Star{}algebra~\(\CP_\sim\).  Let~\(\CP_1\) be the (Hausdorff)
  completion of~\(\CP_\sim\) for this \(\Cst\)\nb-seminorm.  This is
  a \(\Cst\)\nb-algebra with \Star{}homomorphisms
  \(\vartheta^0_p\colon \Comp(\Hilm_p)\to \CP_1\) for all \(p\in P\)
  that satisfy \(\vartheta^0_{p q}\circ\varphi_{p,q} =
  \vartheta^0_p\) for all \(p,q\in P\).

  Now take a transformation to~\(B\) as above.  The resulting maps
  \(\vartheta_p\colon \Comp(\Hilm_p)\to \Bound(\Hilm[F])\) satisfy
  \(\vartheta_{p q}\circ\varphi_{p,q} = \vartheta_p\).  Hence the
  map \(\bigsqcup_{p\in G} \vartheta_p\colon \CP_\sqcup \to
  \Bound(\Hilm[F])\) descends to a map \(f\colon \CP_\sim\to
  \Bound(\Hilm[F])\).  Since all~\(\vartheta_p\) are
  \Star{}homomorphisms, so is~\(f\).  Since we took the maximal
  \(\Cst\)\nb-seminorm on~\(\CP_\sim\), we may extend~\(f\) uniquely
  to a \Star{}homomorphism \(\Theta\colon \CP_1\to
  \Bound(\Hilm[F])\) with \(\Theta\circ \vartheta^0_p =
  \vartheta_p\) for all \(p\in P\).
\end{proof}

Any functor \((p_n,q_n)\colon (\N,\le)\to\Cat_P\) induces a
\Star{}homomorphism from the inductive limit \(\Cst\)\nb-algebra of
the inductive system \((\Comp(\Hilm_{p_n}),\varphi_{p_n,q_n})\)
described above Lemma~\ref{lem:Ore_limit_star} to~\(\CP_1\).  If the
functor is cofinal, then this map is an isomorphism.  Hence the
simplified construction for countable~\(P\) gives the same
\(\Cst\)\nb-algebra~\(\CP_1\).

So far, we have described only a part of the Cuntz--Pimsner algebra of
the product system.  For a single endomorphism, this is the
fixed-point subalgebra of the gauge action.  We now describe
the whole Cuntz--Pimsner algebra through a Fell bundle over the group
completion~\(G\) of~\(P\).

Elements of~\(G\) are equivalence classes of formal
fractions~\(p_1p_2^{-1}\) for \(p_1,p_2\in P\), with \(p_1p_2^{-1}
\sim (p_1 q)(p_2 q)^{-1}\).  The fibre of the desired Fell bundle
over~\(G\) at \(1\in G\) is the \(\Cst\)\nb-algebra~\(\CP_1\)
described above.

\begin{definition}
  \label{def:Cat_P_g}
  Fix \(g\in G\).  Let
  \[
  R_g = \{(p_1,p_2)\in P\times P\mid p_1 p_2^{-1} = g \text{ in }G\}
  \]
  be its set of representatives.  Let~\(\Cat_P^g\) be the category
  with object set~\(R_g\) and arrow set \(R_g\times P\), where
  \((p_1,p_2,q)\) is an arrow \((p_1,p_2)\to (p_1 q,p_2 q)\); the
  multiplication is \((p_1 q, p_2 q,t)\cdot (p_1, p_2,q) =
  (p_1,p_2,t q)\).
\end{definition}

\begin{lemma}
  \label{lem:Ore_category_g}
  The categories~\(\Cat_P^g\) for \(g\in G\) are filtered if~\(P\)
  is an Ore monoid.  The functor \(\Cat_P\to\Cat_P^1\) that maps an
  object~\(p\) in~\(\Cat_P\) to~\((p,p)\) and an arrow~\((p,q)\)
  in~\(\Cat_P\) to \((p,p,q)\) is cofinal.
\end{lemma}

\begin{proof}
  First we check that~\(\Cat_P^g\) is filtered.  Let \(p=(p_1,p_2)\)
  and \(q=(q_1,q_2)\) be elements of~\(R_g\).  We must prove two
  things.  First, there should be arrows \(h\colon p\to t\) and
  \(k\colon q\to t\) with the same target \(t\in R_g\).  Secondly,
  if \(h,k\colon p\rightrightarrows q\) are two parallel arrows,
  there is an arrow \(l\colon q\to t\) for some object~\(t\) such
  that \(l\circ h=l\circ k\).  Since \(p\) and~\(q\) both represent
  \(g\in G\), there are \(h,k\in P\) with \(p_1 h = q_1 k\) and
  \(p_2 h = q_2 k\).  Hence \(h\colon (p_1,p_2)\to (p_1 h,p_2 h)\)
  and \(k\colon (q_1,q_2)\to (q_1 k,q_2 k)\) have the same target,
  as desired.  The second claim above is immediate
  from~\ref{enum:Ore2}: we may simply forget \(p_2\) and~\(q_2\).

  The functor \(\Cat_P\to\Cat_P^1\) is fully faithful.  If
  \((p_1,p_2)\in R_1\), then there are \(q,h\in P\) with \((p_1
  h,p_2 h) = (q,q)\).  Hence the functor \(\Cat_P\to\Cat_P^1\) is
  cofinal.
\end{proof}

For \((p_1,p_2)\in R_g\), let \(\CP_{p_1,p_2}\defeq
\Comp(\Hilm_{p_2},\Hilm_{p_1})\); for now, we view this as a Banach
space.  For \(q\in P\), \((p_1,p_2)\in R_g\), we define a contraction
\[
\varphi_{p_1,p_2,q}\colon \Comp(\Hilm_{p_2},\Hilm_{p_1})\to
\Comp(\Hilm_{p_2 q},\Hilm_{p_1 q}),\qquad
T\mapsto \mu_{p_1,q}(T\otimes_A \Id_{\Hilm_q})\mu_{p_2,q}^*.
\]
These maps form a functor from~\(\Cat_P^g\) to the category of
Banach spaces with linear contractions.  Since~\(\Cat_P^g\) is
filtered by Lemma~\ref{lem:Ore_category_g}, the colimit~\(\CP_g\) of
this diagram may be constructed as in
Lemma~\ref{lem:Ore_limit_star}: first take the disjoint union of the
Banach spaces~\(\CP_{p_1,p_2}\) for all \((p_1,p_2)\in R_g\); then
divide out the relations given by the maps~\(\varphi_{p_1,p_2,q}\);
this gives a vector space, and it inherits a canonical seminorm by
taking the quotient seminorms on the images of
\(\Comp(\Hilm_{p_2},\Hilm_{p_1})\); finally, take the completion to
get~\(\CP_g\).  If~\(P\) is countable, then we may also use a
cofinal sequence in~\(\Cat_P^g\) to describe the colimit as an
inductive limit over~\((\N,\le)\).

Since the functor \(\Cat_P\to\Cat_P^1\) is cofinal, the colimit of a
diagram over~\(\Cat_P^1\) is the same as the colimit of its
restriction to~\(\Cat_P\).  Hence the construction of~\(\CP_g\) for
\(g=1\) gives the same \(\Cst\)\nb-algebra~\(\CP_1\) as in
Lemma~\ref{lem:Ore_limit_star}, as suggested by our notation.

If \(g_1,g_2\in G\), \((p_1,p_2)\in R_{g_1}\), \((p_2,p_3)\in
R_{g_2}\), then \(p_1 p_3^{-1} = p_1 p_2^{-1}\cdot p_2 p_3^{-1} = g_1\cdot
g_2\), that is, \((p_1,p_3)\in R_{g_1\cdot g_2}\).  The composition of
compact operators gives a bounded bilinear map \(\CP_{p_1,p_2}\times
\CP_{p_2,p_3} \to \CP_{p_1,p_3}\).  These maps define a bounded
bilinear map
\[
\CP_{g_1}\times \CP_{g_2} \to \CP_{g_1g_2}
\]
because for any \((p_1',p_2')\), \((p_2,p_3)\) in \(R_{g_1}\times
R_{g_2}\) there are \(h,k\in P\) with \(p_2 h = p_2' k\), so that the
composition is defined on \(\CP_{p_1' k,p_2' k}\times \CP_{p_2 h, p_3
  h}\), and these composition maps are compatible with the structure
maps of the inductive systems.  Similarly, taking adjoints gives maps
\(\CP_{p_1,p_2}\to \CP_{p_2,p_1}\), \(T\mapsto T^*\), for all
\(p_1,p_2\in R_g\); these maps induce an involution \(\CP_g\to
\CP_{g^{-1}}^*\).  These multiplication maps and involutions
on~\((\CP_g)_{g\in G}\) give a Fell bundle over the group~\(G\).  The
resulting \(\Cst\)\nb-algebra structure on its unit fibre~\(\CP_1\) is
the one already described in Lemma~\ref{lem:Ore_limit_star}.

\begin{theorem}
  \label{the:Ore_colimit_Fell}
  Let~\(P\) be an Ore monoid and let \((A,\Hilm_p,\mu_{p,q})\) be a
  proper, nondegenerate product system over~\(P\).  Its Cuntz--Pimsner
  algebra is isomorphic to the full sectional \(\Cst\)\nb-algebra of
  the Fell bundle~\((\CP_g)_{g\in G}\) described above.
\end{theorem}

\begin{proof}
  Let~\(C\) denote the Cuntz--Pimsner algebra of our product system.
  By construction, a nondegenerate \Star{}homomorphism \(C\to
  \Mult(B)\) for a \(\Cst\)\nb-algebra~\(B\) is the same as a
  Cuntz--Pimsner covariant representation of our product system
  on~\(B\) that is nondegenerate on the unit fibre~\(A\).  The
  Cuntz--Pimsner covariance condition is equivalent to the
  nondegeneracy condition \(\Hilm_p\cdot B=B\) for all \(p\in P\) by
  Proposition~\ref{pro:CP_vs_nondegenerate} because we assume
  all~\(\Hilm_p\) to be proper and nondegenerate left \(A\)\nb-modules,
  and the left \(A\)\nb-action on~\(B\) is nondegenerate as well.

  We are going to find a natural bijection between representations of
  the product system with \(\Hilm_p\cdot B=B\) for all \(p\in P\) and
  representations of the Fell bundle~\((\CP_g)_{g\in G}\)
  in~\(\Mult(B)\).  By the universal property of the sectional
  \(\Cst\)\nb-algebra of a Fell bundle, this gives a natural bijection
  between nondegenerate \Star{}homomorphisms \(C\to \Mult(B)\) and
  \(\Cst((\CP_g)_{g\in G})\to\Mult(B)\), and this implies \(C\cong
  \Cst((\CP_g)_{g\in G})\).

  By Proposition~\ref{pro:unitary_to_correspondence_representation}, a
  representation of the product system that is nondegenerate in the
  above sense is equivalent to a transformation from
  \((A,\Hilm_p,\mu_{p,q})\) to~\(B\) with underlying Hilbert
  \(B\)\nb-module~\(B\).  We write \(\Hilm[F]=B\) to be consistent
  with our previous notation.  We already constructed
  \Star{}homomorphisms \(\vartheta_p\colon
  \Comp(\Hilm_p)\to\Bound(\Hilm[F])\) with
  \(\vartheta_{p q}\circ\varphi_{p,q}=\vartheta_p\) for all \(p,q\in
  P\).  The same recipe gives linear contractions
  \[
  \vartheta_{p_1,p_2} \colon
  \Comp(\Hilm_{p_2},\Hilm_{p_1})\to\Bound(\Hilm[F]),\qquad
  T\mapsto V_{p_1}(T\otimes \Id_{\Hilm[F]})V_{p_2}^*.
  \]
  These satisfy \(\vartheta_{p_1 q,p_2 q} \circ \varphi_{p_1,p_2,q} =
  \vartheta_{p_1,p_2}\) for all \(p_1,p_2,q\in P\).  Hence they induce
  maps \(\Theta_g\colon \CP_g\to \Bound(\Hilm[F])\) on the Banach
  space inductive limits.  Routine computations show that
  \begin{equation}
    \label{eq:Fell_bundle_rep}
    \vartheta_{p_2,p_1}(T)^* = \vartheta_{p_1,p_2}(T^*),\qquad
    \vartheta_{p_1,p_2}(T) \circ \vartheta_{p_2,p_3}(T_2)
    = \vartheta_{p_1,p_3}(T\circ T_2)
  \end{equation}
  for all \(p_1,p_2,p_3\in P\),
  \(T\in\Comp(\Hilm_{p_2},\Hilm_{p_1})\),
  \(T_2\in\Comp(\Hilm_{p_3},\Hilm_{p_2})\).  Hence the
  maps~\(\Theta_g\) form a representation of the Fell
  bundle~\((\CP_g)_{g\in G}\).

  Conversely, a representation of the Fell bundle~\((\CP_g)_{g\in G}\)
  gives maps
  \[
  \Comp(\Hilm_{p_2},\Hilm_{p_1}) \to \Bound(\Hilm[F])
  \]
  that satisfy~\eqref{eq:Fell_bundle_rep}.  For \(p_2=1\), there is a
  canonical isomorphism \(\Comp(\Hilm_{p_2},\Hilm_{p_1})\cong
  \Hilm_{p_1}\) because \(\Hilm_1=A\).  Hence the Fell bundle
  representation gives maps \(S_p\colon \Hilm_p\to \Bound(\Hilm[F])\).
  Since \(A=\Comp(\Hilm_1)\subseteq \CP_1\), the conditions of a Fell
  bundle representation imply that the maps~\(S_p\) form a
  representation of the product system.  Since the maps
  \(\Comp(\Hilm_p)\to \CP_1\to \Bound(\Hilm[F])\) are nondegenerate,
  \(S_p(\Hilm_p)\Hilm[F]\supseteq \Comp(\Hilm_p)\Hilm[F] =
  \Hilm[F]\).  This gives the desired bijection between Fell bundle
  representations and Cuntz--Pimsner covariant representations of the
  product system and finishes the proof.
\end{proof}

Theorem~\ref{the:Ore_colimit_Fell} is similar to
\cite{Kwasniewski-Szymanski:Ore}*{Theorem 3.8}, only the assumptions
on the product system differ slightly.  Unlike
in~\cite{Kwasniewski-Szymanski:Ore}, we assume the product system to
be nondegenerate, but allow the left action to be non-injective.
Similar ideas are used
in~\cite{Cuntz-Echterhoff-Li:K-theory}*{Section~4} to dilate certain
actions of Ore semigroups by endomorphisms to group actions on a
larger algebra.  In fact, actions by endomorphisms are a special
case of product systems by Remark~\ref{rem:left_right}.  To reduce
the dilation results in \cite{Cuntz-Echterhoff-Li:K-theory}
or~\cite{Laca:Endomorphisms_back} to
Theorem~\ref{the:Ore_colimit_Fell}, another step is missing:
checking when the Fell bundle in Theorem~\ref{the:Ore_colimit_Fell}
comes from an ordinary group action by automorphisms.  We refrain
from doing this because, from our point of view, saturated Fell
bundles are already actions of the underlying group.

\begin{proposition}
  \label{pro:Fell_bundle_saturated}
  The Fell bundle~\((\CP_g)_{g\in G}\) is saturated if~\(\Hilm_p\) is
  a full Hilbert \(A\)\nb-module for each \(p\in P\).
\end{proposition}

\begin{proof}
  Let \(g\in G\) and let \(p\in P\).  We want to show that the image
  of~\(\Comp(\Hilm_p)\) in~\(\CP_1\) is contained in the space of
  right inner products from~\(\CP_g\).  There is \((p_1,p_2)\in R_g\)
  and \(q\in P\) with \(p q = p_1\).  The image of~\(\Comp(\Hilm_p)\)
  in~\(\CP_1\) is contained in the image of \(\Comp(\Hilm_{p q}) =
  \Comp(\Hilm_{p_1})\).

  Since \(\Hilm_{p_1}\) and~\(\Hilm_{p_2}\) are full, both
  \(\Comp(\Hilm_{p_1})\) and \(\Comp(\Hilm_{p_2})\) are
  Morita--Rieffel equivalent to~\(A\) and hence equivalent to each
  other.  The equivalence between them is
  \(\Comp(A,\Hilm_{p_1})\otimes_A\Comp(\Hilm_{p_2},A) \cong
  \Comp(\Hilm_{p_2},\Hilm_{p_1})\).  Hence the latter is a full
  Hilbert bimodule over \(\Comp(\Hilm_{p_2})\)
  and~\(\Comp(\Hilm_{p_1})\).  Since
  \(\Comp(\Hilm_{p_2},\Hilm_{p_1})\subseteq \CP_g\), it follows that
  the right inner products from~\(\CP_g\) give a dense subspace
  of~\(\Comp(\Hilm_{p_1})\) in~\(\CP_1\).  Thus the Fell
  bundle~\((\CP_g)_{g\in G}\) is saturated.
\end{proof}

\begin{remark}
  \label{rem:full_not_necessary}
  The criterion in Proposition~\ref{pro:Fell_bundle_saturated} is not
  necessary for rather trivial reasons.  If the left \(A\)\nb-actions
  on~\(\Hilm_p\) are not faithful, then it may happen that
  \(\CP_1=0\).  Since this has nothing to do with~\(\Hilm_p\) being
  full as a right Hilbert module, the Fell bundle~\((\CP_g)\) may be
  saturated although not all~\(\Hilm_p\) are full.
\end{remark}

Saturated Fell bundles over a group~\(G\) are interpreted as actions
of~\(G\) by correspondences in~\cite{Buss-Meyer-Zhu:Higher_twisted}.
Long before, it was known that one may replace a saturated Fell
bundle~\((\CP_g)_{g\in G}\) with unit fibre~\(\CP_1\) by an action
of~\(G\) by automorphisms on a \(\Cst\)\nb-algebra~\(\tilde{\CP}_1\)
that is Morita--Rieffel equivalent to~\(\CP_1\): this is the
Packer--Raeburn Stabilisation Trick.  Non-saturated Fell bundles
over~\(G\) are interpreted in~\cite{Buss-Meyer:Actions_groupoids} as
actions of~\(G\) by Hilbert bimodules, that is, \emph{partial
  Morita--Rieffel equivalences}.  The analogue of the Packer--Raeburn
Stabilisation Trick says that any Fell bundle, saturated or not, is
equivalent to an action of~\(G\) by partial \Star{}isomorphisms.

A saturated Fell bundle over~\(G\) may, of course, be restricted to a
product system over~\(P\).  Which product systems are of this form?

\begin{proposition}
  \label{pro:Fell_bundle_vs_product_system}
  A proper product system over~\(P\) is the restriction of a saturated
  Fell bundle over~\(G\) if and only if each~\(\Hilm_p\) is an
  \(A,A\)\nb-imprimitivity bimodule, that is, each~\(\Hilm_p\) is a
  full right Hilbert \(A\)\nb-module and the left action is by an
  \emph{isomorphism} \(A\cong\Comp(\Hilm_p)\).  The saturated Fell
  bundle over~\(G\) is unique up to isomorphism.
\end{proposition}

\begin{proof}
  In a saturated Fell bundle over~\(G\), each~\(\Hilm_g\) is an
  imprimitivity bimodule.  Conversely, assume that~\(\Hilm_p\) is an
  imprimitivity bimodule for each \(p\in P\).  Then all the maps
  \(\Comp(\Hilm_p)\to \Comp(\Hilm_{p q})\) in our inductive system are
  isomorphisms, so that the inductive limit~\(\CP_1\) is
  isomorphic to \(A= \Comp(\Hilm_1)\).  Similarly, \(\CP_p \cong
  \Hilm_p\) for all \(p\in P\).  Thus our product system is the
  restriction to~\(P\) of a Fell bundle over~\(G\).  Since
  all~\(\Hilm_p\) are assumed to be full, this Fell bundle is
  saturated by Proposition~\ref{pro:Fell_bundle_saturated}.

  Now start with a saturated Fell bundle~\((\CP_g)_{g\in G}\),
  restrict it to~\(P\), and then go back to a Fell bundle over~\(G\).
  The maps \(\Comp(\Hilm_{p_2},\Hilm_{p_1}) \to \Comp(\Hilm_{p_2 q},
  \Hilm_{p_1 q})\) are isomorphisms for all \(p_1,p_2,q\in P\), so the
  inductive systems that give the fibres of the new Fell bundle are
  constant.  Thus the colimit~\(\CP_g\) is canonically isomorphic
  to \(\Comp(\Hilm_{p_2},\Hilm_{p_1})\) for any \((p_1,p_2)\in R_g\),
  and our construction of a Fell bundle from~\((\Hilm_p)_{p\in P}\)
  reproduces the original Fell bundle up to isomorphism.  Hence the
  product system on~\(P\) determines the saturated Fell bundle
  over~\(G\) uniquely up to isomorphism.
\end{proof}

A non-saturated Fell bundle over~\(G\) need not give a \emph{proper}
product system on~\(P\): this requires~\(\Hilm_p\) to be full as a
left Hilbert \(A\)\nb-module for each \(p\in P\).

\begin{theorem}
  \label{the:CP_nuclear}
  If~\(A\) is nuclear or exact, then so is~\(\CP_1\).  If~\(A\) is
  nuclear and the group~\(G\) generated by~\(P\) is amenable, then
  the Cuntz--Pimsner algebra~\(\CP\) is nuclear.  If~\(A\) is exact
  and~\(G\) is amenable, then~\(\CP\) is exact.
\end{theorem}

\begin{proof}
  The first claim follows because~\(\CP_1\) is an inductive limit of
  \(\Cst\)\nb-algebras Morita--Rieffel equivalent to~\(A\) and
  because nuclearity and exactness are hereditary under
  Morita--Rieffel equivalence and filtered inductive limits.

  The other statements follow from
  Theorem~\ref{the:Ore_colimit_Fell} and general results about
  nuclearity and exactness of Fell bundle \(\Cst\)\nb-algebras.
  First, if the group is amenable, then any Fell bundle over it has
  the approximation property, which implies that the full and
  reduced sectional \(\Cst\)\nb-algebras coincide
  (see~\cite{Exel:Amenability}).  The exactness of the reduced
  sectional \(\Cst\)\nb-algebra is proved
  in~\cite{Exel:Exact_groups_Fell}, assuming exact unit fibre and an
  exact group.  The nuclearity of the full sectional
  \(\Cst\)\nb-algebra is proved in~\cite{Abadie:Tensor}, assuming
  nuclear unit fibre and an amenable group.
\end{proof}

Next we describe the \(\K\)\nb-theory of the unit fibre~\(\CP_1\) of
our Fell bundle.  Since~\(\Hilm_p\) is a proper correspondence
from~\(A\) to~\(A\), it gives an element \([\Hilm_p]\in\KK_0(A,A)\)
with zero operator~\(F\).  This gives a map
\[
(\Hilm_p)_*\colon \K_*(A)\to\K_*(A);
\]
here~\(\K_*(A)\) denotes the \(\Z/2\)\nb-graded \(\K\)\nb-theory of~\(A\)
comprising both \(\K_0\) and~\(\K_1\).  The Kasparov product of
\([\Hilm_p]\) and~\([\Hilm_q]\) is \([\Hilm_p\otimes_A \Hilm_q]\) with
zero operator; since the Fredholm operator is irrelevant, this case
of the Kasparov product is easy.  The
isomorphisms~\(\mu_{p,q}\) now show that \([\Hilm_p]\otimes_A
[\Hilm_q] = [\Hilm_{p q}]\) and hence \((\Hilm_q)_*\circ (\Hilm_p)_* =
(\Hilm_{p q})_*\) for all \(p,q\in P\).  The order of \(p\) and~\(q\)
is changed here because~\(\otimes_A\) is the composition product
in~\(\KK\) in reverse order.  Hence our product system over~\(P\)
gives an action of~\(P^\op\) on~\(\K_*(A)\).  We view this as a right
module structure over the monoid ring~\(\Z[P]\).  The group
ring~\(\Z[G]\) is a left module over~\(\Z[P]\).

\begin{theorem}
  \label{the:K_fibre}
  Let~\(P\) be an Ore monoid and let \((A,\Hilm_p,\mu_{p,q})\) give
  a proper, nondegenerate product system over~\(P\).  Assume also
  that all~\(\Hilm_p\) are full right Hilbert \(A\)\nb-modules.
  Then the \(\K\)\nb-theory of~\(\CP_1\) is isomorphic to \(\K_*(A)
  \otimes_{\Z[P]} \Z[G]\) as a right \(\Z[G]\)-module; here we use
  the canonical right module structure on \(\K_*(A) \otimes_{\Z[P]}
  \Z[G]\) by right multiplication and the module structure
  on~\(\K_*(\CP_1)\) induced by the saturated Fell
  bundle~\((\CP_g)_{g\in G}\).
\end{theorem}

\begin{proof}
  It is well-known that \(\K\)\nb-theory is compatible with inductive
  limits.  This extends to colimits over countable filtered categories
  by Lemma~\ref{lem:countable_filtered_cofinal}.  We leave it to the
  reader interested in uncountable monoids to check that the result
  remains true for arbitrary filtered colimits.  Hence \(\K_*(\CP_1)\)
  is the colimit of the diagram over~\(\Cat_P\) that maps \(p\in P\)
  to \(\K_*(\Comp(\Hilm_p))\) and \((p,q)\colon p\to p q\) to
  \((\varphi_{p,q})_*\colon \K_*(\Comp(\Hilm_p)) \to
  \K_*(\Comp(\Hilm_{p q}))\).

  Since~\(\Hilm_p\) is a full Hilbert bimodule, it gives a
  Morita--Rieffel equivalence from~\(\Comp(\Hilm_p)\) to~\(A\).  This
  correspondence with zero operator~\(F\) is a cycle for
  \(\KK_0(\Comp(\Hilm_p),A)\).  This is a \(\KK\)\nb-equivalence: the
  inverse is the inverse imprimitivity bimodule~\(\Comp(\Hilm_p,A)\)
  with zero operator~\(F\).  We use this \(\KK\)\nb-equivalence to
  identify \(\K_*(\Comp(\Hilm_p))\cong \K_*(A)\) for all \(p\in P\).
  Composing the maps
  \[
  \K_*(A) \xrightarrow[\cong]{}
  \K_*(\Comp(\Hilm_p)) \xrightarrow{(\varphi_{p,q})_*}
  \K_*(\Comp(\Hilm_{p q})) \xrightarrow[\cong]{}
  \K_*(A)
  \]
  requires composing three \(\KK_0\)-cycles with zero
  operator~\(F\), which amounts to tensoring the underlying
  correspondences.  Identifying \(\Hilm_{p q} \cong \Hilm_p\otimes_A
  \Hilm_q\) as in the definition of~\(\varphi_{p,q}\), we see that
  this composite is~\([\Hilm_q]\).  Thus the inductive system with
  colimit~\(\K_*(\CP_1)\) is isomorphic to the inductive system with
  entries \(\K_*(A)\) at all \(p\in P\), where the arrow \((p,q)\colon
  p\to p q\) in~\(\Cat_P\) induces the map \((\Hilm_q)_*\colon
  \K_*(A)\to \K_*(A)\).

  Define a diagram of left \(\Z[P]\)-modules over~\(\Cat_P\) by taking
  the free module~\(\Z[P]\) at all objects and letting \(q\colon
  p\to p q\) act by \(\delta_x\mapsto \delta_{x q}\) for all
  \(x,p,q\in P\).  The colimit of this diagram of modules is
  isomorphic to~\(\Z[G]\) by mapping \(\Z[P]\ni\delta_x\) at the
  object~\(p\) of~\(\Cat_P\) to \(\delta_{x p^{-1}}\in\Z[G]\).  Hence
  \(M\otimes_{\Z[P]} \Z[G]\) for a right \(\Z[P]\)\nb-module~\(M\) is the
  colimit of the diagram over~\(\Cat_P\) with entries
  \(M\otimes_{\Z[P]} \Z[P] \cong M\) and with \(q\colon p\to p q\)
  acting by \(m\mapsto m\cdot q\) for all \(m\in M\), \(p,q\in P\).
  Now compare this with our description of the inductive system that
  computes~\(\K_*(\CP_1)\) to get \(\K_*(\CP_1) \cong \K_*(A)
  \otimes_{\Z[P]} \Z[G]\).

  Since the Hilbert modules~\(\Hilm_p\) are full, the Fell
  bundle~\((\CP_g)_{g\in G}\) is saturated by
  Proposition~\ref{pro:Fell_bundle_saturated}.  Then each~\(\CP_g\) is
  a proper correspondence from~\(\CP_1\) to itself and hence gives a
  class~\([\CP_g]\) in \(\KK_0(\CP_1,\CP_1)\).  This induces maps
  \((\CP_g)_*\colon \K_*(\CP_1)\to \K_*(\CP_1)\).  Since we have a
  saturated Fell bundle, we have \(\CP_g\otimes_{\CP_1} \CP_h \cong
  \CP_{gh}\).  Therefore, \(g\mapsto (\CP_g)_*\) defines a
  representation of~\(G^\op\) on the Abelian group~\(\K_*(\CP_1)\); we
  view this as a right \(\Z[G]\)-module structure.

  To describe this action, it suffices to compute, for \(p\in P\),
  how~\((\CP_g)_*\) acts on the image of~\(\K_*(\Comp(\Hilm_p))\)
  in~\(\K_*(\CP_1)\) under the map \(\K_*(\vartheta^0_p)\) induced by
  \(\vartheta^0_p\colon \Comp(\Hilm_p)\to \CP_1\).  First choose
  \((p_1,p_2)\in R_g\) and then \(q\in P\) with \(p q = p_1\).  Then
  \(\K_*(\vartheta^0_p) = \K_*(\vartheta^0_{p_1}) \circ
  \K_*(\varphi_{p,q})\), so it suffices to describe how~\((\CP_g)_*\)
  acts on the image of~\(\K_*(\Comp(\Hilm_{p_1}))\).  The map
  \(\Comp(\Hilm_{p_2},\Hilm_{p_1})\to\CP_g\) shows that
  \((\vartheta^0_{p_1})^*(\CP_g) \cong \Comp(\Hilm_{p_2},\Hilm_{p_1})
  \otimes_{\vartheta^0_{p_2}} \CP_1\) as correspondences
  from~\(\Comp(\Hilm_{p_1})\) to~\(\CP_1\).  Now compose these
  correspondences with the \(\KK\)\nb-equivalences between
  \(\Comp(\Hilm_{p_1})\) and~\(A\).  Then we see that~\((\CP_g)_*\)
  acts on the entry~\(\K_*(A)\) at~\(p_1\) in the inductive system
  describing \(\K_*(\CP_1)\) by sending it to the same entry
  at~\(p_2\).  Right multiplication by \(g = p_1 p_2^{-1}\) in
  \(\K_*(A) \otimes_{\Z[P]} \Z[G]\) has the same effect.  Thus the
  action of~\(G\) on~\(\K_*(\CP_1)\) induced by the Fell bundle
  corresponds to the one by right multiplication on \(\K_*(A)
  \otimes_{\Z[P]} \Z[G]\).
\end{proof}

By the Packer--Raeburn Stabilisation Trick, there is a \(G\)\nb-action
by automorphisms on the stabilisation \(\tilde{\CP}_1 \defeq
\CP_1\otimes\Comp(L^2G)\) such that the (full) crossed product
\(G\ltimes \tilde{\CP}_1\) is Morita--Rieffel equivalent to the
Cuntz-Pimsner algebra~\(\CP\) (this follows from
\cite{Buss-Meyer-Zhu:Higher_twisted}*{Corollary 5.5}).  Thus computing
the \(\K\)\nb-theory of the Cuntz--Pimsner algebra becomes a matter
for the (full) Baum--Connes conjecture for~\(G\) with certain
coefficients.

For a-T-menable groups, the Baum--Connes assembly map is known to be
an isomorphism for all coefficients, also for the full crossed product
(see~\cite{Higson-Kasparov:E_and_KK}).  The meaning of the
Baum--Connes conjecture here is that we may compute \(\K_*(\CP)\) by
topological means from \(\K_*(\CP|_H)\), the section algebras for
restrictions of~\((\CP_g)_{g\in G}\) to all finite subgroups~\(H\).
These topological means may be expressed as a spectral sequence, and
it can be quite hard to perform this computation in practice.  At
least, the results above show that the computation for a
Cuntz--Pimsner algebra over~\(P\) is not more difficult than in the
special case of an action of~\(G\) by automorphisms.

For instance, let \(P=(\N^k,+)\)
for \(k\in\N\).
Then \(G=\Z^k\),
and the computation of \(\K_*(\CP)\)
is a matter of iterating the Pismner--Voiculescu sequence \(k\)~times.
We will consider a concrete case where a \(\K\)\nb-theory
computation along these lines is feasible in
Section~\ref{sec:higher-rank_DR}.  Already two iterations may be very
hard because the boundary maps for the second iteration are not
determined by the original data.

\begin{remark}
  \label{rem:PV_twice_Deaconu}
  The iteration of the Pimsner--Voiculescu sequence that we get is
  equivalent to one by Deaconu in~\cite{Deaconu:Iterating}.  This is
  because the Pimsner--Voiculescu sequence for \(\Z\)\nb-actions
  \emph{can} be obtained from the Cuntz--Toeplitz algebra of the
  product system over~\(\N\) associated to the \(\Z\)\nb-action.
\end{remark}

Our theory contains the case of semigroup crossed products for
actions of left Ore monoids by endomorphisms by
Remark~\ref{rem:left_right}.  In that case, \((\Hilm_p)_*\colon
\K_*(A)\to \K_*(A)\) is simply the map induced by the underlying
endomorphism.

The K\nb-theory computation for semigroup \(\Cst\)\nb-algebras
in~\cite{Cuntz-Echterhoff-Li:K-theory} also uses the Baum--Connes
isomorphism for the group~\(G\).  In the situation
of~\cite{Cuntz-Echterhoff-Li:K-theory}, a dilation of the semigroup
action to an action of~\(G\) on a larger \(\Cst\)\nb-algebra is easy
to write down by hand, giving a direct route to \(\K_*(\CP_1)\).
But then extra assumptions on the action of~\(G\) on \(\K_*(\CP_1)\)
are needed to compute \(\K_*(\CP)\).

\subsection{Making left actions faithful}
\label{sec:making_faithful}

Let \((A,\Hilm_p,\mu_{p,q})\) be a proper product system over an Ore
monoid~\(P\).  Taking suitable quotients of~\(A\) and~\(\Hilm_p\), we
are going to construct another product system
\((A',\Hilm'_p,\mu'_{p,q})\) with the same nondegenerate
representations and hence the same Cuntz--Pimsner algebra, such that
the left actions \(\varphi'_p\colon A'\to\Bound(\Hilm'_p)\) are
injective for all \(p\in P\).

For \(p\in P\), let \(\varphi_p\colon A\to \Comp(\Hilm_p)\) denote
the left action map and let \(I_p\defeq \ker \varphi_p\).  Recall
the maps \(\varphi_{p,q}\colon \Comp(\Hilm_p) \to
\Comp(\Hilm_{pq})\) for \(p,q\in P\).  Since
\(\varphi_{p,q}\circ\varphi_p = \varphi_{pq}\), we have
\(I_p\subseteq I_{pq}\) for all \(p,q\in P\).  Since~\(\Cat_P\) is
filtered, this implies that the ideals~\(I_p\) form a directed set
of ideals in~\(A\).  Thus \(I \defeq \overline{\bigcup_{p\in P}
  I_p}\) is another ideal in~\(A\).  We let \(A'\defeq A/I\) and
\(\Hilm'_p \defeq \Hilm_p\otimes_A A' \cong
\Hilm_p\mathbin/(\Hilm_p\cdot I)\).

\begin{lemma}
  \label{lem:Aprime_acts_Hilmprime}
  The induced left action \(A\to\Comp(\Hilm_p')\) factors
  through~\(A'\), and the isomorphism \(\mu_{p,q}\colon
  \Hilm_p\otimes_A \Hilm_q \congto \Hilm_{pq}\) descends to an
  isomorphism of correspondences
  \(\mu'_{p,q}\colon \Hilm'_p\otimes_A \Hilm'_q \congto
  \Hilm'_{pq}\).  This gives a product system
  \((A',\Hilm'_p,\mu'_{p,q})\).
\end{lemma}

\begin{proof}
  Let \(p\in P\).  To prove that the induced left \(A\)\nb-module
  structure on~\(\Hilm'_p\) descends to~\(A'\), we must show that \(I
  \Hilm_p\subseteq \Hilm_p I\).  Since~\(P\) is an Ore monoid, the
  subset \(pP\) is cofinal in~\(P\), so \(\bigcup_{q\in P} I_{pq}\) is
  still dense in~\(I\).  Thus it suffices to prove \(I_{pq}
  \Hilm_p\subseteq \Hilm_p I\) for all \(p,q\in P\).  We will prove
  the following more precise result:
  \begin{equation}
    \label{eq:I_pq}
    I_{p q} = \{a\in A\mid a\Hilm_p\subseteq \Hilm_p I_q\}.
  \end{equation}
  Let \(\xi\in\Hilm_p\).  We have \(\xi\otimes_A \eta=0\) in
  \(\Hilm_p\otimes_A \Hilm_q\) for all \(\eta\in\Hilm_q\) if and only
  if
  \[
  0 = \braket{\xi\otimes\eta_1}{\xi\otimes\eta_2}
  = \braket{\eta_1}{\varphi_q(\braket{\xi}{\xi}_A)\eta_2}
  \]
  for all \(\eta_1,\eta_2\in\Hilm_q\), if and only
  \(\varphi_q(\braket{\xi}{\xi}_A)=0\), if and only if
  \(\braket{\xi}{\xi}_A\in I_q\).  We claim that this is equivalent to
  \(\xi\in\Hilm_p\cdot I_q\).  Since~\(I_q\) is an ideal, we have
  \(\braket{\xi}{\xi}_A\in I_q\) for \(\xi\in\Hilm_p\cdot I_q\).
  Conversely, if \(\braket{\xi}{\xi}_A\in I_q\), then the closure of
  \(\xi\cdot A\) in~\(\Hilm_p\) is a Hilbert \(I_q\)\nb-module
  containing~\(\xi\), and
  thus it is nondegenerate as a right \(I_q\)\nb-module, so that
  \(\xi\in \Hilm_p I_q\).  Hence \(\xi\otimes_A\eta=0\) in
  \(\Hilm_p\otimes_A \Hilm_q\) for all \(\eta\in\Hilm_q\) if and only
  \(\xi\in\Hilm_p I_q\).

  Now let \(a\in A\).  Then \(a\xi\in\Hilm_p I_q\) for all
  \(\xi\in\Hilm_p\) if and only if \(a\xi\otimes_A\eta=0\) for all
  \(\xi\in\Hilm_p\), \(\eta\in\Hilm_q\), if and only if the left
  action by~\(a\) vanishes on \(\Hilm_p\otimes_A \Hilm_q \cong
  \Hilm_{pq}\).  This is equivalent to \(a\in I_{pq}\).  This
  finishes the proof of~\eqref{eq:I_pq}.  In turn, this implies that
  the left \(A\)\nb-module structure on~\(\Hilm'_p\) descends
  to~\(A'\).  Now
  \[
  \Hilm'_p \otimes_{A'} \Hilm'_q
  = \Hilm_p \otimes_A A' \otimes_{A'} \Hilm_q \otimes_A A'
  = \Hilm_p \otimes_A \Hilm_q \otimes_A A'
  \cong \Hilm_{p q} \otimes_A A'
  = \Hilm'_{p q}.
  \]
  This gives the multiplication maps~\(\mu'_{p,q}\).  From another
  point of view, \(\mu'_{p,q}\) is the map on the quotient spaces
  \(\Hilm'_p\) and~\(\Hilm'_q\) induced by~\(\mu_{p,q}\).  Hence these
  maps inherit associativity from the maps~\(\mu_{p,q}\), so we have
  constructed a product system.
\end{proof}

\begin{theorem}
  \label{the:make_faithful}
  The product system \((A',\Hilm'_p,\mu'_{p,q})\) has faithful left
  action maps \(A' \to \Comp(\Hilm'_p)\), and it has the same
  nondegenerate representations as the original system.  Hence it also
  has the same Cuntz--Pimsner algebra.
\end{theorem}

\begin{proof}
  Fix \(p\in P\).  An operator on~\(\Hilm_p\) induces the zero
  operator on \(\Hilm_p/\Hilm_p I_q \cong \Hilm_p \otimes_A (A/I_q)\)
  if and only if it maps~\(\Hilm_p\) into~\(\Hilm_p I_q\).
  Thus~\eqref{eq:I_pq} shows that the map \(\varphi_p\colon
  A\to\Comp(\Hilm_p)\) descends to an injective \Star{}homomorphism
  \(A/I_{p q} \hookrightarrow \Comp(\Hilm_p/\Hilm_p I_q)\).  The
  \(\Cst\)\nb-algebras \(A/I_{p q}\) and \(\Comp(\Hilm_p/\Hilm_p
  I_q)\) for \(q\in P\) form inductive systems indexed by the filtered
  category~\(\Cat_P\), and the maps \(A/I_{p q} \hookrightarrow
  \Comp(\Hilm_p/\Hilm_p I_q)\) form a morphism of inductive systems,
  consisting of injective maps.  It follows that the induced map
  between the inductive limits \(\varinjlim A/I_{p q} = A/
  \overline{\bigcup_{q\in P} I_{p q}} = A'\) and \(\varinjlim
  \Comp(\Hilm_p/\Hilm_p I_q) = \Comp(\Hilm'_p)\) is injective as well.
  That is, the left action \(A' \to \Comp(\Hilm'_p)\) is faithful.

  Now let \(\vartheta'\colon A'\to \Bound(\Hilm[F])\) and \(S'_p\colon
  \Hilm'_p\to\Bound(\Hilm[F])\) for \(p\in P\) give a nondegenerate
  representation of the product system \((A',\Hilm'_p,\mu'_{p,q})\).
  Composing with the quotient maps \(A\to A'\) and
  \(\Hilm_p\to\Hilm'_p\) then gives a nondegenerate representation of
  \((A,\Hilm_p,\mu_{p,q})\).  We claim that any nondegenerate
  representation~\((\vartheta,S_p)\) of \((A,\Hilm_p,\mu_{p,q})\)
  factors through the quotient maps \(A\to A'\) and
  \(\Hilm_p\to\Hilm'_p\) and thus comes from a unique
  representation~\((\vartheta',S'_p)\).  This gives a bijection on the
  level of nondegenerate representations and thus an isomorphism of
  Cuntz--Pimsner algebras because they are universal for nondegenerate
  representations by Proposition~\ref{pro:CP_vs_nondegenerate}.

  Recall the maps \(\vartheta_p\colon
  \Comp(\Hilm_p)\to\Bound(\Hilm[F])\) with \(\vartheta_{pq}\circ
  \varphi_{p,q} = \vartheta_p\) for all \(p,q\in P\).  In particular,
  \(\vartheta= \vartheta_p\circ \varphi_p\colon A\to
  \Bound(\Hilm[F])\), so \(\vartheta\) must vanish on~\(I_p\).  Since
  this holds for all \(p\in P\), we get \(\vartheta|_I=0\),
  so~\(\vartheta\) factors through the quotient map \(A\to A'\).
  Since \(S_p(\xi)^* S_p(\xi) = \vartheta(\braket{\xi}{\xi}_A)\) for
  \(\xi\in \Hilm_p\) and \(\braket{\xi}{\xi} \in I\) for \(\xi\in
  \Hilm_p\cdot I\), we also get \(S_p(\xi)=0\) for \(\xi\in
  \Hilm_p\cdot I\).  Thus~\(S_p\) factors through~\(\Hilm'_p\).
\end{proof}

\begin{proposition}
  \label{pro:injective_maps_CP}
  If the maps \(A\to\Comp(\Hilm_p)\) are injective for all \(p\in P\),
  then so are the induced maps \(\varphi_{p,q,t}\colon
  \Comp(\Hilm_q,\Hilm_p) \to \Comp(\Hilm_{q t},\Hilm_{p t})\) for
  \(p,q,t\in P\) and the maps \(\Comp(\Hilm_q,\Hilm_p) \to \CP\) to
  the Cuntz--Pimsner algebra.
\end{proposition}

\begin{proof}
  We assume that \(I_p=\{0\}\) for all \(p\in P\).  The proof
  of~\eqref{eq:I_pq} shows that \(\xi\in\Hilm_p\) satisfies
  \(\xi\otimes_A \eta=0\) in \(\Hilm_p \otimes_A \Hilm_q\) for all
  \(\eta\in\Hilm_q\) if and only if \(\xi=0\).  Hence the
  maps~\(\varphi_{p,q,t}\) are injective.  Since \(\CP_g\subseteq
  \CP\) is the filtered colimit of the spaces
  \(\Comp(\Hilm_q,\Hilm_p)\), this implies the same for the maps
  \(\Comp(\Hilm_q,\Hilm_p)\to \CP_g\subseteq \CP\).
\end{proof}

\subsection{What happens without the Ore conditions?}
\label{sec:non-Ore}

We now consider an example of a monoid without the Ore conditions
where we can, nevertheless, describe the Cuntz--Pimsner algebra by
hand.  Let~\(F_n^+\) be the free monoid on $n$~generators, \(n\ge2\).
Elements in~\(F_n^+\) are finite words in the letters
\(a_1,\dotsc,a_n\), including the empty word.  This monoid violates
the Ore conditions: there are no words \(w_1,w_2\in F_n^+\) with \(a_1
w_1 = a_2 w_2\).  A proper product system over~\(F_n^+\) is equivalent
to a \(\Cst\)\nb-algebra~\(A\) with proper correspondences~\(\Hilm_i\)
from~\(A\) to itself for \(i=1,\dotsc,n\), without any further data or
conditions: given this data, we may define~\(\Hilm_w\) for a
word~\(w\) by composing the correspondences for the letters in~\(w\),
and we use the canonical multiplication maps between them.

\begin{proposition}
  \label{pro:CP_over_free}
  Let~\(A\) be a \(\Cst\)\nb-algebra and let~\(\Hilm_i\) for \(1\le
  i\le n\) be proper correspondences from~\(A\) to~\(A\).
  Let~\(\CP_i\) be the Cuntz--Pimsner algebra of~\(\Hilm_i\) for
  \(1\le i\le n\).  The Cuntz--Pimsner algebra of the resulting
  product system over~\(F_n^+\) is the amalgamated free product of
  the Cuntz--Pimsner algebras~\(\CP_i\) over~\(A\).
\end{proposition}

\begin{proof}
  Let~\(D\) be another \(\Cst\)\nb-algebra and let~\(\Hilm[G]\) be a
  Hilbert module over~\(D\).  A nondegenerate representation of our
  product system over~\(F_n^+\) on~\(\Hilm[G]\) is already
  determined by what it does on the correspondences~\(\Hilm_i\),
  and~\(\Hilm_i\) may act by arbitrary nondegenerate representations
  because~\(F_n^+\) is a free monoid.  A nondegenerate
  representation of~\(\Hilm_i\) is equivalent to a representation of
  the Cuntz--Pimsner algebra~\(\CP_i\) by
  Proposition~\ref{pro:CP_vs_nondegenerate}.  Since all these
  representations give the same representation when we compose with
  the canonical map \(A\to\CP_i\), we get a representation of the
  amalgamated free product of the \(\Cst\)\nb-algebras~\(\CP_i\)
  over~\(A\).  Conversely, a representation of this free product
  gives nondegenerate representations of the
  correspondences~\(\Hilm_i\) and thus of~\(A\), and it gives the
  same representation on~\(A\) for each~\(i\).  This data may be
  extended to a nondegenerate representation of the product system
  over~\(F_n^+\).
\end{proof}

Free products with amalgamation are, unfortunately, rather large and
complicated.  In particular, they are almost never nuclear or exact.
Thus we view Proposition~\ref{pro:CP_over_free} as a negative
result: it tells us that we should not expect Cuntz--Pimsner
algebras for proper product systems over~\(F_n^+\) to have a nice
structure.  Standard assumptions in the theory of Cuntz--Toeplitz
and Cuntz--Pimsner algebras are that the underlying semigroup be
``quasi-lattice-ordered'' and the product system ``compactly
aligned,'' see~\cite{Fowler:Product_systems}.  Both assumptions are
satisfied for proper product systems over~\(F_n^+\).  If two
elements in~\(F_n^+\) have an upper bound, they have a least upper
bound because two elements in~\(F_n^+\) only have an upper bound if
one of them is a subword of the other, and then the longer of the
two is a least upper bound.  Hence Cuntz--Pimsner algebras of
compactly aligned product systems over quasi-lattice-ordered monoids
need not be tractable.

\subsection{Higher-rank Doplicher--Roberts algebras}
\label{sec:higher-rank_DR}

In this section, we consider higher-rank analogues of the
\(\Cst\)\nb-algebras introduced by Doplicher and Roberts
in~\cite{Doplicher-Roberts:Duals}.  The Doplicher--Roberts
\(\Cst\)\nb-algebras were an important motivation for Kumjian, Pask,
Raeburn and Renault when they defined graph \(\Cst\)\nb-algebras
in~\cite{Kumjian-Pask-Raeburn-Renault:Graphs}.

Our higher-rank analogue is constructed from a compact Lie
group~\(G\) and finite-dimensional representations
\(\pi_1,\dotsc,\pi_k\) of~\(G\); in addition, we need a
representation \(\rho\colon G\to\U(\Hils)\) on a Hilbert
space~\(\Hils\) that contains each irreducible representation
of~\(G\).  Different choices for~\(\rho\) will, however, give
Morita--Rieffel equivalent \(\Cst\)\nb-algebras, so we
consider~\(\rho\) to be auxiliary data only.  From the above data,
we are going to construct a product system over the commutative
monoid~\((\N^k,+)\) and then take its Cuntz--Pimsner algebra.  The
case \(k=1\) is considered
in~\cite{Kumjian-Pask-Raeburn-Renault:Graphs}.

For \(m = (m_1,\dotsc,m_k)\in\N^k\), we form the representation
\[
\pi^m \defeq \pi_1^{\otimes m_1} \otimes \dotsb \otimes
\pi_k^{\otimes m_k}\colon G \to \U(V^m);
\]
here~\(V^m\) denotes the finite-dimensional Hilbert space on
which~\(\pi^m\) acts.  There are canonical unitary operators
\(\mu_{m_1,m_2}\colon V^{m_1} \otimes V^{m_2} \cong V^{m_1+m_2}\)
that intertwine the representations \(\pi^{m_1} \otimes \pi^{m_2}\)
and~\(\pi^{m_1+m_2}\), and which satisfy the properties of a
symmetric monoidal category.

Let \(\Hilm_m \subseteq \Comp(\Hils,V^m\otimes\Hils)\) be the space
of all compact intertwining operators between the representations
\(\rho\) and~\(\pi^m\otimes\rho\).  Define multiplication maps
\(\Hilm_{m_1} \times \Hilm_{m_2} \to \Hilm_{m_1+m_2}\) by
mapping~\((T_1,T_2)\) to the composite intertwining operator
\[
\rho
\xrightarrow{T_2} \pi^{m_2}\otimes\rho
\xrightarrow{1\otimes T_1} \pi^{m_2}\otimes\pi^{m_1}\otimes\rho
\xrightarrow{\mu_{m_2,m_1}\otimes1} \pi^{m_1+m_2}\otimes\rho;
\]
this composite is compact because~\(T_2\) is compact.

\begin{lemma}
  \label{lem:associative_DR}
  The multiplication above is associative.
\end{lemma}

\begin{proof}
  Let \(m_1,m_2,m_3\in\N^k\) and let \(T_i\in\Hilm_{m_i}\) for
  \(i=1,2,3\).  Then the products \((T_1 T_2) T_3\) and \(T_1 (T_2
  T_3)\) are equal to the composite operators
  \[
  \begin{tikzpicture}
    \matrix (m) [cd] {
      \rho&
      \pi^{m_3}\otimes\rho&
      \pi^{m_3}\otimes\pi^{m_2}\otimes\rho&&
      \pi^{m_3}\otimes\pi^{m_2}\otimes\pi^{m_1}\otimes\rho\\
      && \pi^{m_3+m_2}\otimes\rho&&
      \pi^{m_3+m_2}\otimes\pi^{m_1}\otimes\rho\\
      &&&&\pi^{m_3+m_2+m_1}\otimes\rho\\
    };
    \draw[cdar] (m-1-1) -- node {\(\scriptstyle T_3\)} (m-1-2);
    \draw[cdar] (m-1-2) -- node {\(\scriptstyle 1\otimes T_2\)} (m-1-3);
    \draw[cdar] (m-1-3) -- node {\(\scriptstyle 1\otimes 1\otimes T_1\)} (m-1-5);
    \draw[cdar] (m-1-3) -- node[swap] {\(\scriptstyle \mu_{m_3,m_2}\otimes 1\)} (m-2-3);
    \draw[cdar] (m-1-5) -- node {\(\scriptstyle \mu_{m_3,m_2}\otimes 1\)} (m-2-5);
    \draw[cdar] (m-2-3) -- node {\(\scriptstyle 1\otimes 1\otimes T_1\)} (m-2-5);
    \draw[cdar] (m-2-5) -- node {\(\scriptstyle \mu_{m_3+m_2,m_1}\otimes 1\)} (m-3-5);
  \end{tikzpicture}
  \]
  because
  \(\mu_{m_3+m_2,m_1} \circ (\mu_{m_3,m_2}\otimes 1) =
  \mu_{m_3,m_2+m_1} \circ (1\otimes \mu_{m_2,m_1})\).
\end{proof}

The unit fibre~\(\Hilm_0\) is the \(\Cst\)\nb-algebra of all compact
intertwining operators of~\(\rho\).  The multiplication maps above
turn each~\(\Hilm_m\) into an \(\Hilm_0\)-bimodule.  We define an
\(\Hilm_0\)-valued right inner product on~\(\Hilm_m\) by
\(\braket{T_1}{T_2} \defeq T_1^* T_2\) for \(T_1,T_2\in\Hilm_m\).
This turns each~\(\Hilm_m\) into a correspondence from the
\(\Cst\)\nb-algebra \(\Hilm_0\) to itself.

\begin{lemma}
  \label{lem:DR_compact_operators}
  The \(\Cst\)\nb-algebra \(\Comp(\Hilm_m)\) is isomorphic to the
  \(\Cst\)\nb-algebra of compact intertwiners of the representation
  \(\pi^m\otimes\rho\), acting on~\(\Hilm_m\) by left
  multiplication.  More generally,
  \(\Comp(\Hilm_{m_1},\Hilm_{m_2})\) is isomorphic to the space of
  compact intertwiners \(\pi^{m_1}\otimes\rho\to
  \pi^{m_2}\otimes\rho\).
\end{lemma}

\begin{proof}
  The map sending \(\ket{T_1}\bra{T_2}\in\Comp(\Hilm_m)\) for
  \(T_1,T_2\in\Hilm_m\) to the intertwiner \(T_1 T_2^*\colon
  \pi^m\otimes\rho\to\pi^m\otimes\rho\) extends to a
  \Star{}homomorphism from \(\Comp(\Hilm_m)\) to the
  \(\Cst\)\nb-algebra of compact intertwiners of
  \(\pi^m\otimes\rho\).  Since \(\ket{T_1}\bra{T_2} T_3 = T_1 T_2^*
  T_3\), this representation is faithful.  It remains to show that
  it is surjective.

  Any compact intertwiner on \(\pi^m\otimes\rho\) may be
  approximated by linear combinations of intertwiners with
  irreducible range because the representation~\(\pi^m\otimes\rho\),
  like any representation of~\(G\), is a direct sum of irreducible
  representations.  Since any irreducible representation of~\(G\)
  occurs in~\(\rho\), any intertwiner with irreducible range factors
  through the representation~\(\rho\).  Thus we may write it as
  \(T_1 T_2^*\) for \(T_1,T_2\in \Hilm_m\).  This shows that
  \(\Comp(\Hilm_m)\) is mapped onto the \(\Cst\)\nb-algebra of
  compact intertwiners of \(\pi^m\otimes\rho\).

  The same argument still works in the more general case of
  \(\Comp(\Hilm_{m_1},\Hilm_{m_2})\).
\end{proof}

If \(T\in\Hilm_0\), then the induced operator \(1\otimes T\colon
\pi^m\otimes\rho \to \pi^m\otimes \rho\) is compact as well because
the representation~\(\pi^m\) has finite dimension.  Thus the
correspondence~\(\Hilm_m\) is proper by
Lemma~\ref{lem:DR_compact_operators}.

\begin{lemma}
  \label{lem:mult_DR_surjective}
  The multiplication maps induce unitary operators \(\Hilm_{m_1}
  \otimes_{\Hilm_0} \Hilm_{m_2} \to \Hilm_{m_1+m_2}\).
\end{lemma}

\begin{proof}
  It is routine to check that the map \(\Hilm_{m_1} \times
  \Hilm_{m_2} \to \Hilm_{m_1+m_2}\) defined above preserves the
  inner products, so it gives an isometry \(\Hilm_{m_1}
  \otimes_{\Hilm_0} \Hilm_{m_2} \to \Hilm_{m_1+m_2}\).  This induces
  a \Star{}homomorphism \(\Comp(\Hilm_{m_1}) \to
  \Comp(\Hilm_{m_1+m_2})\), \(T\mapsto T\otimes1\).  In terms of
  Lemma~\ref{lem:DR_compact_operators}, this is given by the map
  \(T\mapsto T\otimes 1\) from compact intertwiners of
  \(\pi^{m_1}\otimes\rho\) to compact intertwiners of
  \(\pi^{m_1}\otimes\pi^{m_2}\otimes\rho \cong
  \pi^{m_1+m_2}\otimes\rho\).  This \Star{}homomorphism on compact
  operators is nondegenerate.  Hence the underlying isometry
  \(\Hilm_{m_1} \otimes_{\Hilm_0} \Hilm_{m_2} \to \Hilm_{m_1+m_2}\)
  must be surjective.
\end{proof}

Lemma~\ref{lem:mult_DR_surjective} says that the
correspondences~\(\Hilm_m\) with the above multiplication maps form
an essential product system over the commutative
monoid~\((\N^k,+)\).  As remarked above,
Lemma~\ref{lem:DR_compact_operators} implies that this product
system is proper.  Since any irreducible representation occurs
in~\(\rho\), the Hilbert \(\Hilm_0\)-module~\(\Hilm_m\) is full and
carries a faithful left \(\Hilm_0\)-action.

\begin{definition}
  \label{def:DR}
  The Cuntz--Pimsner algebra of the product
  system~\((\Hilm_m)_{m\in\N^k}\) over~\((\N^k,+)\) is the
  \emph{higher-rank Doplicher--Roberts algebra} for the
  representations \(\pi_1,\dotsc,\pi_n\) of~\(G\), relative
  to~\(\rho\).
\end{definition}

\begin{lemma}
  \label{lem:DR_unique}
  The higher-rank Doplicher--Roberts algebras for different choices
  of~\(\rho\) are canonically Morita--Rieffel equivalent.
\end{lemma}

\begin{proof}
  Let \(\rho\) and~\(\rho'\) be two representations of~\(G\) that
  contain all irreducible representations.
  Lemma~\ref{lem:DR_compact_operators} identifies \(\Hilm_0^\rho\)
  and~\(\Hilm_0^{\rho'}\) with the \(\Cst\)\nb-algebras of compact
  intertwiners on the representations \(\rho\) and~\(\rho'\),
  respectively.  Let \(\Hilm[F]_{\rho \rho'}\) be the space of all
  compact intertwining operators \(\rho'\to\rho\).  This is a full
  Hilbert bimodule for \(\Hilm_0^\rho\) and \(\Hilm_0^{\rho'}\).
  Furthermore, we may naturally identify both \(\Hilm[F]_{\rho
    \rho'} \otimes_{\Hilm_0^{\rho'}} \Hilm_m^{\rho'}\) and
  \(\Hilm_m^\rho \otimes_{\Hilm_0^\rho} \Hilm[F]_{\rho \rho'}\) with
  the space of compact intertwiners from \(\rho'\) to
  \(\pi^m\otimes\rho\).  These identifications provide a Morita
  equivalence between the product systems for \(\rho\) and~\(\rho'\)
  and thus induce a Morita--Rieffel equivalence between their
  Cuntz--Pimsner algebras.
\end{proof}

To clarify the link to previous constructions, take \(k=1\) and
let~\(\rho\) be the direct sum of all irreducible representations
with multiplicity~\(1\).  Then the \(\Cst\)\nb-algebra~\(\Hilm_0\)
of compact intertwiners of~\(\rho\) is \(\Cont_0(\hat{G})\).  Since
\(k=1\), our product system is determined by the single
self-correspondence
of~\(\Cont_0(\hat{G})\) given by~\(\Hilm_1\).  Such a
self-correspondence is equivalent to a graph with vertex
set~\(\hat{G}\).  Since the left action on~\(\Hilm_1\) is faithful
and~\(\Hilm_1\) is proper and full, our graph has neither sources
nor sinks and no infinite emitters.  Hence our absolute
Cuntz--Pimsner algebra agrees with the relative one used to define
graph \(\Cst\)\nb-algebras.  Our Doplicher--Roberts algebra is
exactly the graph \(\Cst\)\nb-algebra considered in
\cite{Kumjian-Pask-Raeburn-Renault:Graphs}*{Section~7}.  As shown
there, the \(\Cst\)\nb-algebra defined by Doplicher and Roberts
in~\cite{Doplicher-Roberts:Duals} is isomorphic to a full corner in
this graph \(\Cst\)\nb-algebra (assuming that each irreducible
representation of~\(G\) occurs in~\(\pi^m\) for some \(m\in\N\)).

\begin{remark}
  For \(k>1\), it seems unlikely that our higher-rank
  Doplicher--Roberts algebras are higher-rank graph
  \(\Cst\)\nb-algebras.  For \(k=1\), any product system
  over~\((\N^k,+)\) with unit fibre of the form~\(\Cont_0(V)\) for a
  discrete set~\(V\) (``vertices'') gives a higher-rank graph
  \(\Cst\)\nb-algebra.  For \(k>1\), this fails: we also need the
  multiplication isomorphisms in the product system to be given by
  permutation matrices in some chosen bases for our
  self-correspondences.
\end{remark}

Let~\(\mathcal{D}\) denote our higher-rank Doplicher--Roberts
algebra.  The general theory above applies here and shows that
\(\mathcal{D} = \bigoplus_{m\in \Z^k} \mathcal{D}_m\) is the section
\(\Cst\)\nb-algebra of a Fell bundle~\((\mathcal{D}_m)_{m\in\Z^k}\)
over~\(\Z^k\), where~\(\mathcal{D}_m\) is the filtered colimit of
the Banach spaces \(\Comp(\Hilm_{a},\Hilm_{a+m})\) for \(a\in\N^k\)
with \(a+m\in\N^k\).  Here Lemma~\ref{lem:DR_compact_operators}
identifies this Banach space with the space of compact intertwiners
\(\pi^a\otimes\rho \to \pi^{a+m}\otimes\rho\).  In particular, the
zero fibre~\(\mathcal{D}_0\) is the inductive limit of the system of
\(\Cst\)\nb-algebras \(\Comp(\Hilm_a)\) for \(a\in\N^k\); the
sequence given by \(a=(a_1,a_1,\dotsc,a_1)\) for \(a_1\in\N\) is
cofinal in~\(\N^k\),
so we may as well take this sequence.

For each \(a\in\N^k\), \(\Comp(\Hilm_a)\) is a \(\Cst\)\nb-algebra
of compact operators, hence a direct sum of matrix algebras.  The
summands are in bijection with~\(\hat{G}\) because all irreducible
representations of~\(G\) occur in~\(\rho\) and hence in
\(\pi^a\otimes\rho\).  In particular, \(\Comp(\Hilm_a)\) is
Morita--Rieffel equivalent to \(\Cst(G)\) for each \(a\in\N^k\), and
it is an AF-algebra whose K-theory is equal to the K-theory
\(\K_0(\Cst(G))\) of the group \(\Cst\)\nb-algebra of~\(G\) and
hence to the representation ring of~\(G\).  Concretely, its elements
are functions \(\hat{G}\to(\Z,+)\), and the positive cone in
\(\K_0(\Comp(\Hilm_a))\) consists of all functions
\(\hat{G}\to(\N,+)\).

A countable inductive limit of AF-algebras remains an AF-algebra,
so~\(\mathcal{D}_0\) is AF.  We compute its \(\K\)\nb-theory.  The
group~\(\K_0(\Cst(G))\) is a commutative ring through the tensor
product of representations: the representation ring of~\(G\).  The
map
\[
\K_0(\Cst(G)) \cong \K_0(\Comp(\Hilm_a)) \to
\K_0(\Comp(\Hilm_{a+m})) \cong \K_0(\Cst(G))
\]
induced by the canonical \Star{}homomorphism
\(\Comp(\Hilm_a)\to\Comp(\Hilm_{a+m})\) is the multiplication with
\([\pi^m]\) in the ring structure on \(\K_0(\Cst(G))\).  Hence
\[
\K_0(\mathcal{D}_0) = \varinjlim
\bigl(\K_0(\Cst(G)) \xrightarrow{[\pi^{1_k}]}
\K_0(\Cst(G)) \xrightarrow{[\pi^{1_k}]}
\K_0(\Cst(G)) \to \dotsb
\bigr),
\]
where \([\pi^{1_k}]\) denotes multiplication with the class of the
representation \(\pi_1\otimes\dotsb\otimes\pi_k\) in the
representation ring.  This inductive limit is the localisation of
the representation ring \(\K_0(\Cst(G))\) of~\(G\) in which we
invert~\([\pi^{1_k}]\).  Our K\nb-theory computation
determines~\(\mathcal{D}_0\) uniquely up to Morita--Rieffel
equivalence by Elliott's classification of AF-algebras.

Each~\(\mathcal{D}_m\) is a Hilbert bimodule from~\(\mathcal{D}_0\)
to itself, which acts on~\(\K_0\) by multiplication with the class
of~\(\pi^m\).  From this information, it is sometimes possible to
compute the \(\K\)\nb-theory of~\(\mathcal{D}\).  We do not pursue
this in general but merely consider one special case where we can
completely describe the higher-rank Doplicher--Roberts algebra.

\begin{theorem}
  \label{the:DR_example}
  Let \(G=\mathrm{SU}(n)\)
  for \(n\ge2\)
  and let \(\pi_i = \Lambda^i(\C^n)\in\hat{G}\)
  for \(i=1,\dotsc,n-1\)
  be the exterior powers of the standard representation on~\(\C^n\).
  The associated rank-\(n-1\)
  Doplicher--Roberts algebra~\(\mathcal{D}\)
  is purely infinite, simple, separable, nuclear and in the bootstrap
  class, and has \(\K_0(\mathcal{D})=\Z\), \(\K_1(\mathcal{D})=0\).
\end{theorem}

Kirchberg's Classification Theorem implies that~\(\mathcal{D}\)
is isomorphic to the Cuntz algebra~\(\mathcal{O}_\infty\),
but we would not expect this isomorphism to be constructible.

\begin{proof}
  The representations \(\pi_1,\dotsc,\pi_{n-1}\in\hat{G}\)
  are the fundamental representations of \(\mathrm{SU}(n)\),
  that is, they are irreducible and generate a ring isomorphism
  \(\Z[x_1,\dotsc,x_{n-1}]\mapsto R(G)\),
  \(x_i\mapsto [\pi_i]\).

  Our Fell bundle description of~\(\mathcal{D}\)
  shows that it is stably isomorphic to a crossed product of an action
  of~\(\Z^{n-1}\)
  on an AF-algebra.  Hence it is separable, nuclear and in the
  bootstrap class.  We compute the \(\K\)\nb-theory
  of~\(\mathcal{D}\)
  by iterating the Pimsner--Voiculescu exact sequence \(n-1\)~times.
  We write \(\mathcal{D}\rtimes\Z^i\)
  for the crossed product with \(\Z^i\subseteq \Z^{n-1}\),
  although this is really a Fell bundle section algebra; the action by
  automorphisms only occurs on a stably isomorphic
  \(\Cst\)\nb-algebra.

  The AF-algebra~\(\mathcal{D}_0\)
  has \(\K_0(\mathcal{D}_0) = \Z[x_1^{\pm1},\dotsc,x_{n-1}^{\pm1}]\)
  because the representation ring of~\(G\)
  is isomorphic to the polynomial algebra \(\Z[x_1,\dotsc,x_{n-1}]\)
  with \(x_i=[\pi_i]\)
  and localising at the elements \(x_1,\dotsc,x_{n-1}\)
  simply adjoins their inverses.  The elements
  \(1-x_1,\dotsc,1-x_{n-1}\)
  form a regular sequence in this algebra; that is, for each~\(i\),
  multiplication by \(1-x_{i+1}\)
  is injective on the quotient
  \(\Z[x_1^{\pm1},\dotsc,x_{n-1}^{\pm1}]/(1-[x_1],\dotsc,1-[x_i])\)
  by the ideal generated by \(1-[x_1],\dotsc,1-[x_i]\).
  This is what allows us to compute the \(\K\)\nb-theory
  by repeated application of the Pimsner--Voiculescu exact sequence.

  In each step, we are supposed to consider the kernel and cokernel of
  the map \(1-\alpha_i\)
  on \(\K_*(\mathcal{D}_0\rtimes \Z^{i-1})\),
  where~\(\alpha_i\)
  is induced by the action of the \(i\)th
  factor of~\(\Z\)
  on \(\mathcal{D}_0\rtimes \Z^{i-1}\).
  By induction, we show that \(\K_1(\mathcal{D}_0\rtimes \Z^i)\)
  vanishes and that \(\K_0(\mathcal{D}_0\rtimes \Z^i)\)
  is the quotient ring
  \(\Z[x_1^{\pm1},\dotsc,x_{n-1}^{\pm1}]/(1-x_1,\dotsc,1-x_i) \cong
  \Z[x_{i+1}^{\pm1},\dotsc,x_{n-1}^{\pm1}]\).
  This is clear for \(i=0\).
  In each induction step, we use the Pimsner--Voiculescu exact
  sequence.  Since we have a regular sequence, multiplication by
  \(1-x_i\)
  is injective on
  \(\Z[x_1^{\pm1},\dotsc,x_{n-1}^{\pm1}]/(1-x_1,\dotsc,1-x_{i-1})\).
  Thus \(\K_1(\mathcal{D}_0\rtimes \Z^i)\)
  vanishes and \(\K_0(\mathcal{D}_0\rtimes \Z^i)\)
  is \(\Z[x_1^{\pm1},\dotsc,x_{n-1}^{\pm1}]/(1-x_1,\dotsc,1-x_i)\).
  After \(n-1\) steps, we get \(\K_1(\mathcal{D})=0\) and
  \[
  \K_0(\mathcal{D}) \cong
  \Z[x_1^{\pm1},\dotsc,x_{n-1}^{\pm1}]/(1-x_1,\dotsc,1-x_{n-1})
  \cong \Z.
  \]
  Thus~\(\mathcal{D}\) has the asserted \(\K\)\nb-theory.

  Next we prove that~\(\mathcal{D}\)
  is simple.  We do not claim that the AF-algebra~\(\mathcal{D}_0\)
  is simple.  The crossed product \(\mathcal{D}_0\rtimes \Z^1\),
  however, is simple by
  \cite{Kumjian-Pask-Raeburn-Renault:Graphs}*{Corollary 7.3} because
  the representation~\(\pi_1\)
  of~\(G\)
  on~\(\C^n\)
  is faithful.  This crossed product is just a rank-\(1\)
  Doplicher--Roberts algebra, hence stably isomorphic to a graph
  algebra.  The graph algebra description of
  \(\mathcal{D}_0\rtimes \Z^1\) shows also that it is purely infinite.

  The \(\K\)\nb-theory
  computation above shows that \(\mathcal{D}_0\rtimes \Z^1\)
  has \(\K\)\nb-theory
  isomorphic to \(\Z[x_2^{\pm1},\dotsc,x_{n-1}^{\pm1}]\),
  where the automorphism associated to
  \(\pi_2^{m_2}\dotsb \pi_{n-1}^{m_{n-1}}\)
  acts by multiplication with \(x_2^{m_2}\dotsb x_{n-1}^{m_{n-1}}\).
  Since this is never the identity map, none of these automorphisms
  can be inner.  Since \(\mathcal{D}_0\rtimes \Z^1\)
  is simple, separable and purely infinite, the (reduced) crossed
  product by the group~\(\Z^{n-2}\)
  remains simple and purely infinite by
  \cite{Kishimoto-Kumjian:Crossed_Cuntz}*{Lemma 10}.  Since the
  stabilisation \((\mathcal{D}_0\rtimes \Z^1)\rtimes \Z^{n-2}\)
  of~\(\mathcal{D}\)
  is simple and purely infinite, so is~\(\mathcal{D}\) itself.
\end{proof}

There is another way to construct higher-rank Doplicher--Roberts
algebras using the comultiplication \(\Delta\colon \Cst(G)\to
\Cst(G)\otimes\Cst(G)\), which is defined by \(\lambda_g\mapsto
\lambda_g\otimes\lambda_g\) for the standard
multipliers~\(\lambda_g\) of~\(\Cst(G)\) for \(g\in G\).  This
comultiplication turns~\(\Cst(G)\) into a discrete quantum group.
Our comultiplication is only a morphism, that is, its image is only
in the multiplier algebra of~\(\Cst(G)\).  We know, however, that
\((\Cst(G)\otimes1)\cdot \Delta(\Cst(G)) = \Cst(G)\otimes\Cst(G)\).

Let \(\pi\colon \Cst(G)\to\Mat_n(\C)\) be some finite-dimensional
representation of~\(\Cst(G)\) or, equivalently, of~\(G\).  We get a
morphism
\[
\Cst(G)
\xrightarrow{\Delta} \Cst(G)\otimes\Cst(G)
\xrightarrow{\pi\otimes\Id} \Mat_n\otimes\Cst(G)
\cong \Comp(\Cst(G)^n).
\]
Since \((\Cst(G)\otimes1)\cdot \Delta(\Cst(G)) =
\Cst(G)\otimes\Cst(G)\), this morphism has values in
\(\Comp(\Cst(G)^n)\).  Thus any finite-dimensional
representation~\(\pi\) of~\(\Cst(G)\) induces a \emph{proper}
correspondence~\(\Hilm(\pi)\) from \(\Cst(G)\) to itself.

Let \(\pi_i\colon \Cst(G)\to\Mat_{n_i}(\C)\) for \(i=1,2\) be two
finite-dimensional representations and let \(\pi_1\otimes\pi_2\colon
\Cst(G)\to \Mat_{n_1 n_2}(\C)\) be their tensor product
representation.  The coassociativity of~\(\Delta\) gives an
isomorphism of correspondences
\[
\Hilm(\pi_1) \otimes_{\Cst(G)} \Hilm(\pi_2) \cong \Hilm(\pi_1\otimes\pi_2).
\]
Hence the obvious intertwining unitary \(\pi_1\otimes\pi_2\cong
\pi_2\otimes\pi_1\) gives canonical isomorphisms \(\Hilm(\pi_1)
\otimes_{\Cst(G)} \Hilm(\pi_2) \cong \Hilm(\pi_2) \otimes_{\Cst(G)}
\Hilm(\pi_1)\).  The category of representations of~\(G\) with the
tensor product of representations and the obvious associators and
commuters \(\pi_1\otimes\pi_2\cong \pi_2\otimes\pi_1\) is a
symmetric monoidal category.  Therefore, \(k\)~representations
\(\pi_1,\dotsc,\pi_k\) of~\(G\) give a product system over the
monoid~\((\N^k,+)\) with fibres
\[
\Hilm(m_1,\dotsc,m_k) \defeq \Hilm(\pi_1^{\otimes m_1} \otimes
\dotsb \otimes \pi_k^{\otimes m_k})
\]
and with the canonical isomorphisms
\[
\Hilm(m_1,\dotsc,m_k) \otimes_{\Cst(G)}
\Hilm(m'_1, \dotsc,m'_k)
\cong \Hilm(m_1+m'_1, \dotsc, m_k+m'_k).
\]

We claim that this product system is the same as the one constructed
above if~\(\rho\) is the regular representation.  A first point is
that the \(\Cst\)\nb-algebra of compact intertwiners of~\(\rho\) is
canonically isomorphic to~\(\Cst(G)\), acting by the right regular
representation.  Hence \(\Hilm_0^\rho \cong \Cst(G)\).  Furthermore,
\(\rho\) absorbs every other representation by the Fell absorption
principle: \(\pi\otimes \rho\) is \emph{canonically} isomorphic to a
sum of \(n\) copies of~\(\rho\) if~\(\pi\) has dimension~\(n\); the
intertwiners \(L^2(G,\C^n)\leftrightarrow L^2(G,\C^n)\) are given by
pointwise multiplication with the matrix~\(\pi_g\) at \(g\in G\).
Hence we may identify \(\Cst(G)^n\) canonically with the Hilbert
\(\Hilm_0^\rho\)-module of compact intertwiners
\(\rho\to\pi\otimes\rho\).  These identifications provide an
isomorphism between our product systems because the tensor product
of representations of~\(G\) is induced by the
comultiplication~\(\Delta\).

\section{Actions of Ore monoids on spaces}
\label{sec:monoid_on_space}

Now let~\(X\) be a locally compact, Hausdorff space and let
\(A=\Cont_0(X)\).  Since any automorphism of~\(A\) comes from a
homeomorphism on~\(X\), we may turn an action of a group~\(G\)
on~\(A\) by automorphisms into an action of~\(G\) on the space~\(X\)
and form a transformation groupoid \(G\ltimes X\).  The crossed
product \(G\ltimes \Cont_0(X)\) is canonically isomorphic to the
groupoid \(\Cst\)\nb-algebra of~\(G\ltimes X\).  When is there such
a groupoid model for the Cuntz--Pimsner algebra of a
self-correspondence on~\(A\)?

As a counterexample, consider a Hermitian vector bundle over~\(X\).
It gives a proper self-correspondence from~\(A\) to itself by taking
the Hilbert module of sections with its usual inner product and the
left action by pointwise multiplication.  The resulting Cuntz--Pimsner
algebra is a locally trivial field of \(\Cst\)\nb-algebras over~\(X\)
with Cuntz algebras as fibres.  Such \(\Cst\)\nb-algebras are
classified by D\u{a}d\u{a}rlat in~\cite{Dadarlat:Cstar_vector} in
terms of certain cohomology groups.  Unless the field of
\(\Cst\)\nb-algebras over~\(X\) is particularly simple, it seems to
have no natural groupoid model.

Therefore, we restrict attention to self-correspondences
of~\(\Cont_0(X)\) that are induced by topological correspondences
(see~\cite{Katsura:class_I}).  We define product systems of such
topological correspondences in the obvious fashion, so that they
induce a product system of \(\Cst\)\nb-correspondences.  We will
build a ``transformation groupoid'' for a \emph{proper} product
system of topological correspondences and show that its groupoid
\(\Cst\)\nb-algebra is isomorphic to the Cuntz--Pimsner algebra of
the product system.  Our transformation groupoid construction is
similar in spirit to the boundary path groupoid of
Yeend~\cite{Yeend:Groupoid_models} for a higher-rank topological
graph, that is, for the case \(P=\N^k\) for some \(k\ge1\).  Yeend's
construction, however, depends on special features of~\(\N^k\).  In
contrast, our construction depends on the properness of the product
systems.

A topological correspondence between two spaces \(X\) and~\(Y\) is
given by a third space~\(M\) with two maps \(\rg\colon M\to X\) and
\(\s\colon M\to Y\).  We want to turn this into a
\(\Cst\)\nb-correspondence from~\(\Cont_0(X)\) to~\(\Cont_0(Y)\).
There are two ways to do this.  First, we may assume that~\(\s\) is
a local homeomorphism; this is Katsura's definition of a
\emph{topological correspondence} in~\cite{Katsura:class_I}.  Other
names for this are \emph{continuous graphs}
(see~\cite{Deaconu:Continuous_graphs}) or \emph{polymorphisms} (see
\cites{Arzumanian-Renault:Pseudogroups,
  Cuntz-Vershik:Endomorphisms}).  Secondly, we may add extra
data, namely, a family of measures~\((\lambda_x)_{x\in X}\) on the
fibres of~\(\s\); this is what Muhly and Tomforde call a
\emph{topological quiver} in~\cite{Muhly-Tomforde:Quivers}.  The
family of measures~\((\lambda_x)\) is equivalent to a \emph{transfer
  operator} for~\(\s\) in the notation of
Exel~\cite{Exel:Look_endomorphism}, or a Markov operator in the
notation of~\cite{Ionescu-Muhly-Vega:Markov}.  A topological
correspondence gives a topological quiver when combined with a
suitably normalised family of counting measures on the (discrete)
fibres of~\(\s\).

A topological quiver \((M,\rg,\s,\lambda_x)\) gives a
\(\Cst\)\nb-correspondence~\(\Hilm_{\rg,M,\s}\) over~\(\Cont_0(X)\):
complete~\(\Contc(M)\) with respect to the \(\Cont_0(X)\)\nb-valued
inner product
\[
\braket{\xi_1}{\xi_2}(x) \defeq \int_{\s^{-1}(x)}
\bigl(\conj{\xi_1}\xi_2\bigr)(y) \,\diff\lambda_x(y)
\]
for \(\xi_1,\xi_2\in\Contc(M)\), \(x\in X\); the left and right
module structures are \((f\xi)(m) \defeq f(\rg(m))\xi(m)\) and
\((\xi f)(m) \defeq \xi(m)f(\s(m))\) for all \(m\in M\),
\(f\in\Cont_0(X)\), \(\xi\in\Contc(M)\).  In particular, this
construction applies to topological correspondences, where we always
take the family of counting measures.

\begin{proposition}
  \label{pro:proper_topological_correspondence}
  The \(\Cst\)\nb-correspondence~\(\Hilm_{\rg,M,\s}\) is proper if
  and only if~\(\rg\) is proper and~\(\s\) is a local homeomorphism.
  In that case, the isomorphism class of~\(\Hilm_{\rg,M,\s}\) does
  not depend on~\((\lambda_x)\), so we may always use the family of
  counting measures.  The
  \(\Cst\)\nb-correspondence~\(\Hilm_{\rg,M,\s}\) is full if and
  only if~\(\s\) is surjective.
\end{proposition}

\begin{proof}
  The \(\Cst\)\nb-correspondence~\(\Hilm_{\rg,M,\s}\) is proper if
  and only if \(\varphi^{-1}(\Comp(\Hilm_{\rg,M,\s})) =
  \Cont_0(X)\).  In the notation of
  \cite{Muhly-Tomforde:Quivers}*{Definition 3.14}, all vertices are
  finite emitters.  \cite{Muhly-Tomforde:Quivers}*{Corollary 3.12}
  shows that this happens if and only if~\(\rg\) is proper
  and~\(\s\) is a local homeomorphism.

  Let \((\lambda_x)\) and~\((\lambda'_x)\) be two families of
  measures that make \((M,\rg,\s)\) into a topological quiver.
  Since both \(\lambda_x\) and~\(\lambda'_x\) have the same discrete
  subset~\(\s^{-1}(x)\) as support, they are equivalent, say,
  \(\lambda'_x = f_x\cdot \lambda_x\) for a unique function
  \(f_x\colon s^{-1}(x)\to (0,\infty)\).  The functions~\(f_x\) may
  be pieced together to a function \(f\colon M\to (0,\infty)\).  The
  continuity of \((\lambda_x)\) and~\((\lambda'_x)\) implies
  that~\(f\) is a continuous function.  Hence multiplication
  with~\(\sqrt{f}\) is a unitary operator between the Hilbert
  modules over~\(\Cont_0(X)\) associated to the two families of
  measures.  This unitary also intertwines the left actions, which
  are by multiplication operators.

  It is routine to check that~\(\Hilm_{\rg,M,\s}\) is full if and
  only if~\(\s\) is surjective.
\end{proof}

\begin{definition}
  A topological correspondence is called \emph{proper} if~\(\rg\) is
  proper and~\(\s\) is a local homeomorphism.
\end{definition}

\begin{lemma}
  \label{lem:compose_top_corr}
  Consider two topological correspondences
  \[
  X\xleftarrow{\rg_1} M_1 \xrightarrow{\s_1} X \xleftarrow{\rg_2} M_2
  \xrightarrow{\s_2} X.
  \]
  Define \(M\defeq M_1\times_{\s_1,X,\rg_2} M_2\), \(\rg\colon M\to
  X\), \((m_1,m_2) \mapsto \rg_1(m_1)\), \(\s\colon M\to X\),
  \((m_1,m_2) \mapsto \s_2(m_2)\).  Then
  \[
  \Hilm_{\rg_1,M_1,\s_1} \otimes_{\Cont_0(X)} \Hilm_{\rg_2,M_2,\s_2}
  \cong \Hilm_{\rg,M,\s}.
  \]

  If \(\rg_1\) and~\(\rg_2\) are proper, so is~\(\rg\).  If \(\s_1\)
  and~\(\s_2\) are surjective, so is~\(\s\).
\end{lemma}

\begin{proof}
  The first part is routine to prove and holds even for topological
  quivers, see \cite{Muhly-Tomforde:Quivers}*{Lemmas 6.1--4}.  The
  statements about proper and surjective maps are easy as well; they
  amount to the statement that tensor products of proper or full
  \(\Cst\)\nb-correspondences are again proper or full, respectively.
\end{proof}

Proposition~\ref{pro:proper_topological_correspondence} says that the
\(\Cst\)\nb-correspondence associated to a topological quiver is
proper if and only if we are dealing with a proper topological
correspondence; the family of measures does not matter.  We restrict
attention to proper topological correspondences from now on.

The notion of a ``topological graph algebra'' interprets a
topological correspondence as a ``topological graph,'' where
vertices and (oriented) edges form topological spaces.  This
interpretation, however, fails to elucidate the lack of symmetry
between \(\rg\) and~\(\s\) in the construction of the
\(\Cst\)\nb-correspondence.  Another interpretation is that a
topological correspondence~\((\rg,M,\s)\) is a multi-valued map
from~\(Y\) to~\(X\), where \(\rg(m)\in X\) for \(m\in \s^{-1}(y)\)
are the possible values at \(y\in Y\).  If~\(\s\) is a local
homeomorphism and~\(\rg\) is proper, then the subset of values
\(\rg(\s^{-1}(y))\) of~\(y\) is discrete.  The interpretation as a
multivalued map breaks down, however, if there are different
\(m,m'\in M\) with \(\s(m)=\s(m')\) and \(\rg(m)=\rg(m')\).  We
suggest the following more dynamical interpretation of a (proper)
topological correspondence.

We consider points in~\(M\) as possible \emph{developments} or,
briefly, \emph{stories}.  Each story \(m\in M\) assumes a certain
\emph{initial situation} \(\s(m)\in Y\) and leads to a certain
\emph{ending} \(\rg(m)\in X\).  Several stories may have the same
initial situation and ending.

How does this interpretation account for the assumptions that~\(\s\)
be a local homeomorphism and~\(\rg\) be proper?  That~\(\s\) is a
local homeomorphism means the following: if we modify the initial
situation~\(\s(m)\) of a story~\(m\) a little bit, then there is a
unique story~\(m_x\) close to~\(m\) with initial situation~\(x\).
Roughly speaking, \(m_x\) describes how ``the same'' story would go
in a slightly different initial situation, and fits our intuition of
story-telling.  That~\(\rg\) is proper means that, given a compact
set of possible endings, the set of stories with such an ending is
also compact.  This is a rather technical finiteness condition on
the space of possible stories.  It ensures that the space of
complete histories defined below is locally compact.

\begin{definition}
  \label{def:action_on_pace}
  Let~\(P\) be a monoid.  An \emph{action of~\(P\) on~\(X\) by
    proper topological correspondences} consists of the following
  data:
  \begin{itemize}
  \item proper topological correspondences~\((M_p,\rg_p,\s_p)\)
    from~\(X\) to~\(X\) for \(p\in P\setminus\{1\}\);
  \item homeomorphisms \(\sigma_{p,q}\colon M_{p q} \to
    M_p\times_{\s_p,X,\rg_q} M_q\) for \(p,q\in P\setminus\{1\}\).
  \end{itemize}
  Let \(M_1=X\) and \(\rg_1=\s_1=\Id_X\), and let \(\sigma_{p,1}\)
  and~\(\sigma_{1,q}\) be the canonical homeomorphisms \(M_p \cong
  M_p\times_{\s_p,X,\Id_X} X\) and \(M_q \cong
  X\times_{\Id_X,X,\rg_q} M_q\) for \(p,q\in P\).  For an action
  of~\(P\), we require the diagram
  \begin{equation}
    \label{eq:associative_top_corr}
    \begin{tikzpicture}[baseline=(current bounding box.west)]
      \matrix (m) [cd,column sep=6em] {
        M_p\times_X M_q\times_X M_t&
        M_{p q}\times_X M_t\\
        M_p\times_X M_{q t}&
        M_{p q t}\\
      };
      \draw[cdar] (m-1-2) -- node[swap] {\(\sigma_{p,q}\times_X \Id_{M_t}\)} (m-1-1);
      \draw[cdar] (m-2-1) -- node {\(\Id_{M_p}\times_X\sigma_{q,t}\)} (m-1-1);
      \draw[cdar] (m-2-2) -- node[swap] {\(\sigma_{p q,t}\)} (m-1-2);
      \draw[cdar] (m-2-2) -- node {\(\sigma_{p,q t}\)} (m-2-1);
    \end{tikzpicture}
  \end{equation}
  to commute for all \(p,q,t\in P\setminus\{1\}\) (since \(p q = 1\)
  or \(q t=1\) is possible, we have to define \((M_1,\s_1,\rg_1)\),
  \(\sigma_{1,q}\) and~\(\sigma_{p,1}\) for this condition to make
  sense).  This diagram commutes automatically if \(p=1\), \(q=1\)
  or \(t=1\), so our assumption implies that it commutes for all
  \(p,q,t\in P\).
\end{definition}

\begin{example}
  \label{exa:higher-rank_graph}
  An action of~\(\N^k\) on a countable discrete set~\(X\) by proper
  topological correspondences is equivalent to a row-finite
  rank-\(k\) graph by~\cite{Fowler-Sims:Product_Artin}.  The
  Cuntz--Pimsner algebra that we shall attach to this data is
  \emph{not} always the higher-rank graph \(\Cst\)\nb-algebra,
  however, because we do not incorporate Katsura's modification of
  the Cuntz--Pimsner algebra into our definition.  See
  also~\cite{Raeburn-Sims:Product_graphs}.
\end{example}

We fix an action of~\(P\) on~\(X\) by proper topological
correspondences as above.  The proper topological
correspondences~\((M_p,\rg_p,\s_p)\) induce proper
\(\Cst\)\nb-correspondences~\(\Hilm_p\) from \(\Cont_0(X)\) to
itself for \(p\in P\setminus\{1\}\), and we let
\(\Hilm_1\defeq\Cont_0(X)\).  The homeomorphisms~\(\sigma_{p,q}\)
induce isomorphisms of \(\Cst\)\nb-correspondences
\[
\mu_{p,q}\colon \Hilm_p \otimes_{\Cont_0(X)} \Hilm_q \to \Hilm_{p q}
\]
for \(p,q\in P\setminus\{1\}\) by
Lemma~\ref{lem:compose_top_corr}, and we let \(\mu_{1,q}\)
and~\(\mu_{p,1}\) be the canonical isomorphisms.  The
diagram~\eqref{eq:associative_top_corr} ensures the associativity of
these multiplication maps~\(\mu_{p,q}\) for all \(p,q,t\in
P\setminus\{1\}\) (even if \(p q = 1\) or \(q t=1\)); associativity is
automatic if \(p=1\), \(q=1\) or \(t=1\).  So an action of~\(P\)
on~\(X\) by topological correspondences induces a proper product
system over~\(P\) with unit fibre~\(\Cont_0(X)\), as expected.

The defining property of the fibre product means that \(\sigma_{p,q}
= (\rg_{p,q},\s_{p,q})\) for two continuous maps
\[
\rg_{p,q}\colon M_{p q} \to M_p,\qquad
\s_{p,q}\colon M_{p q} \to M_q
\]
with \(\s_p\circ \rg_{p,q} = \rg_q\circ \s_{p,q}\).
Since~\(\sigma_{p,1}\) and~\(\sigma_{q,1}\) are the canonical maps,
\[
\s_{p,1} = \s_p,\quad
\rg_{p,1} = \Id_{M_p},\qquad
\s_{1,q} = \Id_{M_q},\quad
\rg_{1,q} = \rg_q
\]
for all \(p,q\in P\).  The associativity
condition~\eqref{eq:associative_top_corr} is equivalent to
\begin{equation}
  \label{eq:associative_top_corr_2}
  \rg_{p,q}\circ \rg_{p q,t} = \rg_{p,q t},\quad
  \s_{p,q}\circ \rg_{p q,t} = \rg_{q,t}\circ \s_{p,q t},\quad
  \s_{q,t}\circ \s_{p,q t} = \s_{p q,t}.
\end{equation}

\begin{lemma}
  \label{lem:proper_surjective_localhomeo_inherited}
  The maps~\(\rg_{p,q}\) are proper and the maps~\(\s_{p,q}\) are
  local homeomorphisms.  If all~\(\s_p\) are surjective, then so are
  the maps~\(\s_{p,q}\).
\end{lemma}

\begin{proof}
  The map~\(\rg_{p,q}\) is the composite of the
  homeomorphism~\(\sigma_{p,q}\) and the coordinate projection \(M_p
  \times_{\s_p,X,\rg_q} M_q \to M_p\).  This coordinate projection
  is proper if~\(\rg_q\) is proper because properness is hereditary
  under this type of fibre products.  Similarly, the
  map~\(\s_{p,q}\) is the composite of the
  homeomorphism~\(\sigma_{p,q}\) and the coordinate projection \(M_p
  \times_{\s_p,X,\rg_q} M_q \to M_q\); the latter inherits the
  property of being surjective or a local homeomorphism
  from~\(\s_p\).
\end{proof}

We interpret elements of~\(P\) as a (multi-dimensional) kind of
\emph{time}, and elements of~\(M_p\) as \emph{stories of length \(p\in
  P\)}; a story \(m\in M_p\) \emph{starts} in the
situation~\(\s_p(m)\) and \emph{ends} in~\(\rg_p(m)\).  The maps
\(\rg_{p,q}\colon M_{p q}\to M_p\) and \(\s_{p,q}\colon M_{p q}\to
M_q\) cut a story~\(m\) of length~\(p q\) into two stories of length
\(p\) and~\(q\): its \emph{ending} \(m_1=\rg_{p,q}(m)\in M_p\) and its
\emph{beginning} \(m_2=\s_{p,q}(m)\in M_q\).  These satisfy
\(\s_p(m_1)=\rg_q(m_2)\), that is, the story~\(m_1\) starts
where~\(m_2\) ends.  The inverse of~\(\sigma_{p,q}\) combines a pair
\(m_1\in M_p\), \(m_2\in M_q\) of stories of lengths \(p\) and~\(q\)
to a story \(m_1\circ m_2\) of length~\(p q\), provided~\(m_1\) starts
where~\(m_2\) ends.  The assumption that~\(\sigma_{p,q}\) be a
homeomorphism says that \(m\in M_{pq}\) and \((m_1,m_2)\in M_p\times
M_q\) with \(\s_p(m_1)=\rg_q(m_2)\) determine each other uniquely and
continuously.

The length \(1\in P\) is the neutral element, so nothing can happen
in time~\(1\), and adding a story of length~\(1\) before or after
another story does nothing.  This means that \(M_1=X\) and that
\(\sigma_{p,1}\) and~\(\sigma_{1,q}\) are the canonical maps.  The
associativity conditions~\eqref{eq:associative_top_corr_2} say that
the two ways of cutting a story of length~\(p q t\) into three pieces
of length \(p\), \(q\) and~\(t\) give the same results.

If~\(P\) is a free monoid on~\(n\) generators (which, however, is
not Ore), then the situation above may be interpreted as describing
a game where the players may do~\(n\) different things in each time
interval.  If, say, the player has the three options~\(a,b,c\), then
\(p=baac\) means a time interval of length~\(4\) in which the player
first does~\(c\), then twice~\(a\), then~\(b\).  If the game was in
situation \(x\in X\) initially, then the points in
\(\s_p^{-1}(x)\subseteq M_p\) are the possible game developments in
this length\nb-\(4\) time period, provided the player's actions are
\(baac\).  And \(\rg_p(m)\) for \(m\in\s_p^{-1}(x)\) is the
situation after this time period.  If~\(\s_p^{-1}(x)\) has more than
one point, then the game contains randomness.  It makes sense to
quantify this randomness by a transfer operator with
\(\sum_{\s_p(m)=y} \mu_p(m)=1\) for all \(x\in X\),
where~\(\mu_p(m)\) is the probability that the game develops as in
story~\(m\), given the initial situation~\(y\).  We do not add such
probabilities to our setup because they are irrelevant for us by
Proposition~\ref{pro:proper_topological_correspondence}.

A relation in the monoid~\(P\) means that certain actions of the
player always and automatically have the same effect on the game.
For instance, if~\(P\) is the free \emph{Abelian} monoid~\(\N^n\)
on~\(n\) generators, then the order in which the player does various
things does not matter.  I know no game with this property; so the
interpretation through games works best for free monoids.

There are three simple cases of actions by proper
topological correspondences:
\begin{enumerate}
\item \(M_p=X\) and \(\s_p=\Id_X\) for all \(x\in X\), \(p\in P\);
  that is, a situation \(x\in X\) \emph{determines its future}
  uniquely;
\item \(M_p=X\) and \(\rg_p=\Id_X\) for all \(x\in X\), \(p\in P\);
  that is, a situation \(x\in X\) \emph{determines its past}
  uniquely;
\item \(M_p\) is arbitrary, but \(\s_p=\rg_p\) for all \(p\in P\);
  that is, \emph{the situation never changes}; then \(\s_p=\rg_p\)
  must be both proper and a local homeomorphism; equivalently, it is a
  finite covering map.
\end{enumerate}

From now on, we assume that~\(P\) is a right Ore monoid.  In this
case, the Cuntz--Pimsner algebra of the product system
\((\Hilm_p,\mu_{p,q})\) over~\(P\) is described more concretely in
Section~\ref{sec:Ore_monoids}.  We are going to identify this
Cuntz--Pimsner algebra with the groupoid \(\Cst\)\nb-algebra of an
étale, locally compact groupoid~\(H\).

We first describe the object space~\(H^0\) of this groupoid.  The
first associativity condition in~\eqref{eq:associative_top_corr_2}
says that the spaces~\(M_p\) for \(p\in P\) and the continuous
maps~\(\rg_{p,q}\) for \(p,q\in P\) form a projective system of
locally compact spaces indexed by the directed category~\(\Cat_P\).
We let
\[
H^0 \defeq \varprojlim_{\Cat_P} {}(M_p,\rg_{p,q}).
\]
Thus a point in~\(H^0\) consists of \(m_p\in M_p\) for all \(p\in
P\) that satisfy \(\rg_{p,q}(m_{p q})=m_p\) for all \(p,q\in P\).
In other words, the \(m_p\) are stories that are consistent in the
sense that~\(m_p\) is the ending of~\(m_{pq}\) for each \(p,q\in
P\).  We call a point in~\(H^0\) a \emph{complete history} and think
of~\(m_p\) as describing what happened in the last length-\(p\) time
period.

\begin{lemma}
  \label{lem:H_locally_compact}
  The space~\(H^0\) is locally compact and Hausdorff.  The
  projection maps \(\pi_q\colon H^0\to M_q\), \((m_p)_{p\in
    P}\mapsto m_q\), are proper for all \(q\in P\).
\end{lemma}

\begin{proof}
  Fix \((m_p)_{p\in P}\in H^0\) and let \(K\subset X=M_1\) be a
  compact neighbourhood of~\(m_1\).  The preimage of~\(K\)
  in~\(H^0\) is the subset of all \((m'_p)_{p\in P}\) with \(m'_1\in
  K\) and hence \(\rg_p(m'_p)\in K\) for all \(p\in P\).  Since all
  the maps~\(\rg_p\) are proper, the subsets
  \(\rg_p^{-1}(K)\subseteq M_p\) are compact.  Hence so is the
  product \(L\defeq \prod_{p\in I} \rg_p^{-1}(K)\) by Tychonov's
  Theorem.  Thus the map \(\pi_1\colon H^0\to X\) is proper.  The
  same argument shows that all the maps~\(\pi_q\) are proper.
  Since~\(L\) is also a compact neighbourhood of~\((m_p)_{p\in P}\)
  in~\(H^0\), the space~\(H^0\) is locally compact.  If \((m'_p)\neq
  (m_p)\), then there is \(p\in P\) with \(m'_p\neq m_p\).  There
  are open neighbourhoods in~\(M_p\) that separate \(m'_p\)
  and~\(m_p\).  These yield open neighbourhoods in~\(H^0\) that
  separate \((m'_p)\) and~\((m_p)\), so~\(H^0\) is Hausdorff.
\end{proof}

Given a complete history \((m_p)_{p\in P}\) and \(t\in P\), we may
forget what happened in the last time period of length~\(t\); this
gives another complete history, defined formally by \(m'_p \defeq
\s_{t,p}(m_{tp})\) for \(p\in P\); the second condition
in~\eqref{eq:associative_top_corr_2} implies \(\rg_{p,q}(m'_{p
  q})=m'_p\) for all \(p,t\in P\), that is, \((m'_p)_{p\in P}\) is
again a complete history as expected.  Thus \((m_p)\mapsto (m'_p)\)
defines a map \(\tilde{\s}_t\colon H^0\to H^0\).

\begin{lemma}
  \label{lem:tilde_s_local_homeo}
  Let \(t\in P\).  The map \((\pi_t,\tilde{s}_t)\colon H^0 \to
  M_t\times_{\s_t,X,\pi_1} H^0\) is a homeomorphism, and
  \(\tilde{\s}_t\colon H^0\to H^0\) is a local homeomorphism.
  If~\(\s_t\) is surjective, so is~\(\tilde{\s}_t\).
  If~\(\s_t\) is a homeomorphism, so is~\(\tilde{\s}_t\).
\end{lemma}

\begin{proof}
  Let \((m_p)_{p\in P}\in H^0\).  Then \(\pi_t(m_p) = m_t\) and
  \(\tilde{\s}_t((m_p)) = (\s_{t,p}(m_{tp}))_{p\in P}\).  Since
  \(m_{t p} = \rg_{t,p}(m_{tp})\cdot \s_{t,p}(m_{tp}) = m_t\cdot
  \tilde{\s}_t((m_p))_p\), we have \(\pi_1\circ\tilde{\s}_t =
  \s_t\circ \pi_t\), that is, the image of \((\pi_t,\tilde{s}_t)\)
  is contained in the fibre product \(M_t\times_{\s_t,X,\pi_1}
  H^0\).  Since the map \((\pi_t,\tilde{s}_t)\) is clearly
  continuous, we must prove that it is a bijection with a continuous
  inverse.  So we let \((m'_p)\in H^0\) be a complete history and
  let \(m_t\in M_t\) be a length-\(t\) story with \(m'_1=\s_t(m_t)
  \in X\).  We must show that there is a unique complete
  history~\((m_p)\) with given~\(m_t\) and with
  \(\tilde{\s}_t((m_p)) = (m'_p)\), which depends continuously on
  \((m'_p)\) and~\(m_t\).

  First assume that~\((m_p)_{p\in P}\) as above has been found.  Let
  \(q\in P\).  Then we may write \(qu=tp\) for some \(u,p\in P\)
  because~\(P\) is a right Ore monoid.  The story~\(m_{tp}\) is the
  concatenation \(m_t\circ m'_p\) of \(\rg_{t,p}(m_{pt})=m_t\) and
  \(\s_{t,p}(m_{tp})=m'_p\), which exists because \(\rg_p(m'_p) =
  m'_1=\s_t(m_t)\).  Thus \(m_q = \rg_{q,u}(m_{qu}) =
  \rg_{q,u}(m_t\circ m'_p)\).  So there is at most one possible
  solution~\((m_p)_{p\in P}\), and it depends continuously on
  \((m'_p)\) and~\(m_t\).  We must show that the length\nb-\(q\)
  stories \(m_q \defeq \rg_{q,u}(m_t\circ m'_p)\) with \(u,p\in P\)
  as above form a complete history, that is, \(\rg_{q,v}(m_{qv}) =
  m_q\) for all \(q,v\in P\).

  We have \(m_{qv} = \rg_{q v,u_2}(m_t\circ m'_{p_2})\) for some
  \(u_2,p_2\in P\) with \(q v u_2 = t p_2\).  Since~\(P\) is a right
  Ore monoid, there are \(u_3,u_4\in P\) with \(v u_2 u_3 = u u_4\).
  Then \(t p_2 u_3 = q v u_2 u_3 = q u u_4 = t p u_4\).  Since~\(P\)
  is a right Ore monoid, there is \(u_5\in P\) with \(p_2 u_3 u_5 =
  p u_4 u_5\).  To simplify notation, we replace \((u_3,u_4)\) by
  \((u_3 u_5,u_4 u_5)\); thus \(p_2 u_3 = p u_4\).

  Since \(q u u_4 = t p u_4\), we could also use~\((u u_4,p u_4)\)
  instead of~\((u,p)\) to define~\(m_q\).  The associativity
  conditions in~\eqref{eq:associative_top_corr_2} show that this
  gives the same result:
  \begin{multline*}
    \rg_{q,u u_4}(m_t\circ m'_{p u_4})
    = \rg_{q,u} \rg_{q u,u_4}(m_t\circ m'_{p u_4})
    \\= \rg_{q,u} \rg_{q u,u_4}(m_t\circ m'_p \circ
    \s_{p,u_4}(m'_{p u_4}))
    = \rg_{q,u}(m_t\circ m'_p).
  \end{multline*}
  Similarly, we get the same result for~\(m_{q v}\) if we use \((u_2
  u_3,p_2 u_3)\) instead of~\((u_2,p_2)\).  Thus we may assume that
  \(p_2=p\) and \(u=v u_2\).  Then
  \[
  \rg_{q,v}(m_{qv})
  = \rg_{q,v}(\rg_{qv, u_2}(m_t\circ m'_p))
  = \rg_{q,v u_2}(m_t\circ m'_p)
  = \rg_{q,u}(m_t\circ m'_p) = m_q.
  \]
  This finishes the proof that \((\pi_t,\tilde{s}_t)\) is a
  homeomorphism.  This homeomorphism transforms the
  map~\(\tilde{\s}_t\) into the second coordinate projection
  \(M_t\times_{\s_t,X,\pi_1} H^0\to H^0\), which is a (local)
  homeomorphism if~\(\s_t\) is one, and surjective if~\(\s_t\) is.
\end{proof}

First forgetting the last length\nb-\(t\) time period and then the
last length\nb-\(u\) time period gives the same result as directly
forgetting the last time period of length~\(tu\).  That is,
\(\tilde{\s}_u\circ \tilde{\s}_t = \tilde{\s}_{tu}\) for all
\(t,u\in P\).  Formally, this follows from the third condition
in~\eqref{eq:associative_top_corr_2}.  Thus the monoid~\(P^\op\)
acts on~\(H^0\) by local homeomorphisms.

Why do we get the opposite monoid here?  The maps
\(\tilde{\s}_t\colon H^0\to H^0\) and \(\tilde{\rg}_t\defeq
\Id_{H^0}\) form an action of~\(P\) by topological correspondences
with the extra property that any situation determines its past
uniquely: a ``situation'' in~\(H^0\) is a complete history, which
simply contains its past.  Thus we still have an action by
topological correspondences, but one where the
maps~\(\tilde{\rg}_p\) are all identity maps, so that we may forget
about them.  This gives an action of the opposite monoid~\(P^\op\)
by local homeomorphisms because of the direction of the
maps~\(\tilde{\s}_p\).

\begin{example}
  \label{exa:action_proper_maps}
  Suppose that we start with an action of~\(P\)
  on~\(X\)
  by proper maps~\(\rg_p\)
  and let~\(\s_p\)
  be identity maps; that is, every situation determines its future.
  Then the maps~\(\tilde{\s}_t\)
  are homeomorphisms by Lemma~\ref{lem:tilde_s_local_homeo}.  Thus our
  action of~\(P^\op\)
  on~\(H^0\)
  extends to the group completion~\(G\)
  of~\(P^\op\).
  The groupoid model we are going to construct is the transformation
  groupoid of this group action.
\end{example}

\begin{definition}
  \label{def:covering_semidirect}
  The \emph{transformation groupoid} \(H\defeq P^\op\ltimes H^0\)
  associated to the \(P^\op\)\nb-action~\((\tilde{\s}_p)\)
  on~\(H^0\) by local homeomorphisms has object space~\(H^0\), arrow
  set
  \[
  H^1
  \defeq \{(x,g,y)\in H^0\times G\times H^0 \mid
  \exists p_1, p_2\in P,\ g=p_1 p_2^{-1},\ \tilde{\s}_{p_1}(x) =
  \tilde{\s}_{p_2}(y)\},
  \]
  range and source maps \(\rg(x,g,y)\defeq x\), \(\s(x,g,y)\defeq
  y\), and multiplication
  \[
  (x_1,g_1,y_1)\cdot (x_2,g_2,y_2) \defeq (x_1,g_1 g_2,y_2)
  \]
  if \(y_1=x_2\).  The unit on \(x\in H^0\) is \((x,1,x)\), the
  inverse of \((x,g,y)\) is \((y,g^{-1},x)\).

  We describe the topology on~\(H^1\).  For \(p_1, p_1\in P\), let
  \[
  H^1_{p_1,p_2} \defeq H^0\times_{\tilde{\s}_{p_1},H^0,\tilde{\s}_{p_2}} H^0
  \defeq \{(x,y)\in H^0\times H^0\mid \tilde{\s}_{p_1}(x) =
  \tilde{\s}_{p_2}(y)\},
  \]
  the fibre product of the diagram
  \(H^0\xrightarrow{\tilde{\s}_{p_1}}
  H^0\xleftarrow{\tilde{\s}_{p_2}} H^0\).  We give
  each~\(H^1_{p_1,p_2}\) the subspace topology from the product
  \(H^0\times H^0\), and \(\bigsqcup_{p_1,p_2\in P} H^1_{p_1,p_2}\)
  the disjoint union topology.  We map \(H^1_{p_1,p_2}\to H^1\) by
  \((x,y)\mapsto (x,p_1 p_2^{-1},y)\in H^1\).  This gives a
  surjection \(\bigsqcup H^1_{p_1,p_2} \to H^1\).  We give~\(H^1\) the
  quotient topology from this map.
\end{definition}

To verify that the transformation groupoid has desirable properties,
we rewrite it using filtered colimits.  Let \(H^1_g \defeq
\{(x,g,y)\in H^1\}\) for \(g\in G\); so \(H^1 = \bigsqcup_{g\in G}
H^1_g\).  We describe~\(H^1_g\) for fixed \(g\in G\) as a colimit
over~\(\Cat_P^g\) (see Definition~\ref{def:Cat_P_g}), which is a
filtered category by Lemma~\ref{lem:Ore_category_g}.  If
\(p_1q=p_3\), \(p_2q=p_4\), then \(\tilde{\s}_{p_1}(x) =
\tilde{\s}_{p_2}(y)\) implies \(\tilde{\s}_{p_3}(x) =
\tilde{\s}_{p_4}(y)\), so \(H^1_{p_1,p_2}\subseteq H^1_{p_3,p_4}
\subseteq H^0\times H^0\).  Since right multiplication with~\(q\) is
locally injective, any \((x,y)\in H^1_{p_3,p_4}\) has a
neighbourhood in \(H^0\times H^0\) so that for \((x',y')\) in this
neighbourhood, \(\tilde{\s}_{p_1}(x')\neq \tilde{\s}_{p_2}(y')\)
implies \(\tilde{\s}_{p_1q}(x') \neq \tilde{\s}_{p_2q}(y')\).  Thus
the subset~\(H^1_{p_1,p_2}\) is relatively open
in~\(H^1_{p_3,p_4}\); so the spaces \(H^1_{p_1,p_2}\) for
\((p_1,p_2)\in P^2\) and the inclusion maps \(H^1_{p_1,p_2} \to
H^1_{p_3,p_4}\) form a diagram of subsets of~\(H^0\times H^0\) with
open inclusion maps.

\begin{lemma}
  \label{lem:H_union}
  If \(p_1 p_2^{-1}= p_3 p_4^{-1}\), then there are \(p_5,p_6\in P\)
  with \(p_5 p_6^{-1} = p_1 p_2^{-1}\) and so that both
  \(H^1_{p_1,p_2}\) and~\(H^1_{p_3,p_4}\) are open subsets
  of~\(H^1_{p_5,p_6}\); hence the subspace topologies from
  \(H^1_{p_1,p_2}\) and~\(H^1_{p_3,p_4}\) coincide on \(H^1_{p_1,p_2}\cap
  H^1_{p_3,p_4}\), and this subset is open both in \(H^1_{p_1,p_2}\) and
  in~\(H^1_{p_3,p_4}\).  Each~\(H^1_{p_1,p_2}\) is open in~\(H^1_g\),
  and the topology on~\(H^1_g\) restricts to the given topology on
  each~\(H^1_{p_1,p_2}\).
\end{lemma}

\begin{proof}
  Since~\(\Cat_P^g\) is filtered, there is an object~\((p_5,p_6)\)
  that dominates both \((p_1,p_2)\) and~\((p_3,p_4)\).  This has all
  required properties.  Thus all the embeddings \(H^1_{p_1,p_2} \to
  H^1_{p_3,p_4}\) are open.  This implies that the quotient topology
  on \(H^1_g = \bigcup_{(p_1,p_2)\in R_g} H^1_{p_1,p_2}\) from
  \(\bigsqcup_{(p_1,p_2)\in R_g} H^1_{p_1,p_2}\) restricts to the
  given topology on each subset~\(H^1_{p_1,p_2}\).
\end{proof}

Thus the subsets~\(H^1_{p_1,p_2}\) for \(p_1,p_2\in P\) form an
\emph{open} covering of~\(H^1\), and the topology on~\(H^1\)
restricts to the usual topology on each~\(H^1_{p_1,p_2}\).  In the
following, we identify~\(H^1_{p_1,p_2}\) with its image in~\(H^1\),
which is an open subset.

\begin{proposition}
  \label{pro:H_etale_locally_compact}
  The groupoid~\(H\) is étale, locally compact and Hausdorff.  The
  decomposition \(H^1 = \bigsqcup_{g\in G} H^1_g\) satisfies
  \(H^1_g\cdot H^1_h\subseteq H^1_{gh}\) and \((H^1_g)^{-1} =
  H^1_{g^{-1}}\).  If the maps~\(\s_p\) for \(p\in P\) are
  surjective, then \(H^1_g\cdot H^1_h=H^1_{gh}\) for all
  \(g,h\in G\).

  The groupoid \(H^1_1\subseteq H^1\) is an increasing union of open
  subgroupoids that are proper and étale and describe equivalence
  relations on~\(H^0\).
\end{proposition}

\begin{proof}
  The space~\(H^0\) is locally compact and Hausdorff by
  Lemma~\ref{lem:H_locally_compact}.  The coordinate projections
  \(H^1_{p_1,p_2}\rightrightarrows H^0\) are \'etale because~\(P\)
  acts by local homeomorphisms.  Since the subsets~\(H^1_{p_1,p_2}\)
  for \(p_1,p_2\in P\) form an open covering of~\(H^1\), the
  coordinate projections \(H^1\rightrightarrows H^0\) are étale.
  Any two points of~\(H^1_g\) are contained in the same Hausdorff,
  locally compact, open subset~\(H^1_{p_1,p_2}\) for suitable
  \(p_1,p_2\), so they may be separated by open subsets
  of~\(H^1_g\); since the subsets~\(H^1_g\) are open, it is also
  possible to separate points in \(H^1_g\) and~\(H^1_h\) for \(g\neq
  h\).  Thus~\(H^1\) is Hausdorff.

  If the maps~\(\s_p\) for \(p\in P\) are surjective, then so are
  the maps~\(\tilde{\s}_p\) for \(p\in P\) by
  Lemma~\ref{lem:tilde_\s_local_homeo}.  Now let \((x,g h,y)\in
  H^1_{g h}\).  Hence there are \(p_1,p_2\in P\) with \(g h = p_1
  p_2^{-1}\) and \(\tilde{\s}_{p_1}(x) = \tilde{\s}_{p_2}(y)\).
  Write \(g = p_3 p_4^{-1}\), \(h=p_5 p_6^{-1}\).  Then also \(g=p_3
  q (p_4 q)^{-1}\) and \(h=(p_5 t)(p_6 t)^{-1}\) for all \(q,t\in
  P\).  The Ore condition~\ref{enum:Ore1} allows us to choose \(q\)
  and~\(t\) such that \(p_4 q = p_5 t\).  Hence we may assume
  without loss of generality that \(p_4=p_5\).  Then
  \[
  p_1 p_2^{-1}
  = g h = p_3 p_4^{-1} p_5 p_6^{-1} = p_3 p_6^{-1}.
  \]
  If
  \(\tilde{\s}_{p_1}(x) = \tilde{\s}_{p_2}(y)\), then also
  \(\tilde{\s}_{p_1 t}(x) = \tilde{\s}_{p_2 t}(y)\) for any \(t\in
  P\), so we may rewrite \(g h = (p_1 t) (p_2 t)^{-1}\).  We may
  also replace \((p_3,p_4,p_5,p_6)\) by \((p_3 q,p_4 q,p_5 q,p_6
  q)\) for any \(q\in P\).  Choosing \(q\) and~\(t\) by
  condition~\ref{enum:Ore1}, we may achieve \(p_3 q = p_1 t\), \(p_6
  q = p_2 t\), by the definition of the group~\(G\).  Hence we find
  \(p_1,p_2,p_3\in P\) with \(g h = p_1 p_2^{-1}\), \(g=p_1
  p_3^{-1}\), \(h=p_3 p_2^{-1}\) and \(\tilde{\s}_{p_1}(x) =
  \tilde{\s}_{p_2}(y)\).  Since~\(\tilde{\s}_{p_3}\) is surjective,
  we may choose \(z\in H^0\) with \(\tilde{\s}_{p_3}(z) =
  \tilde{\s}_{p_1}(x) = \tilde{\s}_{p_2}(y)\).  Then \((x,g,z)\in
  H^1_g\) and \((z,h,y)\in H^1_h\) satisfy \((x,g,z)\cdot
  (z,h,y)=(x,g h,y)\).

  The open subgroupoid~\(H^1_1\) is defined as the union of
  \(H^1_{p_1,p_2}\) with \((p_1,p_2)\in R_1\).  Since~\(\Cat_P\) is
  cofinal in~\(\Cat_P^1\) by Lemma~\ref{lem:Ore_category_g},
  \(H^1_1= \bigcup_{p\in P} H^1_{p,p}\).  Here~\(H^1_{p,p}\) is the
  set of all \((x,y)\in H^0\times H^0\) with \(\tilde{\s}_{p}(x) =
  \tilde{\s}_{p}(y)\), and it carries the subspace topology
  from~\(H^0\times H^0\).  So~\(H^1_{p,p}\) is a proper equivalence
  relation on~\(H^0\), and~\(H^1_1\) is the union of these open
  subgroupoids.
\end{proof}

If~\(P\) is countable, then we may choose a cofinal sequence
in~\(\Cat^1_P\) and write~\(H^1_1\) as an increasing union of
a sequence of proper étale equivalence relations.  Hence~\(H^1_1\)
is an \emph{approximately proper equivalence relation} in the
notation of~\cite{Renault:Radon-Nikodym_AP}.  These are called
\emph{hyperfinite relations} in~\cite{Kumjian:Diagonals}.  We allow
ourselves to call~\(H^1_1\) approximately proper also if~\(P\) is
uncountable, replacing a sequence of proper (finite) open
subrelations by a directed set of such subrelations.

If the maps~\(\s_p\) for \(p\in P\) are surjective, then the
subsets~\(H^1_g\) for \(g\in G\) form a \(G\)\nb-grading in the
notation of~\cite{Buss-Meyer:Actions_groupoids}.  This is equivalent
to an action of~\(G\) on the groupoid~\(H^1_1\) by Morita
equivalences with transformation groupoid~\(H\).  Thus we may think
of~\(H\) as the transformation groupoid associated to an action
of~\(G\) on the noncommutative orbit space~\(H^1/H^1_1\).  Points in
this orbit space are equivalence classes of complete histories,
where two complete histories are identified if they coincide in the
distant past, that is, \(\tilde{\s}_{p}(x) = \tilde{\s}_{p}(y)\) for
some \(p\in P\).  The group~\(G\) acts on this by ``time
translations.''

If the maps~\(\s_p\) are not surjective, then the \(G\)\nb-action
on~\(H^1\) is only a \emph{partial} action because time translations
\(x\mapsto p x\) into the future are not everywhere defined.  A
partial \(G\)\nb-action is the same as an action of a certain inverse
semigroup associated to~\(G\), see~\cite{Exel:Partial_actions}.

\begin{definition}
  \label{def:impossible_situation}
  A situation \(x\in X\) is (historically) \emph{possible} if \(x\in
  \rg_p(M_p)\) for all \(p\in P\).
\end{definition}

Let \(X'\subseteq X\) be the subset of possible situations.  We have
\(X'=X\) if and only if all the maps~\(\rg_p\) are surjective.  Let
\(M'_p = \s_p^{-1}(X')\) and let \(\rg'_p\) and~\(\s'_p\) be the
restrictions of \(\rg_p\) and~\(\s_p\) to~\(M'_p\).  Any situation
that occurs at some point in a complete history is possible, so we
have \(m_p\in M'_p\) for any \((m_p)_{p\in P}\in H^0\).  Conversely,
a situation that is possible is the endpoint~\(m_1\) of
some complete history~\((m_p)_{p\in P}\) because the maps~\(\rg_p\)
are proper (Lemma~\ref{lem:H_locally_compact}) and a projective
limit of non-empty compact spaces is non-empty.  Thus
\(\pi_p^{-1}(H^0) = M_p'\), \(\rg'_p(M'_p)=X'\) for all \(p\in P\),
and \(\s'_p(M'_p)\subseteq X'\) by associativity: if a situation has
a possible past, then it is itself possible because we may
concatenate stories.  We still have isomorphisms \(M'_p\times_{X'}
M'_q \cong M'_{p q}\), so restricting to the possible
situations gives a new action by topological correspondences.  By
construction, both systems \((X,M,\s_p,\rg_p,\sigma_{p,q})\) and
\((X',M'_p,\s'_p,\rg'_p,\sigma'_{p,q})\) have the same complete
histories and thus the same transformation groupoid~\(H\).

\begin{lemma}
  \label{lem:unique_truncation}
  Let \(q,a_1,a_2\in P\) satisfy \(q a_1 = q a_2\).  Then
  \(\rg_{q,a_1}|_{M'_{q a_1}}=\rg_{q,a_2}|_{M'_{q a_2}}\).
\end{lemma}

\begin{proof}
  Condition~\ref{enum:Ore2} gives us \(b\in P\) with \(a_1 b= a_2 b\).
  The associativity property~\eqref{eq:associative_top_corr_2} of the
  maps~\(\rg_{p,p}\) gives \(\rg_{q,a_i}\circ \rg_{q a_i,b} = \rg_{q,
    a_i b}\).  Hence \(\rg_{q,a_1}\) and \(\rg_{q,a_2}\) coincide on
  the range of \(\rg_{q,a_1 b} = \rg_{q,a_2 b}\).  The
  subspace~\(M'_{q a_1}\) is contained in that range.
\end{proof}

Write \(p\ge q\) for \(p,q\in P\) if there is \(a\in P\) with \(p=q
a\).  Lemma~\ref{lem:unique_truncation} shows that after restricting
to the possible situations, the truncation map \(\rg_{q,a}\colon
M'_p\to M'_q\) for \(p\ge q\) does not depend on the choice of~\(a\).

\begin{theorem}
  \label{the:groupoid_model}
  The groupoid \(\Cst\)\nb-algebra~\(\Cst(H)\) is canonically
  isomorphic to the Cuntz--Pimsner algebra of the product
  system~\((\Hilm_p)_{p\in P}\) over~\(P\) described above.
\end{theorem}

\begin{proof}
  The proof has two parts.  In the first part, we show that the
  original action on~\(X\) has the same colimit as the induced
  action on~\(H^0\).  So we are reduced to the special case of
  actions by correspondences of the special form where the maps
  \(\rg_p\colon M_p \to X\) are all identity maps.  In the second
  part, we do this case by hand.

  Let~\(D\) be a \(\Cst\)\nb-algebra.  A transformation from the
  product system~\((\Hilm_p)_{p\in P}\) to~\(D\) consists of a
  correspondence~\(\Hilm[F]\) to~\(D\) together with isomorphisms of
  correspondences \(V_p\colon \Hilm_p\otimes_{\Cont_0(X)} \Hilm[F]
  \cong \Hilm[F]\).  The left action of~\(\Cont_0(X)\)
  on~\(\Hilm_p\) extends to an action of~\(\Cont_0(M_p)\) by
  pointwise multiplication.  Thus we get a canonical left action
  of~\(\Cont_0(M_p)\) on \(\Hilm_p\otimes_{\Cont_0(X)}
  \Hilm[F]\cong\Hilm[F]\), where we use use the isomorphism~\(V_p\).
  Thus \(\Cont_0(M_p)\) acts on~\(\Hilm[F]\) in a canonical way for
  each \(p\in P\).

  If \(p,q\in P\), then the isomorphism \(V_{p q}\colon \Hilm_{pq}
  \otimes_{\Cont_0(X)} \Hilm[F] \cong \Hilm[F]\) is equal to the
  composite isomorphism where we first identify \(\Hilm_{pq} \cong
  \Hilm_p \otimes_{\Cont_0(X)} \Hilm_q\) and then apply \(V_q\)
  and~\(V_p\).  As a consequence, the action of \(\Cont_0(M_p)\)
  on~\(\Hilm[F]\) is the composite of the action of \(\Cont_0(M_{p
    q})\) on~\(\Hilm[F]\) and the \Star{}homomorphism
  \(\rg_{p,q}\colon \Cont_0(M_p) \to \Cont_0(M_{p q})\).  So
  the left actions of \(\Cont_0(M_p)\) fit together to an action of
  the inductive limit
  \[
  \varinjlim_{\Cat_P} (\Cont_0(M_p),\rg_{p,q}^*)
  \cong \Cont_0\Bigl(\varprojlim_{\Cat_P} (M_p,\rg_{p,q})\Bigr)
  = \Cont_0(H^0).
  \]
  Thus~\(\Hilm[F]\) carries a nondegenerate \Star{}representation
  of~\(\Cont_0(H^0)\), turning it into a
  correspondence~\(\tilde{\Hilm[F]}\) from~\(\Cont_0(H^0)\)
  to~\(D\).

  Now let \(\tilde{\Hilm}_p \defeq \Hilm_p \otimes_{\Cont_0(X)}
  \Cont_0(H^0)\).  Then
  \[
  \tilde{\Hilm}_p \otimes_{\Cont_0(H^0)} \tilde{\Hilm[F]}
  \cong \Hilm_p \otimes_{\Cont_0(X)} \Cont_0(H^0)
  \otimes_{\Cont_0(H^0)} \tilde{\Hilm[F]}
  \cong \Hilm_p \otimes_{\Cont_0(X)} \Hilm[F]
  \cong \Hilm[F],
  \]
  where the last isomorphism is~\(V_p\) and all other isomorphisms
  are trivial.  By definition, the
  correspondence~\(\tilde{\Hilm}_p\) from~\(\Cont_0(X)\)
  to~\(\Cont_0(H^0)\) is obtained from a topological correspondence
  as well, namely, we replace~\(M_p\) by \(\tilde{M}_p =
  M_p\times_{\s_p,X,\pi_1} H^0\) and use the map \((m,\omega)\mapsto
  \rg_p(m)\) as range and the map \((m,\omega)\mapsto \omega\) as
  source map.  Lemma~\ref{lem:tilde_s_local_homeo} describes a
  homeomorphism \(\tilde{M}_p \cong H^0\) such that the second
  coordinate projection becomes \(\tilde{\s}_p\colon H^0\to H^0\).
  We let~\(\tilde{\rg}_p\) be the identity map \(H^0\to H^0\) and
  thus turn~\(\tilde{\Hilm}_p\) into a correspondence
  from~\(\Cont_0(H^0)\) to itself.  Since
  \(\tilde{\s}_p\circ\tilde{\s}_q = \tilde{\s}_{q p}\), the usual
  composition of topological correspondences defines an action
  of~\(P\) on~\(H^0\) by topological correspondences.  Thus the
  associated \(\Cst\)\nb-correspondences~\(\tilde{\Hilm}_p\) form a
  product system over~\(P\) with unit fibre~\(\Cont_0(H^0)\).

  We claim that the isomorphism of Hilbert modules
  \(\tilde{\Hilm}_p\otimes_{\Cont_0(H^0)} \tilde{\Hilm[F]} \cong
  \tilde{\Hilm[F]}\)
  constructed above is an isomorphism of correspondences
  from~\(\Cont_0(H^0)\)
  to~\(D\)
  with this choice of left action of~\(\Cont_0(H^0)\)
  on~\(\tilde{\Hilm}_p\).
  It suffices that \(\Cont_0(M_q)\)
  acts in the same way on both sides for each \(q\in P\).
  The Ore condition~\ref{enum:Ore1} gives \(t,u\in P\)
  with \(p t = q u\).
  Since the action of \(\Cont_0(M_q)\)
  factors through \(\Cont_0(M_{q u})\)
  and \(q u = p t\), we may assume \(q=p t\).

  The left action of~\(\Cont_0(M_{p t})\)
  on \(\tilde{\Hilm}_p\otimes_{\Cont_0(H^0)} \tilde{\Hilm[F]}\)
  is obtained as follows.  First, as Hilbert \(D\)\nb-modules,
  we identify
  \[
  \tilde{\Hilm}_p\otimes_{\Cont_0(H^0)} \tilde{\Hilm[F]}
  \cong
  (\Hilm_p\otimes_{\Cont_0(X)} \Cont_0(M_t)) \otimes_{\Cont_0(M_t)}
  (\Hilm_t \otimes_{\Cont_0(X)} \Hilm[F]).
  \]
  Then we identify \(\Hilm_p\otimes_{\Cont_0(X)} \Cont_0(M_t)\)
  with the \(\Cst\)\nb-correspondence
  from \(\Cont_0(M_{p t})\)
  to~\(\Cont_0(M_t)\)
  associated to the topological correspondence
  \(M_{p t} \cong M_p\times_{\s_p,X,\rg_t} M_t \to M_t\),
  where the range map is the homeomorphism~\(\sigma_{p,t}\)
  and the source map is~\(\pr_2\).
  By Lemma~\ref{lem:compose_top_corr}, the composite of this with the
  \(\Cst\)\nb-correspondence~\(\Hilm_t\)
  from \(\Cont_0(M_t)\)
  to~\(\Cont_0(X)\)
  is the \(\Cst\)\nb-correspondence
  associated to the composite topological correspondence,
  \((M_p\times_{\s_p,X,\rg_t} M_t) \times_{\pr_2,M_t,\rg_t} M_t\).
  This is \(M_p\times_{\s_p,X,\rg_t} M_t\)
  now viewed as a topological correspondence from~\(M_{p t}\)
  to~\(X\).
  As such, it is isomorphic to~\(M_{p t}\).
  Thus we may also get the left action of \(\Cont_0(M_{p t})\)
  on \(\tilde{\Hilm}_p\otimes_{\Cont_0(H^0)} \tilde{\Hilm[F]}\)
  by identifying this in the canonical way with
  \(\Hilm_{p t} \otimes_{\Cont_0(X)} \Hilm[F]\)
  and then acting on the first tensor factor by pointwise
  multiplication.  This, however, is exactly how the left action of
  \(\Cont_0(M_{p t})\)
  on~\(\tilde{\Hilm[F]}\)
  is defined.  Thus the left actions of~\(\Cont_0(M_t)\)
  on \(\tilde{\Hilm}_p\otimes_{\Cont_0(H^0)} \tilde{\Hilm[F]}\)
  and~\(\tilde{\Hilm[F]}\) coincide as expected.

  Since \(\Cont_0(X)\subseteq \Cont_0(H^0)\),
  we may view~\(\Hilm_p\)
  as a subspace of~\(\tilde{\Hilm}_p\).
  The map from the algebraic tensor product
  \(\Hilm_p \odot \Hilm[F]\)
  to \(\tilde{\Hilm}_p \otimes_{\Cont_0(H^0)} \Hilm[F]\)
  has dense range because the target is isomorphic to~\(\Hilm[F]\),
  which is also isomorphic to the correspondence tensor product
  \(\Hilm_p \otimes_{\Cont_0(X)} \Hilm[F]\),
  and there \(\Hilm_p \odot \Hilm[F]\)
  is certainly dense.  Hence the map from
  \(\Hilm_p \odot \Hilm_q \odot \Hilm[F]\)
  to
  \(\tilde{\Hilm}_p \otimes_{\Cont_0(H^0)} \tilde{\Hilm}_q
  \otimes_{\Cont_0(H^0)} \Hilm[F]\)
  also has dense range.  The coherence condition for a transformation
  of product systems~\eqref{def:transformation_product_system} holds
  on this dense subspace by assumption, and hence it holds everywhere.
  Thus our isomorphisms of \(\Cst\)\nb-correspondences
  \(\tilde{\Hilm}_p \otimes_{\Cont_0(H^0)} \Hilm[F] \cong \Hilm[F]\)
  form a transformation as expected.

  Thus a transformation from the product system~\(\Hilm_p\)
  to~\(D\)
  gives a transformation from the product system~\(\tilde{\Hilm}_p\)
  to~\(D\),
  with the same underlying Hilbert module~\(\Hilm[F]\).
  Conversely, a transformation from the product
  system~\(\tilde{\Hilm}_p\)
  to~\(D\)
  gives one from the product system~\(\Hilm_p\)
  to~\(D\)
  because \(\Hilm_p \subseteq \tilde{\Hilm}_p\).
  Since \(\Hilm_p \odot \Hilm[F]\)
  is dense in \(\tilde{\Hilm}_p \otimes_{\Cont_0(H^0)} \Hilm[F]\),
  these two constructions must be inverse to each other.  Summing up,
  we have found a bijection between the transformations from our two
  product systems to~\(D\)
  that does not change the underlying Hilbert module~\(\Hilm[F]\).
  The results in Section~\ref{sec:endos_CP} show that such
  transformations are in bijection with nondegenerate representations
  of the Cuntz--Pimsner algebras of the two product systems
  on~\(\Hilm[F]\),
  respectively.  Having found bijections between the representations
  of both Cuntz--Pimsner algebras on any Hilbert module, we conclude
  by the Yoneda Lemma that they must be isomorphic.

  We have now reduced the general case of an action by topological
  correspondences to the special case of an action with \(M_p=X\)
  and \(\rg_p=\Id\) for all \(p\in P\).  This case is much easier
  because the groupoid model has \(H^0=X\), so we merely take a
  transformation groupoid and do not change the underlying object
  space.

  By definition, \(\Cst(H)\) is the \(\Cst\)\nb-completion of the
  dense \Star{}subalgebra \(\Contc(H^1)\) of compactly supported,
  continuous functions on~\(H^1\), equipped with the usual
  convolution and involution
  \[
  f_1*f_2(h) \defeq \sum_{h_1h_2=h} f_1(h_1) f_2(h_2),\qquad
  f^*(h) \defeq \conj{f(h^{-1})},
  \]
  for \(f_1,f_2,f\in\Contc(H^1)\), \(h\in H^1\) (see
  \cite{Exel:Inverse_combinatorial}*{Section 3}).  Here
  ``\(\Cst\)\nb-completion'' means that we complete in the largest
  \(\Cst\)\nb-seminorm on~\(\Contc(H^1)\).  There is no need to
  assume boundedness for the \(I\)\nb-norm.  First, the argument
  in~\cite{Exel:Inverse_combinatorial} shows that every Hilbert
  space representation and hence every \(\Cst\)\nb-seminorm is
  continuous for the inductive limit topology; secondly,
  \cite{Renault:Representations}*{Corollaire 4.8} shows that such
  representations and \(\Cst\)\nb-seminorms are bounded for the
  \(I\)\nb-norm.

  The disjoint decomposition \(H^1 = \bigsqcup_{g\in G}
  H^1_g\) gives \(\Contc(H^1) = \bigoplus_{g\in G} \Contc(H^1_g)\).
  This is a non-saturated \(G\)\nb-grading, that is,
  \(\Contc(H^1_g)*\Contc(H^1_h) \subseteq \Contc(H^1_{gh})\) and
  \(\Contc(H^1_g)^* = \Contc(H^1_{g^{-1}})\).  This \(G\)\nb-grading
  turns~\(\Cst(H)\) into the section algebra of a Fell bundle
  over~\(G\).  Of course, our proof will show that this Fell bundle
  structure corresponds to the same structure on the Cuntz--Pimsner
  algebra.

  The space~\(H^1_g\) is an increasing union of the open
  subsets~\(H^1_{p_1,p_2}\).  Thus any function in~\(\Contc(H^1_g)\)
  already belongs to~\(\Contc(H^1_{p,q})\) for some \(p,q\in
  P\) with \(p q^{-1}=g\):
  \[
  \Contc(H^1_g) = \bigcup_{p q^{-1}=g} \Contc(H^1_{p,q}).
  \]
  Since \(X=H^1\), we have \(H^1_{p,q} \cong X\times_{\s_p,X,\s_q}
  X\), the set of pairs \((x,y)\) with \(\s_p(x)=\s_q(x)\).  We are
  now going to relate \(\Contc(X\times_{\s_p,X,\s_q} X)\) to the
  space \(\Comp(\Hilm_q,\Hilm_p)\) in the description of the
  Cuntz--Pimsner algebra in the proof of
  Theorem~\ref{the:Ore_colimit_Fell}.

  Given a function \(k\in \Contc(X\times_{\s_p,X,\s_q} X)\), we
  define
  \[
  T_k\colon \Hilm_q\to\Hilm_p,\qquad
  (T_k\xi)(m_1) \defeq \sum_{\s_p(m_1)=\s_q(m_2)} k(m_1,m_2) \xi(m_2);
  \]
  these sums are uniformly finite for~\(m_1\) in a compact subset
  because \(\s_p\) and~\(\s_q\) are local homeomorphisms and the
  support of~\(k\) is compact.  The operator~\(T_k\) is a rank-one
  operator if \(k(m_1,m_2) = k_1(m_1)\cdot k_2(m_2)\) with \(k_1\in
  \Contc(X)\), \(k_2\in\Contc(X)\).  Since functions~\(k\) of this
  form are dense in \(\Contc(X\times_{\s_p,X,\s_q} X)\) and the map
  \(k\mapsto T_k\) is continuous, we have \(T_k\in
  \Comp(\Hilm_q,\Hilm_p)\) for all \(k\in
  \Contc(X\times_{\s_p,X,\s_q} X)\).  Since compactly supported
  functions are dense in \(\Hilm_p\) and~\(\Hilm_q\), operators of
  the form~\(T_k\) for \(k\in \Contc(X\times_{\s_p,X,\s_q} X)\) are
  dense in \(\Comp(\Hilm_q,\Hilm_p)\).  If \(T_k=0\), then \(k=0\).

  We may express the product and involution on compact operators
  through kernel functions: if \(k\in \Contc(X\times_{\s_p,X,\s_q}
  X)\) and \(l\in \Contc(X\times_{\s_q,X,\s_t} X)\), then \(T_k\circ
  T_l\) has the kernel \((m_1,m_2)\mapsto \sum_{\s_q(m) = \s_p(m_1)
    = \s_t(m_2)} k(m_1,m)l(m,m_2)\), and the adjoint~\(T_k^*\) has
  the kernel \((m_1,m_2)\mapsto \overline{k(m_2,m_1)}\).  These
  formulas correspond to the convolution and involution
  in~\(\Contc(H^1)\).  Therefore, the map \(k\mapsto T_k\) is an
  injective \Star{}homomorphism with dense range from \(\Contc(H^1)
  = \sum_{p,q\in P} \Contc(H^1_{p,q})\) to the \Star{}algebra
  \(\sum_{g\in G} \CP_g\) of compactly supported sections of the
  Fell bundle~\((\CP_g)_{g\in G}\).

  It remains to show that this \Star{}homomorphism extends to an
  isomorphism between the \(\Cst\)\nb-completions.  It suffices to
  prove that the restriction of any \Star{}representation of
  \(\Contc(H^1)\) to~\(\Contc(H^1_g)\) is bounded with respect to
  the norm of~\(\CP_g\).  Since \(\norm{\xi}^2 =
  \norm{\xi^*\xi}_{\CP_1}\) for all \(\xi\in \CP_g\), this holds for
  all~\(g\) once it holds for \(g=1\).  Thus it remains to show that
  the unit fibre~\(\CP_1\) of the Cuntz--Pimsner algebra is the
  \(\Cst\)\nb-completion of~\(\Contc(H^1_1)\) for the
  subgroupoid~\(H^1_1\).

  If \(p\in P\), then the subset \(H^1_{p,p} = X\times_{\s_p,X,\s_p}
  X\) of~\(H^1_1\) is the groupoid describing the equivalence
  relation~\(\sim_p\) induced by the map~\(\s_p\).  This equivalence
  relation is proper, that is, the map \(X\to X/{\sim_p}\) is
  proper, since~\(\s_p\) is a local homeomorphism.  Hence the
  \(\Cst\)\nb-algebra of the groupoid~\(H^1_{p,p}\) is
  Morita--Rieffel equivalent to \(\Cont_0(\s_p(X))\).  The Hilbert
  bimodule constructed in the proof of this Morita--Rieffel
  equivalent in~\cite{Muhly-Renault-Williams:Equivalence} is exactly
  our correspondence~\(\Hilm_p\).  Hence \(\Cst(H^1_{p,p}) =
  \Comp(\Hilm_p)\).  The \(\Cst\)\nb-algebras \(\Cst(H^1)\)
  and~\(\CP_1\) are the colimits of the diagrams of
  \(\Cst\)\nb-algebras \(\Cst(H^1_{p,p})\) and~\(\Comp(\Hilm_p)\)
  over the filtered category~\(\Cat_P\).  Hence they are also
  canonically isomorphic.
\end{proof}

\section{Some relations to previous work}
\label{sec:groupoid_model_examples}

The construction of groupoid models above is very general and
contains many known constructions.  We discuss some of them in this
section.

First let \(P=(\N,+)\).
An action of~\(\N\)
by topological correspondences is already determined by the single
topological correspondence \((M_1,\rg_1,\s_1)\),
where~\(\s_1\)
must be a local homeomorphism (\(1\in\N\)
is not the unit element here, there is a conflict with our usual
multiplicative notation).  This topological correspondence is the same
as a topological graph.  We assume~\(\rg_1\)
to be proper to get a proper topological correspondence; then the
composite correspondences~\(M_n\)
for \(n\in\N\)
are automatically proper.  We also assume~\(\rg_1\)
to be surjective; equivalently, all maps~\(\rg_n\)
are surjective and \(X=X'\).
What we are dealing with is a row-finite topological graph without
sources, which we simply call \emph{regular}.  The space~\(X\)
is its space of vertices, and~\(M_1\)
is its space of edges, with \(m\in M_1\)
giving an edge from~\(\s_1(x)\)
to~\(\rg_1(x)\).
A point in~\(M_n\)
is a path in the topological graph of length~\(n\),
and the maps \(\s_n\)
and~\(\rg_n\)
send such a path to its initial and final point.  The homeomorphism
\(M_n \times_{\s_n,X,\rg_m} M_m \cong M_{n+m}\)
builds a path of length \(n+m\)
by concatenating two paths of length \(n\) and~\(m\).
If the space~\(X\)
of vertices is discrete, then so is the space of edges~\(M_1\)
because~\(\s_1\)
is a local homeomorphism.  Thus the case where~\(X\)
is discrete gives the ordinary graphs among the topological graphs.

For a regular topological graph, the topological graph
\(\Cst\)\nb-algebra
of Katsura \cites{Katsura:class_I, Katsura:class_II,
  Katsura:class_III} is, by definition, the same absolute
Cuntz--Pimsner algebra that we study.  If there are sources, that is,
\(\rg_1\)~is
not surjective, then neither our (absolute) Cuntz--Pimsner algebra nor
its groupoid model see a difference between the topological graph
\((X,M_1,\rg_1,\s_1)\)
and its restriction \((X',M'_1,\rg'_1,\s'_1)\),
which now has surjective~\(\rg'_1\).
Thus we get the topological graph \(\Cst\)\nb-algebra
of the regular graph \((X',M'_1,\rg'_1,\s'_1)\);
this is quite different from Katsura's \(\Cst\)\nb-algebra
for the original topological graph.

Groupoid models played an important role in the definition of
(discrete) graph \(\Cst\)\nb-algebras
by Kumjian, Pask, Raeburn and Renault
in~\cite{Kumjian-Pask-Raeburn-Renault:Graphs}.  They are, however, not
used by Katsura to develop the theory of topological graph
\(\Cst\)\nb-algebras.
A pretty general construction of a groupoid model for topological
graph \(\Cst\)\nb-algebras
is due to Deaconu~\cite{Deaconu:Continuous_graphs}, under the extra
assumption that both maps \(\s_1\)
and~\(\rg_1\)
be surjective local homeomorphisms and both spaces \(X\)
and~\(M_1\)
be compact.  These assumptions are removed by
Yeend~\cite{Yeend:Topological-higher-rank-graphs}, who constructs
groupoid models for the more general class of \emph{higher-rank}
topological graphs.  His construction works for all rank-\(1\)
graphs, that is, he may allow~\(\rg_1\)
to be any map.  The construction simplifies, however, if~\(\rg_1\)
is proper and surjective.

Since points in~\(M_n\)
are paths of length~\(n\),
our complete histories in~\(H^0\)
are the same as infinite paths in the topological graph.  The map
\(\pi_n\colon H^0\to M_n\)
gives the initial segment of a path of length~\(n\).
The local homeomorphism~\(\tilde{\s}_n\)
truncates an infinite path by throwing away the initial segment of
length~\(n\).
The groupoid~\(H\)
is \(\Z\)-graded,
\[
H= \bigsqcup_{n\in\Z} H_n,
\]
and the unit fibre \(H_0\subseteq H\)
describes the equivalence relation of \emph{tail equivalence}: we
identify two infinite paths if they eventually become equal (without
shift).  The whole groupoid~\(H\)
combines tail equivalence with the shift map on infinite paths: an
arrow in~\(H_n\)
means that two infinite paths become eventually equal if we also shift
one of them by~\(n\) steps.

The construction above is exactly how the usual groupoid model for a
regular topological graph is constructed.  Thus our groupoid~\(H\)
is the same as Yeend's groupoid model for the \(\Cst\)\nb-algebra
of a regular topological graph, which in turn generalises the groupoid
models in \cites{Deaconu:Continuous_graphs,
  Kumjian-Pask-Raeburn-Renault:Graphs}.  In particular, we get the
familiar groupoid model for the \(\Cst\)\nb-algebra
of a row-finite graph without sources.  In the irregular case, Yeend
adds certain finite paths to~\(H^0\),
and he defines the topology on the resulting space of ``boundary
paths'' carefully to get a locally compact space.

Next consider \(P=(\N^k,+)\)
for some \(k\ge2\).
An action of the Ore monoid~\(\N^k\)
by topological correspondences is the same, almost by definition, as a
\emph{topological rank-\(k\)
  graph}.  The case where the underlying space~\(X\)
is discrete corresponds to an ordinary rank-\(k\)
graph.  Topological higher-rank graphs are introduced by Yeend, who
also describes a groupoid model for them
in~\cite{Yeend:Groupoid_models}.  He requires the source maps to be
local homeomorphisms, but does not require the range maps to be
proper; instead, he assumes a weaker condition called ``compact
alignment,'' which may be formulated for lattice-ordered semigroups.
He constructs groupoid models for the Toeplitz \(\Cst\)\nb-algebra
and the relative Cuntz--Pimsner algebra of the product system
over~\(\N^k\)
associated to a compactly aligned topological higher-rank graph.  The
relative and absolute Cuntz--Pimsner algebras agree if and only if all
range maps~\(\rg_p\)
are surjective (``no sources'').  We call a topological higher-rank
graph \emph{regular} if the maps~\(\rg_p\)
are surjective and proper for all \(p\in\N^k\).
In the regular case, the groupoid model constructed by
Yeend~\cite{Yeend:Groupoid_models} is the same one that we have
constructed above: the boundary paths that form the object space for
Yeend's groupoid model are the same as our complete histories by
\cite{Yeend:Groupoid_models}*{Lemma 6.6}.  In the irregular case, the
object space of Yeend's groupoid combines infinite paths with certain
finite and partially infinite paths.  It is unclear how to carry this
over to actions of Ore monoids.

How about groupoid models for actions of semigroups other
than~\(\N^k\)?  Here we are only aware of constructions in
particular cases.  We discuss two general situations.  First, if the
action is by local homeomorphisms, then we are already very close to
an inverse semigroup action, which is easily translated to an
\'etale groupoid (see \cites{Exel:Inverse_combinatorial,
  Exel-Renault:Semigroups_interaction, Paterson:Groupoids}).
Secondly, the construction of a semigroup \(\Cst\)\nb-algebra by Xin
Li in \cites{Cuntz-Echterhoff-Li:K-theory, Li:Semigroup_amenability}
may also be based on an action of the semigroup by topological
correspondences.

\subsection{Semigroups of partial local homeomorphisms}
\label{sec:partial_local_homeo}

Let~\(P\) be a monoid, and let~\(P\) act on a locally compact
space~\(X\) by \emph{partial local homeomorphisms}, that is, by
topological correspondences of the special form
\begin{equation}
  \label{eq:partial_local_homeo}
  X \xleftarrow{\textup{inclusion}} U_p\xrightarrow{\alpha_p} X,
\end{equation}
where \(U_p\subseteq X\)
is an open subset and~\(\alpha_p\)
is a local homeomorphism.  These topological correspondences are only
proper if the domains~\(U_p\)
are also closed.  But the following construction of a groupoid does
not need this assumption, and neither does it require the monoid~\(P\)
to be Ore.  We do not claim, however, that the groupoid
\(\Cst\)\nb-algebra
of the resulting groupoid is isomorphic to the Cuntz--Pimsner algebra
of the product system over~\(P\)
associated to our action: our proof only gives this if~\(P\)
is an Ore monoid and the correspondences are proper, that is, the
subsets~\(U_p\) are clopen.

First we make the multiplication maps explicit.  The fibre product
\(U_p \times_{\s_p,X,\rg_q} U_q\) consists of pairs \((x,y)\) with
\(x\in U_p\), \(y\in U_q\), and \(\alpha_p(x)=y\), and the range and
source maps on \(U_p \times_{\s_p,X,\rg_q} U_q\) take \((x,y)\) to
\(x\) and~\(\alpha_q(y)\), respectively.  Since \(y=\alpha_p(x)\),
the map \((x,y)\mapsto x\) identifies \(U_p \times_{\s_p,X,\rg_q}
U_q\) with \(U_p\cap \alpha_p^{-1}(U_q)\); under this identification,
the range and source maps become the inclusion map and the map
\(x\mapsto \alpha_q(\alpha_p(x))\), respectively.  Thus we must have
\begin{equation}
  \label{eq:compose_partial_local_homeo}
  U_{pq} = U_p\cap \alpha_p^{-1}(U_q),\qquad
  \alpha_{pq} = \alpha_q\circ \alpha_p.
\end{equation}
Conversely, these conditions give a unique isomorphism of
topological correspondences \(U_p \times_{\s_p,X,\rg_q} U_q \cong
U_{pq}\).  These isomorphisms automatically verify the associativity
condition required for an action of~\(P\) by topological
correspondences.  Thus an action of~\(P\) by topological
correspondences of the special form~\eqref{eq:partial_local_homeo}
is the same as a homomorphism from~\(P^\op\) to the monoid of
partial local homeomorphisms of~\(X\), with the composition of
partial local homeomorphisms defined
in~\eqref{eq:compose_partial_local_homeo}.

If \(V\subseteq U_p\) is such that~\(\alpha_p|_V\) is injective,
then~\(\alpha_p|_V\) is a partial homeomorphism on~\(X\).
Since~\(\alpha_p\) is a local homeomorphism, any point in~\(U_p\)
has a neighbourhood~\(V\) on which~\(\alpha_p|_V\) is injective, so
that these partial homeomorphisms contained in~\(\alpha_p\)
cover~\(\alpha_p\).  Hence we do not lose any information if we
replace~\(\alpha_p\) by the set of all partial
homeomorphisms~\(\alpha_p|_V\) for \(V\subseteq U_p\) such
that~\(\alpha_p|_V\) is injective.  These partial
homeomorphisms~\(\alpha_p|_V\) form a semigroup because
\(\alpha_q|_V \circ \alpha_p|_W = \alpha_{pq}|_{W\cap
  \alpha_p^{-1}(V)}\) and \(W\cap \alpha_p^{-1}(V)\)
is an open subset of~\(U_{pq}\) on which~\(\alpha_{pq}\) is injective
by~\eqref{eq:compose_partial_local_homeo}.

We let~\(S\)
be the \emph{inverse} semigroup of partial homeomorphisms generated by
these partial homeomorphisms~\(\alpha_p|_V\).
This inverse semigroup acts on~\(X\)
by construction, and this action has an associated transformation
groupoid \(X\rtimes S\),
also called groupoid of germs; see \cites{Exel:Inverse_combinatorial,
  Paterson:Groupoids}.  This groupoid is often the same as the
groupoid model constructed in Section~\ref{sec:monoid_on_space}, but
there are some ``trivial'' counterexamples.  The issue is how to
define the germ relation.  To always get the groupoid constructed in
Section~\ref{sec:monoid_on_space}, we do the following.

First, we assume now that~\(P\)
is an Ore monoid with group completion~\(G\).
Since the construction in Section~\ref{sec:monoid_on_space} is only
for actions by proper correspondences, we also require that the
domains~\(U_p\)
are clopen.  We let~\(S_0\)
be the free inverse semigroup on symbols~\((p,V)\)
for~\(\alpha_p|_V\)
as above.  This comes with a canonical homomorphism
\(\gamma\colon S\to G\)
by mapping \((p,V)\mapsto p^{-1}\),
and with a canonical action on~\(X\)
by mapping \((p,V)\mapsto \alpha_p|_V\).
Let~\(S\)
be the quotient of~\(S_0\)
by the kernel of this map.  That is, we consider two elements
of~\(S_0\)
equivalent if they give the same element of~\(G\)
and the same partial homeomorphism on~\(X\).
Now we take the groupoid of germs of the action of~\(S\)
on~\(X\)
with the germ relation from~\cite{Exel:Inverse_combinatorial}, that
is, two elements \(s,t\in S\)
have the same germ at \(x\in X\)
if there is an idempotent~\(e\)
in~\(S\) defined at~\(x\) so that \(s e = t e\).

\begin{lemma}
  Assume that~\(P\) is an Ore monoid and the subsets~\(U_p\) are
  clopen.  Then the groupoid \(X\rtimes S\) above is
  canonically isomorphic to the groupoid model in
  Definition~\textup{\ref{def:covering_semidirect}}.
\end{lemma}

\begin{proof}
  We map the free inverse semigroup~\(S_0\)
  above to the inverse semigroup of bisections of the groupoid~\(H\)
  in Definition~\ref{def:covering_semidirect} by mapping~\((p,V)\)
  to \(H^1_{1,p}\cap \s^{-1}(V)\);
  this is easily seen to be a bisection of~\(H\)
  that acts on \(X=H^0\)
  by the partial homeomorphism~\(\alpha_p|_V\)
  and has degree~\(p^{-1}\).
  Thus the action of~\(S_0\)
  on~\(X\)
  and the homomorphism~\(\gamma\)
  both factor through the inverse semigroup of bisections of~\(H\),
  where~\(\gamma\)
  maps bisections contained in~\(H^1_g\)
  to~\(g\)
  and where the action of~\(H\)
  on~\(X\) is used to let bisections act on~\(X\).

  By construction, \(s\in S_0\)
  is annihilated by~\(\gamma\)
  if and only if the corresponding bisection in~\(H\)
  is contained in~\(H_1\).
  This groupoid comes from an equivalence relation, so a bisection is
  trivial if and only if it acts trivially on~\(X\).
  Hence \(s\in S_0\)
  becomes idempotent in~\(S\)
  if and only if it is mapped to an idempotent bisection of~\(H\).
  This shows that the quotient~\(S\)
  of~\(S_0\)
  is exactly the image of~\(S_0\)
  in the inverse semigroup of bisections of~\(H\).

  The groupoid~\(H\)
  is covered by bisections that belong to~\(S\)
  and are of the form
  \[
  (H^1_{1,p}\cap \s^{-1}(V))\circ (H^1_{1,q}\cap \s^{-1}(W))^{-1}
  \]
  for \(p,q\in P\)
  and \(V\subseteq U_p\),
  \(W\subseteq U_q\)
  such that \(\alpha_p|_V\)
  and~\(\alpha_q|_W\)
  are injective.  It is not clear whether these bisections already
  form an inverse semigroup; but at least, since they cover the arrow
  space of~\(H\),
  any product of such bisections is again covered by bisections
  of~\(H\)
  of this special form.  Therefore, the inverse semigroup~\(S\)
  and the inverse semigroup of bisections of~\(H\)
  have the same germ groupoids attached to their actions on~\(X\),
  that is, \(H\cong X\rtimes S\).
\end{proof}

The examples considered by Exel and Renault
in~\cite{Exel-Renault:Semigroups_interaction} are actions
of~\((\N^k,+)\) by (globally defined) local homeomorphisms, so they
certainly fit into our framework.  This is remarkable
because~\cite{Exel-Renault:Semigroups_interaction} also contains
counterexamples where Exel's interaction group approach to
non-invertible dynamical systems does not apply.  This leads Exel
and Renault to speculate that something should go wrong in these
counterexamples.

Exel defines \emph{interactions} in~\cite{Exel:Interactions} as a
way to describe dynamical systems that are non-deterministic in both
past and future time directions.  A local homeomorphism with a
transfer operator is a particular example of an interaction, and
the dynamics generated by a single local homeomorphism may be
studied quite well using interactions.  Exel proposed the concept of
an \emph{interaction group} in~\cite{Exel:New_look} in order to
extend this to more general dynamical systems.
In~\cite{Exel-Renault:Semigroups_interaction}, Exel and Renault give
rather simple examples of commuting local homeomorphisms \(S,T\colon
X\to X\) that cannot be embedded in an interaction group
over~\(\Z^2\).  We are going to discuss this, assuming both \(S,T\) to
be surjective because this happens in the counterexamples
in~\cite{Exel-Renault:Semigroups_interaction}.

The problem is the following.  The local homeomorphism~\(S\)
generates an equivalence relation on~\(X\) by \(x\sim_S y\) if
\(S(x)=S(y)\), and similarly for~\(T\).  If there is an interaction
group, then these relations \(\sim_S\) and~\(\sim_T\) must commute,
that is, there is \(z\in X\) with \(x\sim_S z\sim_T y\) if and only
if there is \(w\in X\) with \(x\sim_T w\sim_S y\) (see
\cite{Exel-Renault:Semigroups_interaction}*{Proposition 14.1}).
There are, however, commuting endomorphisms \(S,T\) for which the
relations \(\sim_S\) and~\(\sim_T\) do not commute.  So such \(S,T\)
cannot be part of an interaction group.  Why is this no problem for
our groupoid model?

Since our topological correspondences are already local
homeomorphisms, our groupoid~\(H\) has object space \(H^0=X\).  The
group completion of~\((\N^2,+)\) is~\((\Z^2,+)\), so the
groupoid~\(H\) is \(\Z^2\)-graded, \(H=\bigsqcup_{g\in\Z^2} H_g\).
A point in~\(H_g\) is given by \((x,y)\in X\) and
\(n_1,n_2,m_1,m_2\in\N\) with \(S^{n_1}T^{n_2}(x) = S^{m_1}
T^{m_2}(y)\) and \((m_1,m_2) - (n_1,n_2) = g\), and this is an arrow
\(x\leftarrow y\).  Thus the range and source maps identify~\(H_g\)
with the union of the subsets
\[
H_{n_1,n_2,m_1,m_2} = \{(x,y)\in X\times X \mid
S^{n_1}T^{n_2}(x) = S^{m_1} T^{m_2}(y)\}.
\]
We treat these as relations on~\(X\).  The
subsets~\(H_{n_1,n_2,m_1,m_2}\) are closed in~\(X\times X\).  If
\(k_1,k_2\in\N\) then~\(H_{n_1,n_2,m_1,m_2}\) is both open and
closed in \(H_{n_1+k_1,n_2+k_2,m_1+k_1,m_2+k_2}\) because the
map~\(S^{k_1}T^{k_2}\) is a local homeomorphism, hence locally
injective.  The relation~\(H_{n_1,n_2,m_1,m_2}\) may also be
interpreted as the graph of the multi-valued map \(S^{-n_1}T^{-n_2}
S^{m_1} T^{m_2}\) on~\(X\).

What happens when we compose our relations?  We have \((x,y)\in
H_{k_1,k_2,l_1,l_2} \circ H_{l_1,l_2,m_1,m_2}\) if and only if there
is \(z\in X\) with \(S^{k_1}T^{k_2}(x) = S^{l_1} T^{l_2}(z) =
S^{m_1} T^{m_2}(y)\).  Since \(S^{l_1} T^{l_2}\) is surjective by
assumption, the point~\(z\) can always be found if
\(S^{k_1}T^{k_2}(x) = S^{m_1} T^{m_2}(y)\), so
\begin{equation}
  \label{eq:multiply_in_H_over_N2}
  H_{k_1,k_2,l_1,l_2} \circ H_{l_1,l_2,m_1,m_2} = H_{k_1,k_2,m_1,m_2}.
\end{equation}
Exel and Renault say that \(S\) and~\(T\) \emph{star-commute} if for
all \(x,y\in X\) with \(T(x)=S(y)\) there is a unique \(z\in X\)
with \(S(z)=x\) and \(T(z)=y\).  Under this assumption, they
construct an interaction group containing \(S\) and~\(T\).  In our
notation, \(S\) and~\(T\) star-commute if and only if \(H_{0,1,1,0}
= H_{0,0,1,0} \circ H_{0,1,0,0}\); the inclusion \(H_{0,1,1,0}
\supseteq H_{0,0,1,0} \circ H_{0,1,0,0}\) is trivial.  The
relation~\(H_{0,1,1,0}\) describes the multi-valued map
\(T^{-1}\circ S\), whereas the relation \(H_{0,0,1,0} \circ
H_{0,1,0,0}\) describes the multi-valued map \(S\circ T^{-1}\).  So
\(S\) and~\(T\) \emph{star-commute} if and only if \(T^{-1}\circ S =
S\circ T^{-1}\).

If this fails, there is no good way to define a topological
correspondence or an interaction for the element \((1,-1)\in \Z^2\).
The difference between these two relations is, however, always small
in the sense that
\[
T\circ (S\circ T^{-1})
= S\circ (T\circ T^{-1})
= S
= T\circ (T^{-1}\circ S).
\]
Since \(x,y\in X\) with \(T(x)=T(y)\)
are equivalent in our groupoid~\(H\),
the difference between the relations \(S\circ T^{-1}\)
and \(T^{-1}\circ S\)
does not matter once we take the whole groupoid into account.

Let us also examine this issue from the point of view of the
Cuntz--Pimsner algebra of the resulting product
system~\((\Hilm_p)_{p\in\N^2}\).  The
subspace~\(H_{n_1,n_2,m_1,m_2}\) of the groupoid corresponds to the
subspace \(\Comp(\Hilm_{(m_1,m_2)},\Hilm_{(n_1,n_2)})\) of the
Cuntz--Pimsner algebra, compare the proof of
Theorem~\ref{the:groupoid_model}.  Since we assume the maps \(S,T\)
to be surjective, the correspondences~\(\Hilm_p\) are full for all
\(p\in\N^2\).  Hence
\[
\Comp(\Hilm_{(m_1,m_2)},\Hilm_{(n_1,n_2)}) \cdot
\Comp(\Hilm_{(k_1,k_2)},\Hilm_{(m_1,m_2)})
= \Comp(\Hilm_{(k_1,k_2)},\Hilm_{(n_1,n_2)})
\]
for all \(k_1,k_2,m_1,m_2,n_1,n_2\in \N\).  This corresponds
to~\eqref{eq:multiply_in_H_over_N2}.

The zero fibre of the Cuntz--Pimsner algebra is the inductive limit
of the \(\Cst\)\nb-subalgebras \(\Comp(\Hilm_{n_1,n_2})\) for
\(n_1,n_2\to\infty\).  In particular, \(\Comp(\Hilm_{0,1})\) and
\(\Comp(\Hilm_{1,0})\) are contained in \(\Comp(\Hilm_{1,1})\).
Although \(\Hilm_{1,1} = \Hilm_{1,0}\otimes_{\Cont_0(X)} \Hilm_{0,1}
\cong \Hilm_{0,1}\otimes_{\Cont_0(X)} \Hilm_{1,0}\), we cannot
expect in general that \(\Comp(\Hilm_{0,1}) \cdot
\Comp(\Hilm_{1,0})\) is equal to \(\Comp(\Hilm_{1,1})\): this goes
wrong if \(S\) and~\(T\) do not star-commute.  It may also happen
that \(\Comp(\Hilm_{0,1}) \cdot \Comp(\Hilm_{1,0}) \neq
\Comp(\Hilm_{0,1}) \cdot \Comp(\Hilm_{1,0})\) or, equivalently, that
\(\Comp(\Hilm_{0,1}) \cdot \Comp(\Hilm_{1,0})\) is not a
\(\Cst\)\nb-algebra.

Summing up, we have seen that the commutative semigroup~\(\N^2\) may
well generate some noncommutative phenomena both on the groupoid and
\(\Cst\)\nb-algebra level.  \emph{So the reason why
Cuntz--Pimsner algebras for proper product systems over~\(\N^2\) are
tractable is not that~\(\N^2\) is commutative---it is that~\(\N^2\)
satisfies Ore conditions.}

To make this clearer, consider now an arbitrary semigroup~\(P\).
The Cuntz--Pimsner algebra of a proper product
system~\((\Hilm_p)_{p\in P}\) over~\(P\) must contain
\(\Comp(\Hilm_p)\) for all \(p\in P\).  Given \(p,q\in P\), we
therefore need a \(\Cst\)\nb-algebra containing both
\(\Comp(\Hilm_p)\) and \(\Comp(\Hilm_q)\).  If~\(P\) is an Ore
monoid, then there is \(t\in P\) with \(t\ge p,q\), and then
\(\Comp(\Hilm_p)\) and \(\Comp(\Hilm_q)\) are both contained in
\(\Comp(\Hilm_t)\).  It is irrelevant for the construction how much
of \(\Comp(\Hilm_t)\) is generated by
\(\Comp(\Hilm_p)\) and \(\Comp(\Hilm_q)\) or whether
\(\Comp(\Hilm_p)\cdot\Comp(\Hilm_q)\) is a \(\Cst\)\nb-algebra.
Indeed, we would not even ask for any relationship unless~\(t\) were
chosen minimal, \(t=p\vee q\), which only exists in lattice-ordered
semigroups.  The examples
in~\cite{Exel-Renault:Semigroups_interaction} show that for
commutative~\(P\) the interaction group approach of Exel is trying
implicitly to combine \(\Comp(\Hilm_p)\) and \(\Comp(\Hilm_q)\) in
such a way that \(\Comp(\Hilm_p)\cdot\Comp(\Hilm_q) =
\Comp(\Hilm_q)\cdot\Comp(\Hilm_p)\) is again a \(\Cst\)\nb-algebra.
This led Exel in~\cite{Exel:Blend_Alloys} to study when \(A\cdot B =
B\cdot A\) for two \(\Cst\)\nb-subalgebras \(A\) and~\(B\) of
another \(\Cst\)\nb-algebra.

Finally, there is one thing where an interaction group helps.  If we
only have an action of~\(\N^2\),
then we can only restrict it to submonoids of~\(\N^2\).
An interaction group on~\(\Z^2\)
may also be restricted to subgroups of~\(\Z^2\);
in the above notation, an interaction group gives well-defined
topological correspondences \(S^{n_1} T^{n_2}\)
for all \(n_1,n_2\in\Z\).
Without it, we only have this if \(n_1\)
and~\(n_2\)
have the same sign.  Examples of such restrictions are the
polymorphisms in~\cite{Cuntz-Vershik:Endomorphisms}.  These may,
however, be written directly as topological graphs.

\subsection{Semigroup C*-algebras}
\label{sec:semigroup_Cstar}

How to define the \(\Cst\)\nb-algebra of a monoid~\(P\)?  A
satisfactory answer is given by Xin
Li~\cite{Li:Semigroup_amenability}, assuming~\(P\) to be left
cancellative.  If~\(P\) is Ore, we shall describe Xin Li's
\(\Cst\)\nb-algebra as the Cuntz--Pimsner algebra of a product
system.  More precisely, we change the order of multiplication
in~\(P\) so as to get product systems over~\(P\) instead of
over~\(P^\op\), compare Remark~\ref{rem:left_right}.  Thus for
us~\(P\) is a right cancellative, right Ore monoid, and we describe
the semigroup \(\Cst\)\nb-algebra of~\(P^\op\) in the notation
of~\cite{Li:Semigroup_amenability}.  The discussion below is closely
related to the description of semigroup \(\Cst\)\nb-algebras
in~\cite{Cuntz-Echterhoff-Li:K-theory}.

Why is it non-trivial to construct semigroup \(\Cst\)\nb-algebras?
There is an obvious product system over any monoid: just take the
complex numbers everywhere, with the obvious multiplication maps.  A
nondegenerate representation of this product system, however, is a
representation of~\(P\) by \emph{unitaries}, not isometries.  Hence the
Cuntz--Pimsner algebra of this product system is the group
\(\Cst\)\nb-algebra of the group completion of~\(P\).  To get an
interesting Cuntz--Pimsner algebra, we need a non-trivial product
system.

So why not take the universal \(\Cst\)\nb-algebra for
representations of~\(P\) by isometries?  This is generated by one
isometry~\(s_p\) for each \(p\in P\), with the relations \(s_p^*
s_p=1\) for all \(p\in P\) and \(s_p s_q = s_{q p}\) for all
\(p,q\in P\).  This universal semigroup \(\Cst\)\nb-algebra
of~\(P^\op\) is introduced by
Murphy~\cite{Murphy:Crossed_semigroups}.  It is, however, usually
too ``wild'' to say much about it.  It is rarely simple or exact.

The right way to ``tame'' Murphy's universal \(\Cst\)\nb-algebra of
a semigroup is to impose relations on the range projections of the
isometries~\(s_p\).  Xin Li~\cite{Li:Semigroup_amenability} proposes
a set of such relations modelled on properties of the regular
representation of the semigroup.  We are about to construct a
product system~\((\Hilm_p)_{p\in P}\) that gives Xin Li's
\(\Cst\)\nb-algebra, and such that \(s_p^*\in \Hilm_p\).  Thus \(s_p
s_p^* \in \Hilm_1\).  In our approach, the desired relations among
the range projections~\(s_p s_p^*\) are encoded in the
\(\Cst\)\nb-algebra \(D\defeq \Hilm_1\).

Let~\(D\) be the \(\Cst\)\nb-subalgebra of~\(\ell^\infty(P)\)
generated by the characteristic functions of left ideals of the form
\(P p q^{-1}\cap P\) for \(p,q\in P\), where
\[
P p q^{-1}\cap P = \{x\in P \mid \exists y\in P\colon x q = y p\}.
\]
In the following, we will use the group completion~\(G\) of~\(P\)
and write~\(P g\) for \(g\in G\).  Since \(P\subseteq G\) is right
cancellative, we have \(x q \in P p\) if and only if \(x q t \in P p
t\) for some \(t\in P\); thus \(P g\) for \(g\in G\) is
well-defined, that is, does not depend on how we write \(g = p
q^{-1}\) for \(p,q\in P\).

Since~\(D\) is commutative and unital, it is of the form
\(\Cont(X)\) for a compact space~\(X\).  Right translation by \(p\in
P\) maps \(P g\cap P\) to \(P gp \cap P p\).  The characteristic
function of this intersection is the product of the characteristic
functions of \(P g p\cap P\) and \(P p\cap P\) and hence belongs
to~\(D\).  Thus right translation by~\(p\) gives an endomorphism
\(\varphi_p\colon D\to D\).  These maps form an action of~\(P^\op\)
on~\(D\) by endomorphisms, that is, \(\varphi_p\circ \varphi_q =
\varphi_{q p}\) for all \(p,q\in P\).

The endomorphism~\(\varphi_p\) for \(p\in P\) maps the constant
function~\(1\) to the characteristic function~\(e_p\) of the subset
\(P\cdot p\subseteq P\).  Actually, \(\varphi_p\colon D\to e_p D\)
is an isomorphism because the map
\[
P \to P\cdot p,\qquad
x\mapsto x\cdot p,
\]
is bijective and this map and its inverse map the generators
of~\(D\) to the characteristic functions of subsets of the form \(P
g\cap P p\), which are exactly all the generators of~\(e_p D\).

We may turn the action of~\(P^\op\) by the
endomorphisms~\((\varphi_p)_{p\in P}\) into a product
system~\((\Hilm_p)_{p\in P}\) over~\(P\) as in
Remark~\ref{rem:left_right}, with reversed order of multiplication
in~\(P\).  Explicitly, \(\Hilm_p = \varphi_p(D)\cdot D =
\varphi_p(1_D)\cdot D = e_p D\) as a Hilbert \(D\)\nb-module, with
left action of~\(D\) through~\(\varphi_p\), and the multiplication
maps are \(\Hilm_p\otimes_D \Hilm_q \congto \Hilm_{pq}\), \(x\otimes
y\mapsto \varphi_q(x)\cdot y\).  The Cuntz--Pimsner
\(\Cst\)\nb-algebra of this product system is canonically isomorphic
to the semigroup crossed product for the action~\((\varphi_p)_{p\in
  P}\) of~\(P^\op\) on~\(D\) (see
\cite{Fowler:Product_systems}*{Section~3}), and this is isomorphic
to Xin Li's semigroup \(\Cst\)\nb-algebra of~\(P^\op\) by
\cite{Li:Semigroup_amenability}*{Lemma 2.14}.

Explicitly, the isomorphism looks as follows.  The
Cuntz--Pimsner algebra~\(\CP\) of~\((\Hilm_p)_{p\in P}\) is
generated by~\(D\) and copies \(S_p(\Hilm_p)\) of~\(\Hilm_p\) for
each \(p\in P\).  Since \(\Hilm_p = e_p\cdot D\) and \(S_p(e_p d)=
S_p(e_p) d\) for all \(d\in D\), \(\CP\) is already generated by
\(D\) and the elements \(s_p = S_p(e_p)^*\).  Since
\(\braket{e_p}{e_p}_{\Hilm_p} = e_p\) and \(\ket{e_p}\bra{e_p} =
\Id_{\Hilm_p}\), the element~\(s_p\) is an isometry with range
projection~\(e_p\).  Furthermore, \(s_p^* \cdot s_q^* = s_{p q}^*\)
or, equivalently, \(s_p \cdot s_q = s_{q p}\) for all \(p,q\in P\).
As it turns out, the relations of~\(D\) and the relations \(s_p
s_p^* = e_p\), \(s_p^* s_p = 1\), \(s_p \cdot s_q = s_{q p}\) for
\(p,q\in P\) imply all relations for the Cuntz--Pimsner algebra
of~\((\Hilm_p)_{p\in P}\).  Thus the Cuntz--Pimsner
algebra agrees with Xin Li's semigroup \(\Cst\)\nb-algebra
of~\(P^\op\).

Now we describe a groupoid model for our Cuntz--Pimsner algebra.
Such groupoid models are already constructed
in~\cite{Li:Nuclearity_semigroup}, even under weaker assumptions on
the semigroup~\(P\).

The projection~\(e_p\) in~\(\Cont(X)\) corresponds to a clopen
subset \(V_p\subseteq X\).  The isomorphism \(\varphi_p\colon D\to
e_p D\) corresponds to a homeomorphism \(\rg_p\colon V_p\to X\).
More precisely, \(\varphi_p(x)|_{V_p} = x\circ \rg_p\) and
\(\varphi_p(x)|_{X\setminus V_p}=0\) for all \(x\in D\).  Let
\(\s_p\colon V_p \to X\) be the inclusion map; this is a
homeomorphism onto a clopen subset and hence a local homeomorphism.
Thus \((V_p,\rg_p,\s_p)\) is a proper topological correspondence
on~\(X\).  The resulting correspondence on \(D=\Cont(X)\) is
\(\Cont(V_p) = e_p D\) with the obvious right Hilbert \(D\)-module
structure and the left \(D\)\nb-action \(\varphi_p =
\rg_p^*\).  This is equal to the
\(\Cst\)\nb-correspondence~\(\Hilm_p\) associated to the
endomorphism~\(\varphi_p\).

Now identify \(V_p\cong X\) through~\(\rg_p\) and rewrite our
topological correspondence in the form \((X,\Id_X,\theta_p)\), where
\(\theta_p\colon X\to U_p \subseteq X\) applies the
inverse~\(\rg_p^{-1}\); these topological correspondences are as
in~\eqref{eq:partial_local_homeo}, where \(U_p=X\) and~\(\theta_p\)
is a homeomorphism onto a clopen subset of~\(X\).  We have
\(\theta_p\circ\theta_q = \theta_{qp}\) for all \(p,q\in P\) because
\(\varphi_p\circ \varphi_q = \varphi_{q p}\).  Hence we get an
action of~\(P^\op\) on~\(X\) by partial homeomorphisms, which is an
action of the type considered in
Section~\ref{sec:partial_local_homeo}.  The resulting product system
over~\(P\) is canonically isomorphic to the product
system~\((\Hilm_p)_{p\in P}\) associated to the
action~\((\varphi_p)_{p\in P}\) of~\(P^\op\) on \(D=\Cont(X)\) by
\Star{}endomorphisms.

As in Section~\ref{sec:partial_local_homeo}, we get a groupoid model
for the Cuntz--Pimsner algebra of the product
system~\((\Hilm_p)_{p\in P}\), and this groupoid model is the
groupoid of germs for the pseudogroup of partial homeomorphisms
of~\(X\) generated by the partial homeomorphisms~\(\theta_p\) for
\(p\in P\).

Finally, we mention a quicker alternative definition of Xin Li's
semigroup \(\Cst\)\nb-algebras, see
also~\cite{Li:Nuclearity_semigroup} for an extension of this
approach to more general semigroups.  The main result
of~\cite{Laca:Endomorphisms_back} shows that any semigroup crossed
product is a full corner in the crossed product for a group action
on a larger \(\Cst\)\nb-algebra.  In our case, this larger
\(\Cst\)\nb-algebra is the \(\Cst\)\nb-subalgebra~\(A\)
of~\(\ell^\infty(G)\) generated by the characteristic functions of
subsets of the form \(P g\subseteq G\) for \(g\in G\).  The right
translation action of~\(G\) on~\(\ell^\infty(G)\) restricts to an
action of~\(G\) on~\(A\) by automorphisms.  This also induces an
action of~\(G\) by homeomorphisms on the spectrum~\(Y\) of the
\(\Cst\)\nb-algebra~\(A\).  The \(\Cst\)\nb-algebra~\(D\) is the
full corner in~\(A\) corresponding to the projection~\(1_P\), the
characteristic function of \(P\subseteq G\).  The
action~\((\varphi_p)\) on~\(D\) above is the compression of the
action of~\(G\) on~\(A\).  Hence the semigroup crossed product
discussed above is canonically isomorphic to the full corner \(1_P
(A\rtimes G) 1_P\) in the crossed product \(A\rtimes G\).
Similarly, the groupoid model for the action of~\(P\) on~\(X\) is
the restriction of the transformation groupoid \(Y\rtimes G\) to the
compact-open subset \(X\subseteq Y\).

\section{Properties of the groupoid model}
\label{sec:properties_groupoid_model}

Let an Ore monoid~\(P\) act on a locally compact space~\(X\) by
topological correspondences \((M_p,\sigma_{p,q})\).  The
Cuntz--Pimsner algebra of the resulting product system over~\(P\) is
identified with a groupoid \(\Cst\)\nb-algebra~\(\Cst(H)\) in
Theorem~\ref{the:groupoid_model}.  Many properties of~\(\Cst(H)\)
are equivalent or closely related to properties of the underlying
groupoid~\(H\).  We harvest some known results of this type
regarding nuclearity, simplicity or ideal structure, tracial and KMS
weights, and pure infiniteness.

One interesting aspect of our groupoid model is that it involves a
``precompiler'': given an action of an Ore monoid~\(P\) by
\emph{topological correspondences} on a space~\(X\), we first
construct an action of~\(P\) on another
space~\(H^0\) by \emph{local homeomorphisms}, and then we take the
Cuntz--Pimsner algebra of this new action.  Hence any
\(\Cst\)\nb-algebra that may be obtained as the Cuntz--Pimsner
algebra of some action of~\(P\) by topological correspondences may
also be obtained from an action of~\(P\) on another space by local
homeomorphisms.
Therefore, for some purposes we may assume without loss of generality
that we are dealing with an action of~\(P\) by local homeomorphism.
In this section, however, the main point is to rewrite properties of
the groupoid~\(H\) in terms of the original action on~\(X\).  For
actions of~\(P\) by local homeomorphisms, what we are going to do is
already well-known.  Our criteria simplify further if the
space~\(X\) is discrete; this happens, in particular, for higher-rank
graphs.

We begin with quick criteria for separability, unitality and
nuclearity.

\begin{remark}
  The groupoid \(\Cst\)\nb-algebra
  of an étale locally compact groupoid is separable if and only if the
  underlying groupoid is second countable.  This happens if and only
  if the closed subspace \(X'\subseteq X\)
  is second countable and the group~\(G\)
  is countable.  This follows if~\(X\)
  is second countable and~\(P\)
  is countable; in the latter case, we can see directly that the
  Cuntz--Pimsner algebra is separable.
\end{remark}

\begin{remark}
  The \(\Cst\)\nb-algebra~\(\Cst(H)\)
  is unital if and only if~\(H^0\)
  is compact.  This happens if and only if the closed subspace
  \(X'\subseteq X\)
  is compact because the projection map \(\pi_1\colon H^0\to X\)
  is a continuous, proper map with image~\(X'\)
  (see Lemma~\ref{lem:H_locally_compact} and the discussion after
  Definition~\ref{def:impossible_situation}).
\end{remark}

\begin{theorem}
  \label{the:amenable_groupoid_model}
  The full groupoid \(\Cst\)\nb-algebra~\(\Cst(H)\) is nuclear if
  and only if the groupoid~\(H\) is topologically amenable.  In that
  case, \(\Cst(H)\) belongs to the bootstrap class.  The groupoid
  \(H_1\subseteq H\) is always topologically amenable.  If~\(G\) is
  amenable, then the groupoid~\(H\) is also amenable, and
  \(\Cst(H_1)\) belongs to the bootstrap class.
\end{theorem}

\begin{proof}
  An étale, Hausdorff, locally compact groupoid is (topologically)
  amenable if and only if its reduced \(\Cst\)\nb-algebra is nuclear
  by
  \cite{Renault_AnantharamanDelaroche:Amenable_groupoids}*{Corollary
    6.2.14}.  Furthermore, if~\(H\) is amenable, then its reduced
  and full \(\Cst\)\nb-algebras coincide, so the full one is also
  nuclear.  Conversely, if the full groupoid \(\Cst\)\nb-algebra is
  nuclear, then so is the reduced one because nuclearity is
  hereditary for quotients.  Hence nuclearity of the full groupoid
  \(\Cst\)\nb-algebra is also equivalent to amenability of the
  groupoid.

  Any amenable groupoid is ``a-T-menable'' by
  \cite{Tu:BC_moyennable}*{Lemma 3.5}; that is, it acts properly and
  isometrically on a continuous field of affine Euclidean spaces.  The
  proof of the Baum--Connes conjecture for a-T-menable groupoids also
  shows that their groupoid \(\Cst\)\nb-algebras belong to the
  bootstrap class, see \cite{Tu:BC_moyennable}*{Proposition 10.7}.

  Since~\(\Cont_0(X)\) is nuclear, Theorem~\ref{the:CP_nuclear} shows
  that the unit fibre~\(\CP_1\) in the associated Cuntz--Pimsner
  algebra is always nuclear; then~\(\CP\) itself is nuclear if~\(G\)
  is amenable.  Theorem~\ref{the:groupoid_model} and its proof
  identify \(\CP_1\) and~\(\CP\) with \(\Cst(H_1)\) and~\(\Cst(H)\),
  respectively.  So the statements about amenability of \(H_1\)
  and~\(H\) follow from the first sentence in the theorem.

  It is elementary to prove the topological amenability of~\(H_1\)
  directly.  The open subgroupoids~\(H^1_{p,p}\) for \(p\in P\) are
  proper equivalence relations.  So we may normalise the counting
  measure on the fibres of~\(H^1_{p,p}\) to give an invariant mean
  on~\(H^1_{p,p}\).  When we view these invariant means on~\(H^1_{p,p}\)
  as means on~\(H_1\) for~\(p\) in the filtered category~\(\Cat_P\),
  we get an approximately invariant mean on~\(H_1\).
\end{proof}

We have not yet tried to characterise amenability of~\(H\) in terms of
the original action by topological correspondences.

\subsection{Open invariant subsets and minimality}
\label{sec:open_invariant}

\begin{definition}
  \label{def:minimal}
  A topological groupoid~\(H\) is \emph{minimal} if~\(H^0\) has no
  open, invariant subsets besides \(\emptyset\) and~\(H^0\).
\end{definition}

Being minimal is a necessary condition for~\(\Cst(H)\)
to be simple because open invariant subsets of~\(H^0\)
generate ideals in~\(\Cst(H)\).
We are going to describe the open, invariant subsets of~\(H^0\)
in terms of the original data \((M_p,\sigma_{p,q})\).
The following lemma gives a base for the topology on~\(H^0\)
and will also be used for other purposes.

\begin{lemma}
  \label{lem:H_base}
  For \(p\in P\) and an open subset \(U\subseteq M_p\), let
  \[
  \pi_p^{-1}(U) \defeq \{(m_q)_{q\in P} \in H^0 \mid m_p\in U\} \subseteq H^0.
  \]
  The family~\(\mathcal{B}\) of subsets of this form is a base for the
  topology on~\(H^0\) and, for each \(x\in H^0\), the subsets
  \(\pi_p^{-1}(U)\) with \(x\in U\) form a neighbourhood base at~\(x\).
  This base for the topology is closed under finite unions, finite
  intersections, and under applying~\(\tilde{\s}_t^{-1}\) for all
  \(t\in P\); and
  \begin{equation}
    \label{eq:base_tilde_s}
    \tilde{\s}_t\bigl(\pi_t^{-1}(U)\bigr) = \pi_1^{-1}(\s_t(U))
    \qquad\text{for all }t\in P\text{ and }U\subseteq M_t\text{ open.}
  \end{equation}
\end{lemma}

\begin{proof}
  By definition of the product topology, \emph{intersections}
  \(\bigcap_{p\in F} \pi_p^{-1}(U_p)\) for finite subsets \(F\subseteq
  P\) and open subsets \(U_p\subseteq M_p\) for \(p\in F\) form a base
  of the topology on~\(H^0\), and such intersections with
  \(x\in \pi_p^{-1}(U_p)\) form a neighbourhood base for \(x\in
  H^0\).  If~\(\mathcal{B}\) is closed under finite intersections,
  then~\(\mathcal{B}\) itself is this canonical base of the topology,
  and similarly for neighbourhoods of~\(x\).

  Let \(p,q\in P\).  Then \(\rg_{p,q}(m_{p q})=m_p\) for all
  \((m_t)_{t\in P}\in H^0\).  Thus
  \begin{equation}
    \label{eq:base_H_relation}
    \pi_{p q}^{-1}(\rg_{p,q}^{-1}(U)) = \pi_p^{-1}(U)
  \end{equation}
  for each open subset \(U\subseteq M_p\).
  Since the maps~\(\rg_{p,q}\)
  are continuous, \(\rg_{p,q}^{-1}(U)\)
  is again open.  Now we consider a finite intersection
  \(\bigcap_{i=1}^n \pi_{p_i}^{-1}(U_i)\)
  for \(F=\{p_1,\dotsc,p_n\}\subseteq P\)
  and \(U_i=U_{p_i} \subseteq M_{p_i}\).
  Since~\(P\)
  is a right Ore monoid, there are \(p\in P\)
  and \(q_i\in P\)
  with \(p_iq_i=p\)
  for \(i=1,\dotsc,n\).
  Then \(\pi_{p_i}^{-1}(U_i) = \pi_p^{-1}(\rg_{p_i,q_i}^{-1}(U_i))\).
  Thus
  \begin{align*}
    \bigcap_{i=1}^n \pi_{p_i}^{-1}(U_i)
    &= \pi_p^{-1}\biggl( \bigcap_{i=1}^n \rg_{p_i,q_i}^{-1}(U_i)\biggr),\\
    \bigcup_{i=1}^n \pi_{p_i}^{-1}(U_i)
    &= \pi_p^{-1}\biggl( \bigcup_{i=1}^n \rg_{p_i,q_i}^{-1}(U_i)\biggr).
  \end{align*}
  Thus~\(\mathcal{B}\) is closed under finite intersections and
  finite unions.  We have
  \begin{equation}
    \label{eq:base_tilde_s_inverse}
    \tilde{\s}_t^{-1}(\pi_p^{-1}(U)) = \pi_{tp}^{-1}(\s_{t,p}^{-1}(U))
  \end{equation}
  because \(\tilde{\s}_t((m_p)) = (\s_{t,p}(m_{tp}))_{p\in P}\).
  Thus~\(\mathcal{B}\) is closed under~\(\tilde{\s}_t^{-1}\) for all
  \(t\in P\).

  Lemma~\ref{lem:tilde_s_local_homeo} shows that
  \((\pi_t,\tilde{\s}_t)\colon H^0\to M_t\times_{\s_t,X,\pi_1} H^0\)
  is a homeomorphism.  Given \(U\subseteq M_t\),
  consider the set of all \(\omega\in H^0\)
  for which there is \(m\in U\)
  with \((m,\omega)\in M_t\times_{\s_t,X,\pi_1} H^0\).
  By definition of the fibre product, this is \(\pi_1^{-1}(\s_t(U))\).
  The homeomorphism \((\pi_t,\tilde{\s}_t)\)
  shows, however, that it is also \(\tilde{\s}_t(\pi_t^{-1}(U))\).
  This gives~\eqref{eq:base_tilde_s}.
\end{proof}

\begin{definition}
  \label{def:indicator}
  The \emph{indicator} of an invariant subset \(A\subseteq H^0\) is
  the following subset of~\(X\):
  \[
  A^\# \defeq \{x\in X\mid \pi_1^{-1}(x) \subseteq A\}.
  \]
\end{definition}

Let \(X'\subseteq X\) be the subset of possible situations.  By
definition, any saturated subset contains \(X\setminus X'\).  So we
lose no information if we restrict attention to \(A^\#\cap X'\).

\begin{deflem}
  \label{def:invariant_subset}
  The indicator~\(B\) of an open invariant subset is open in~\(X\) and
  has the following two properties:
  \begin{description}
  \item[hereditary] if \(p\in P\), \(m\in M_p\) satisfy \(\rg_p(m)\in
    B\), then \(\s_p(m)\in B\);
  \item[saturated] if \(p\in P\), \(x\in X\) satisfy \(\s_p(m)\in B\)
    for all \(m\in\rg_p^{-1}(x)\), then \(x\in B\).
  \end{description}
  If \(S\subseteq P\) is a subset that generates~\(P\), then a subset
  \(B\subseteq X\) is hereditary and saturated if and only if it
  satisfies the above two conditions for all \(p\in S\).
\end{deflem}

\begin{proof}
  Let \(B=A^\#\).  First we show that~\(A^\#\) is open.  By
  definition, \(X\setminus A^\# = \pi_1(H^0\setminus A)\).  The map
  \(\pi_1\colon H^0\to X\) is proper by
  Lemma~\ref{lem:H_locally_compact}.  So it maps the closed subset
  \(H^0\setminus A\) to a closed subset of~\(X\).  Hence \(X\setminus
  A^\#\) is closed and~\(A^\#\) is open.

  Next we check that~\(A^\#\)
  is hereditary.  If \(\s_p(m)\notin A^\#\),
  then there is \(\eta\in H^0\setminus A\)
  with \(\pi_1(\eta)=\s_p(m)\).
  The concatenation \(m\cdot\eta\in H^0\)
  exists because \(\s_p(m)=\pi_1(\eta)\),
  and has \(\pi_1(m\cdot\eta) = \rg_p(m)\).
  Since~\(A\)
  is invariant, so is~\(H^0\setminus A\).
  Hence \(m\cdot\eta\notin A\) and \(\rg_p(m)\notin A^\#\).

  We check that~\(A^\#\)
  is saturated.  Assume \(\s_p(m)\in A^\#\)
  for all \(m\in\rg_p^{-1}(x)\),
  and let \(\eta\in H^0\)
  satisfy \(\pi_1(\eta)=x\).
  Decompose~\(\eta\)
  as \(\eta=m\cdot\eta'\)
  with \(\eta'\in H^0\)
  and \(m\defeq \pi_p(\eta)\)
  by Lemma~\ref{lem:tilde_s_local_homeo}.  Since
  \(\rg_p(m)=\pi_1(\eta) = x\),
  the assumption gives \(\pi_1(\eta') = \s_p(m)\in A^\#\), so
  \(\eta'\in A\).
  Since~\(A\)
  is invariant, we get \(\eta=m\cdot\eta'\in A\) and \(x\in A^\#\).

  If a subset \(B\subseteq X\) satisfies the conditions of being
  hereditary and saturated for given \(p,q\in P\), then it also
  satisfies them for \(p\cdot q\) because \(M_{pq} \cong
  M_p\times_{\s_p,X,\rg_q} M_q\).  Hence it suffices to verify these
  conditions for a set of generators for~\(P\).
\end{proof}

\begin{theorem}
  \label{the:invariant_open_in_H}
  The complete lattice of open \(H\)\nb-invariant subsets of~\(H^0\)
  is isomorphic to the complete lattice of open, hereditary, saturated
  subsets of~\(X\).  In one direction, the isomorphism maps an open
  \(H\)\nb-invariant subset \(A\subseteq H^0\) to its indicator; in
  the other direction, it maps an open, hereditary, saturated subset
  \(B\subseteq X\) to
  \[
  A\defeq \bigcup_{p\in P} (\s_p\circ\pi_p)^{-1}(B)
  = \{(m_p)_{p\in P}\in H^0\mid \exists p\colon \s_p(m_p)\in B\}.
  \]
\end{theorem}

\begin{proof}
  Lemma~\ref{def:invariant_subset} shows that the indicator~\(A^\#\)
  of an open invariant subset \(A\subseteq H^0\)
  is open, hereditary and saturated.  Conversely, let \(B\subseteq X\)
  be open, hereditary and saturated, and define \(A\subseteq H^0\)
  as above.  This is clearly open.  We first check that~\(A\)
  is invariant; then we check that its indicator is~\(B\).

  Let \(\eta\in H^0\).  Then \(\s_p\circ\pi_p(\eta)\in X\) is the
  situation at time \(p\in P\) in the complete history~\(\eta\).
  Thus~\(A\) consists of all complete histories that, at some time,
  visit \(B\subseteq X\).  However, if \(\s_p\circ\pi_p(\eta)\in B\),
  then \(\s_{pq}\pi_{pq}(\eta)\in B\) for all \(q\in P\) because~\(B\)
  is hereditary; indeed, \(\pi_{pq}(\eta) \in M_{pq} \cong M_p
  \times_{\s_p,X,\rg_q} M_q\) corresponds to a pair
  \((\pi_p(\eta),m_q)\) with \(\rg_q(m_q) = \s_p\pi_p(\eta) \in B\),
  so \(\s_{pq}(\pi_{pq}(\eta)) = \s_q(m_q) \in B\).  The subset
  \(pP\subseteq P\) is cofinal.  Hence \(\eta\in A\) if and only if
  the set of \(q\in P\) with \(\s_q\circ\pi_q(\eta)\in B\) is cofinal
  in~\(P\).

  A subset~\(A\) of~\(H^0\) is \(H\)\nb-invariant if and only if
  \(\tilde{\s}_t^{-1}(A) = A\) for all \(t\in P\).  Let \(\eta\in
  H^0\) and write it as \(\eta = m\cdot \eta'\) for \(m\in M_t\),
  \(\eta'\in H^0\) with \(\s_t(m) = \rg_t(\eta')\).  Thus
  \(\tilde{\s}_t(\eta)=\eta'\).  Then \(\s_{tp} \pi_{tp}(\eta) =
  \s_p\pi_p(\eta')\) for all \(p\in P\).  So if there is \(p\in P\)
  with \(\s_p\pi_p(\eta')\in B\), then there is \(q\in P\) with
  \(\s_q\pi_q(\eta)\in B\), namely, \(q=tp\); conversely, if there is
  \(q\in P\) with \(\s_q\pi_q(\eta)\in B\), then the set of such
  \(q\in P\) is cofinal and hence contains some element of the
  form~\(tp\) with \(p\in P\) by~\ref{enum:Ore1}.  Then
  \(\s_p\pi_p(\eta')\in B\).  This shows that \(\eta\in A\) if and
  only if \(\eta'\in A\).  So~\(A\) is invariant as desired.

  Now we check that the indicator of~\(A\)
  is~\(B\).
  By construction, if \(x\in B\),
  then \(\pi_1^{-1}(x)\subseteq A\).
  Conversely, let \(x\in X\setminus B\).
  We must construct \(\eta \in H^0\setminus A\)
  with \(\pi_1(\eta)=x\).
  For each \(p\in P\),
  there is \(m_p\in M_p\)
  with \(\rg_p(m_p)=x\)
  and \(\s_p(m_p)\notin B\)
  because otherwise~\(B\)
  would not be saturated.  Since \(X\setminus X'\subseteq B\),
  and \(\s_p(m_p)\notin B\), there is \(\eta_p\in H^0\)
  with \(\pi_p(\eta_p)=m_p\)
  and hence \(\pi_1(\eta_p)=x\) and \(\s_p\circ \pi_p(\eta_p)\notin B\).

  Since~\(B\)
  is open and the map~\(\pi_p\)
  is proper by Lemma~\ref{lem:H_locally_compact}, the set~\(K_p\)
  of all such~\(\eta_p\)
  is a compact subset of~\(H^0\).
  We have seen above that \(\s_p \pi_p(\eta)\in B\)
  implies \(\s_{p q} \pi_{p q}(\eta)\in B\)
  for all \(p,q\in P\).
  Hence \(K_p\supseteq K_{p q}\)
  for all \(p,q\in P\).
  Since~\(P\)
  is Ore, this tells us that \(\{K_p\}_{p\in P}\)
  is a directed set of compact, non-empty subsets in~\(H^0\).
  The intersection of such a family of subsets is non-empty.  A
  point~\(\eta\)
  in the intersection satisfies \(\pi_1(\eta)=x\)
  and \(\s_p \pi_p(\eta)\notin B\)
  for all \(p\in P\), so that \(\eta\notin A\).  Thus \(A^\#=B\).

  We have constructed two maps \(A\mapsto B\) and \(B\mapsto A\) from
  open invariant subsets of~\(H^0\) to open, hereditary, saturated
  subsets of~\(X\) and back, and we have seen that the composite map
  \(B\mapsto A\mapsto B\) is the identity, that is, the indicator of
  the subset~\(A\) defined in the theorem is the given subset~\(B\).
  Conversely, let \(A'\subseteq H^0\) be an invariant open subset.
  Let~\(B\) be its indicator and define \(A\subseteq B\) as in the
  statement of the theorem.  We must show that \(A'=A\).  Since~\(B\)
  is the indicator of~\(A'\), we have \(\pi_1^{-1}(B)\subseteq A'\).
  Since~\(A'\) is invariant, this implies
  \((\s_p\pi_p)^{-1}(B)\subseteq A'\) for all \(p\in P\), that is,
  \(A\subseteq A'\).  It remains to prove \(A'\subseteq A\).  So we
  take \(\eta\in A'\).  Since~\(A'\) is open, Lemma~\ref{lem:H_base}
  gives \(p\in P\) and an open subset \(U\subseteq M_p\) such that
  \(\pi_p^{-1}(U)\subseteq A'\).  If \(\eta'\in H^0\) satisfies
  \(\pi_1(\eta')\in \s_p(U)\), then there is \(m\in U\subseteq M_p\)
  with \(\s_p(m)=\pi_1(\eta')\), so \(m\cdot \eta'\in H^0\) is
  well-defined; it belongs to \(\pi_p^{-1}(U)\subseteq A'\) by
  construction.  Since \(\tilde{\s}_p(m\cdot \eta') = \eta'\)
  and~\(A'\) is invariant, we get \(\eta'\in A'\) for all \(\eta'\in
  H^0\) with \(\pi_1(\eta') \in \s_p(U)\).  Thus \(\s_p(U)\) is
  contained in the indicator~\(B\) of~\(A'\).  Then
  \(\pi_p^{-1}(U)\subseteq A\), so \(\eta \in A\) as desired.
\end{proof}

\begin{corollary}
  \label{cor:H_minimal}
  The groupoid~\(H\) is minimal if and only if~\(X\) has no
  non-trivial open, hereditary, saturated subsets.\qed
\end{corollary}

\begin{example}
  \label{exa:hereditary_saturated_top_graph}
  Let \(P=(\N,+)\)
  and let~\(\rg_1\)
  be surjective, so that we are dealing with a regular topological
  graph (see the beginning of
  Section~\ref{sec:groupoid_model_examples}).  Then our notion of a
  hereditary and saturated subset of the space of vertices~\(X\)
  is the same one used by Katsura~\cite{Katsura:class_III} to describe
  the gauge-invariant ideals of a topological graph
  \(\Cst\)\nb-algebra in the regular case.
\end{example}

\begin{example}
  Let \(P=(\N^k,+)\) and assume that \(X\) is discrete and
  \(\rg_p\colon M_p\to X\) is surjective for all \(p\in\N^k\).  Our
  data is equivalent to that of a row-finite \(k\)\nb-graph without
  sources, and our Cuntz--Pimsner algebra is the higher-rank graph
  \(\Cst\)\nb-algebra in this case.  Since the standard basis
  \(e_1,\dotsc,e_k\) generates~\(\N^k\), a subset is hereditary and
  saturated if and only if it satisfies the conditions in
  Definition~\ref{def:invariant_subset} for \(p=e_1,\dotsc,e_k\).  Our
  notion of being hereditary and saturated is equivalent to the one
  used by Raeburn, Sims and Yeend
  in~\cite{Raeburn-Sims-Yeend:Higher_graph} to describe the
  gauge-invariant ideals in a higher-rank graph \(\Cst\)\nb-algebra.
  Theorem~\ref{the:gauge-invariant_ideal} below will show that, in
  general, the open invariant subsets of~\(H^0\) correspond to those
  ideals in~\(\Cst(H)\) that are ``gauge-invariant'' in a suitable
  sense.
\end{example}

\subsection{Effectivity}
\label{sec:effective}

\begin{definition}
  \label{def:effective}
  An étale topological groupoid~\(H\) is \emph{essentially free} if
  the subset of objects with trivial isotropy is dense in~\(H^0\).  It
  is \emph{effective} if every open subset of \(H^1\setminus H^0\)
  contains an arrow~\(x\) with \(\s(x)\neq \rg(x)\).  Equivalently,
  the interior of the set \(\{h\in H^1\setminus H^0\mid
  \rg(h)=\s(h)\}\) is empty.
\end{definition}

For a second countable, locally compact, \'etale groupoid, being
effective or essentially free are equivalent properties by
\cite{Renault:Cartan.Subalgebras}*{Proposition 3.6} or
\cite{Brown-Clark-Farthing-Sims:Simplicity}*{Lemma 3.1} (these
articles use the name ``topologically principal'' for ``essentially
free'').

Being essentially free is a variant of the aperiodicity condition that
is used to characterise when topological higher-rank graph
\(\Cst\)\nb-algebras are simple (see, for instance,
\cite{Yeend:Groupoid_models}*{Definition 5.2}).  We cannot, however
check whether~\(H\) is essentially free without looking at points
in~\(H^0\), that is, infinite paths.

As we shall see, the following definition characterises when the
groupoid~\(H\)
is effective in terms of the original data of an action by topological
correspondences:

\begin{definition}
  \label{def:effective_tc}
  An action \((M_p,\sigma_{p,q})\) of an Ore monoid~\(P\) on a locally
  compact Hausdorff space~\(X\) by proper topological correspondences
  is \emph{effective} if for all \(p,q\in P\) with \(p q^{-1} \neq1\)
  in~\(G\) and for all non-empty, open subsets \(U\subseteq X'\),
  there are \(a,f,g\in P\) with \(p a f = q a g\) and \(y\in
  M_{paf}=M_{qag}\) with \(\rg_{paf}(y)\in U\) and
  \(\midpart_{p,a,f}(y)\neq\midpart_{q,a,g}(y)\) in~\(M_a\).  Here
  \(X'\subseteq X\) denotes the closed subset of possible situations;
  \(\midpart_{p,a,f}(y)\) denotes the component in the middle
  factor~\(M_a\) after identifying \(M_{paf} \cong M_p\times_X M_a
  \times_X M_f\), and similarly for~\(\midpart_{q,a,g}(y)\).
\end{definition}

Similar criteria for boundary path groupoids of higher-rank
topological graphs being effective have been found by
Wright~\cite{Wright:Aperiodicity-conditions-topological-graphs}; for
higher-rank graphs without topology, such criteria are also given in
\cites{Lewin-Sims:Aperiodicity, Kang-Pask:Aperiodicity,
  Shotwell:Simplicity}.  Before we explain the relationship between
our criterion and others, we prove the theorem suggested by our
notation:

\begin{theorem}
  \label{the:effective}
  The groupoid~\(H\) is effective if and only if the
  action~\((M_p,\sigma_{p,q})\) is effective.
\end{theorem}

\begin{proof}
  We may assume without loss of generality that \(X=X'\).

  First we assume that the action~\((M_p,\sigma_{p,q})\) is not
  effective.  This means that there are \(p,q\in P\) with \(p
  q^{-1}\neq 1\) in~\(G\) and a non-empty open subset \(U\subset
  X'\) such that \(\midpart_{p,a,f}(y) = \midpart_{q,a,g}(y)\)
  in~\(M_a\) for all \(a,f,g\in P\) with \(p a f = q a g\) and all
  \(y\in M_{p a f}\) with \(\rg_{p a f}(y)\in U\).  This means that
  \((\tilde{\s}_{p} x)_a = (\tilde{\s}_{q} x)_a\) for all \(x\in
  \pi^{-1}_1(U)\) and all \(a\in P\).  Hence \(\tilde{\s}_{p} x =
  \tilde{\s}_{q} x\) in~\(H^0\) for all \(x\in \pi^{-1}_1(U)\).
  Thus the elements of the form \((x,p q^{-1},x)\) for \(x\in
  \pi^{-1}_1(U)\) form a bisection~\(B\) in \(H^1\setminus H^0\)
  with \(\rg|_B = \s|_B\), which means that~\(H\) is not effective.

  Now assume that the action~\((M_p,\sigma_{p,q})\) is effective.
  Let \(U\subseteq H^1\setminus H^0\) be a non-empty open subset.
  We need to find \((x,g,y)\in U\) with \(x\neq y\).

  The intersection \(U\cap H^1_{p,q}\) is non-empty for some \(p,q \in P\).
  Replacing~\(U\) by \(U\cap H^1_{p,q}\), we may arrange that
  \(U\subseteq H^1_{p,q}\).  The subgroupoid~\(H_1\) is an
  increasing union of equivalence relations, so it is certainly
  effective.  Hence we are done if \(p q^{-1}=1\) in~\(G\).  Thus we
  may assume from now on that \(p q^{-1}\neq1\) in~\(G\).

  We may shrink~\(U\) to a bisection because~\(H\) is
  étale.  We may then shrink further so that \(\rg(U) =
  \pi^{-1}_t(U_t)\) for some \(t\in P\) and some non-empty
  open subset \(U_t\subseteq M_t\) because subsets of the
  form~\(\pi^{-1}_t(U_t)\) form a base for the topology
  on~\(H^0\) by Lemma~\ref{lem:H_base}.  Since the map
  \(\s_t\colon M_t\to X\) is a local homeomorphism, we may
  shrink~\(U_t\) even further, so that~\(\s_t\) becomes injective
  on~\(U_t\).  Hence~\(\s_t\) restricts to a homeomorphism from
  \(U_t\subseteq M_t\) onto an open subset \(\s_t(U_t)\subseteq X\).

  We are going to show that there is \(x\in \pi^{-1}_t(U_t)\) with
  \(\tilde{\s}_{p} x \neq \tilde{\s}_{q} x\).  Thus~\((x,g,x)\) is not an arrow
  in~\(H\); since \(\rg(U)=\pi^{-1}_t(U_t)\) and \(U\subseteq
  H^1_{p,q}\), there must be \(y\in H^0\) with \((x,g,y)\in U\subseteq
  H^1_{p,q}\).  Since \(\tilde{\s}_{p} x = \tilde{\s}_{q} y \neq \tilde{\s}_{q} x\), we have
  \(x\neq y\), as desired.

  Since~\(P\) is Ore, we may find \(h,i\in P\) with \(p h=t i\).
  Then we may find \(h',i'\in P\) with \(q h h'=t i'\).  Thus \(p
  h h'=t i h'\) and \(q h h'=t i'\).  Since
  \(H^1_{p,q}\subseteq H^1_{p h h',q h h'}\), we may
  replace~\((p,q)\) by \((p h h',q h h')\).  Thus we may assume
  without loss of generality that there are \(p',q'\in P\) with \(p
  =t p'\) and \(q=t q'\).

  Recall that~\(\s_t\)
  restricts to a homeomorphism from~\(U_t\)
  onto \(V\defeq \s_t(U_t)\).
  Hence \(x\mapsto \tilde{\s}_{t} x\)
  is a homeomorphism from \(\pi^{-1}_t(U_t)\)
  onto \(\pi^{-1}_1(V)\)
  (compare Lemma~\ref{lem:tilde_s_local_homeo}).  By assumption, there
  are \(a,f,g \in P\)
  with \(p' a f = q' a g\)
  and \(y\in M_{p' a f}\)
  with \(\rg_{p' a f}(y)\in V\)
  and \(\midpart_{p',a,f}(y)\neq\midpart_{q',a,g}(y)\)
  in~\(M_a\).
  By construction, there is \(z\in M_t\)
  with \(\s_t(z)=\rg_{p' a f}(y)\).
  Then \((z,y) \in M_t\times_X M_{p' a f} \cong M_{t p' a f}\).
  We have \(\tilde{\s}_{p} (z,y) = \tilde{\s}_{p'}(y)\)
  and \(\tilde{\s}_{q} (z,y) = \tilde{\s}_{q'}(y)\)
  because \(p = t p'\)
  and \(q = t q'\).
  The \(M_a\)\nb-component
  of \(\tilde{\s}_{p'}(y)\)
  is \(\midpart_{p',a,f}(y)\),
  and that of \(\tilde{\s}_{q'}(y)\)
  is \(\midpart_{q',a,g}(y)\).
  Since these are different,
  \(\tilde{\s}_{p}(x) \neq \tilde{\s}_{q}(x)\)
  for any \(x\in H^0\)
  with \(t p' a f\)-component~\((z,y)\).
  Such~\(x\)
  exist because we have restricted to possible histories throughout,
  making the maps \(\pi_p\colon H^0\to M_p\)
  surjective for all \(p\in P\).
\end{proof}

\begin{theorem}
  \label{the:Cst_H_simple}
  Assume that~\(P\) is countable and~\(X\) is second countable or,
  more generally, that~\(H\) is second countable.  The Cuntz--Pimsner
  algebra~\(\CP\) or, equivalently, the \(\Cst\)\nb-algebra
  \(\Cst(H)\), is simple if and only if the following three conditions
  are satisfied:
  \begin{enumerate}
  \item \(\Cst(H)=\Cred(H)\);
  \item the action~\((M_p,\sigma_{p,q})\) is effective;
  \item any non-empty, closed, hereditary, saturated of~\(X\)
    contains~\(X'\).
  \end{enumerate}
  The first condition above follows if~\(H\) is amenable and, in
  particular, if~\(G\) is amenable.
\end{theorem}

\begin{proof}
  We use \cite{Brown-Clark-Farthing-Sims:Simplicity}*{Theorem 5.1},
  which characterises when the groupoid \(\Cst\)\nb-algebra of a
  second countable, locally compact, Hausdorff, étale groupoid is
  simple.  The groupoid~\(H\) is always locally compact, Hausdorff,
  and étale.  The third condition is equivalent to the minimality
  of~\(H\) by Corollary~\ref{cor:H_minimal}.  Since~\(H\) is second
  countable, it is essentially free if and only if it is effective
  by \cite{Brown-Clark-Farthing-Sims:Simplicity}*{Lemma 3.1}.  This
  is equivalent to the effectivity of the
  action~\((M_p,\sigma_{p,q})\) by Theorem~\ref{the:effective}.
  Thus our conditions are equivalent to the three conditions in
  \cite{Brown-Clark-Farthing-Sims:Simplicity}*{Theorem 5.1}.
\end{proof}

Kwa\'sniewski and Szyma\'nski~\cite{Kwasniewski-Szymanski:Ore} also
provide an aperiodicity criterion and use it to prove simplicity
results and a uniqueness theorem for certain Cuntz--Pimsner algebras
over Ore monoids.  On the one hand, their criterion still works if the
unit fibre~\(A\) of a product system is only a liminal
\(\Cst\)\nb-algebra.  On the other hand, it does not imply the
simplicity of the Cuntz algebra~\(\CP_n\) because it only uses the
(partial, multivalued) action of~\(P\) on the spectrum of~\(A\)
induced by the product system.  This contains no information if
\(A=\C\) as in the standard construction of~\(\CP_n\).

Now we simplify the definition of being effective for the monoids
\(P=(\N^k,+)\).
Some of the steps also work for other monoids.

\begin{lemma}
  \label{lem:effective_reduced}
  The definition of an effective action of~\(\N^k\)
  does not change if we assume, in addition, that the pair
  \((p,q)\in \N^k\)
  is \emph{reduced}, that is, there is no \(t\in \N^k\)
  with \(p,q\in \N^k t\).
\end{lemma}

\begin{proof}
  Take \(p=p' t\), \(q=q' t\) for \(p',q',t\in \N^k\).  Since
  \(U\subseteq X'\) is non-empty and consists of possible situations,
  there is \(m\in M_t\) with \(\rg_t(m)\in U\).  Since~\(\rg_t\) is
  continuous, there is an open neighbourhood~\(V\) of~\(m\) with
  \(\rg_t(V)\subseteq U\), and then \(U'\defeq \s_t(V)\) is open as
  well.  Let \(a,f,g\) and~\(y'\) verify the effectivity criterion for
  \(p',q',U'\).  Then the same \(a,f,g\) and \(y=m\cdot y'\) verify it
  for \(p,q,U\).
\end{proof}

The role of \(f,g\)
in Definition~\ref{def:effective_tc} is merely as padding to make
\(p a f = q a g\).
In a commutative monoid such as~\(\N^k\),
we may simply take \(f=q\)
and \(g=p\);
or we may take \(f=(p\vee q)-p\)
and \(g=(p\vee q)-p\),
where \(p\vee q\)
denotes the maximum of \(p\)
and~\(q\)
in~\(\N^k\).
The latter choice is minimal and may therefore be optimal.

\begin{lemma}
  \label{lem:effective_commutative}
  Let \(P=\N^k\).
  Then the action on~\(X\)
  is effective if and only if, for all reduced \(p,q\in \N^k\)
  with \(p- q \neq0\)
  in~\(\Z^k\)
  and for all non-empty, open subsets \(U\subseteq X'\),
  there are \(a\in\N^k\)
  and \(y\in M_{p+a+f} = M_{q+a+g}\)
  with \(\rg_{p+a+f}(y)\in U\)
  and \(\midpart_{p,a,f}(y)\neq\midpart_{q,a,g}(y)\)
  in~\(M_a\), with \(f=(p\vee q)-q\) and \(g=(p\vee q)-p\).
\end{lemma}

\begin{proof}
  The criterion in the lemma differs from
  Definition~\ref{def:effective_tc} in two ways.  First, we assume
  \((p,q)\)
  reduced, which makes no difference by
  Lemma~\ref{lem:effective_reduced}.  Secondly, we choose
  particular~\(f,g\).
  To see that this makes no difference, choose \(f',g',y'\)
  as in Definition~\ref{def:effective_tc}.  Then
  \(f'-f = g'-g = h\in\N^k\).
  The truncation \(y \defeq \rg_{p+a+f}(y')\in M_{p+a+f}\)
  still has the same mid-part in~\(M_a\)
  and hence also verifies the criterion in
  Definition~\ref{def:effective_tc}.
\end{proof}

The condition in Lemma~\ref{lem:effective_commutative} is exactly
Condition~(ii) in
\cite{Wright:Aperiodicity-conditions-topological-graphs}*{Theorem
  3.1}.  As shown
in~\cite{Wright:Aperiodicity-conditions-topological-graphs}, this
condition is equivalent to Yeend's aperiodicity condition~(A), which
characterises when the groupoid model of the topological higher-rank
graph is essentially free; Wright also gives an example where her
finite-path version of aperiodicity is much easier to check than the
original criterion.

\begin{proposition}
  \label{pro:effective_N}
  Let \(P=(\N,+)\) and assume \(X=X'\).  Then the action of~\(P\)
  on~\(X\) is effective if and only if the set of base points of loops
  without entrances has empty interior in~\(X\).
\end{proposition}

We will explain during the proof what loops without entrances are.
The condition we arrive at characterises when the groupoid model for a
regular topological graph is effective, compare
\cite{Katsura:class_III}*{Definition 6.6}.

\begin{proof}
  Any reduced pair in~\(\N\) is of the form \((p,0)\) or \((0,p)\).
  Since our notion is symmetric, we may as well take \(q=0\).  Thus
  the condition in Lemma~\ref{lem:effective_commutative} says that for
  any non-empty, open subset \(U\subseteq X'\) and any \(p\in\N\)
  there is \(a\in\N\) and \(y\in M_{p a}\) with \(\rg_{p a}(y)\in U\)
  and \(\midpart_{p,a,0}(y)\neq\midpart_{0,a,p}(y)\) in~\(M_a\).  If
  there is \(y\in M_p\) with \(\rg_p(y)\in U\) and \(\s_p(y)\neq
  \rg_p(y)\), then \(a=0\) and~\(y\) will do because the relevant
  mid-parts are \(\s_p(y)\) and~\(\rg_p(y)\), respectively.  Thus we
  may assume as well that \(\s_p(m) = \rg_p(m)\) for all \(m\in M_p\)
  with \(\rg_p(m)\in U\); such a path is called a \emph{loop}, and
  \(\rg_p(m)\in X\) is its \emph{base point}.

  Take \(m\in M_p\) with \(\rg_p(m)\in U\).  Identify \(M_p \cong
  M_1^{\times_X p}\) and write~\(m\) as a path
  \[
  x_1 \xleftarrow{m_1} x_2
  \xleftarrow{m_2} x_3
  \xleftarrow{m_3} x_4
  \to \dotsb \xleftarrow{m_p} x_{p+1} = x_1
  \]
  with \(x_1,\dotsc,x_p\in X\), \(m_1,\dotsc,m_p\in M_1\).  An
  \emph{entrance} for this loop is \(1\le i\le p\) and an arrow
  \(m'\colon x'\to x_i\) with \(m'\neq m_i\).  If~\(t\) is such an
  entrance, then take \(a=i\) and let \(y\in M_{p a}\) be the
  concatenation of the loop~\(m\) and the path
  \[
  x_1
  \xleftarrow{m_1} x_2
  \xleftarrow{m_2} x_3
  \leftarrow \dotsb \leftarrow x_i
  \xleftarrow{m'} x'.
  \]
  The relevant length-\(a\) mid-parts of this concatenation are the
  paths from \(x'\) to~\(x_1\) and~\(x_p\) to~\(x_1\) involving \(m'\)
  and~\(m_i\), respectively.  Hence an entrance to some loop gives the
  data \((a,y)\) required in Definition~\ref{def:effective_tc}.  If
  the set of base points of loops without entrances has empty
  interior, then we may find some loop with an entrance with base
  point in~\(U\), so the action is effective.

  Conversely, assume that~\(U\)
  is an open subset so that all points in~\(U\)
  are base points of some loop without an entrance.  Then there is
  only one path with range~\(x\)
  for any \(x\in U\):
  we must follow the loop based at that point because it has no
  entrances (note that our paths go backwards).  A continuity argument
  shows that the period of the loop is locally constant.
  Shrinking~\(U\),
  we may arrange that it is constant equal to~\(p\)
  for all \(x\in U\).
  Then \(\midpart_{p,a,0}(m) = \midpart_{0,a,p}(m)\)
  for all \(m\in M_{p+a}\) with \(\rg_{p+a}(m)\in U\).
\end{proof}

If~\(X\) is discrete, so that we are dealing with an ordinary graph
\(\Cst\)\nb-algebra, then Proposition~\ref{pro:effective_N} simplifies
further to the condition that there are no loops without entrances,
which is a standard condition in the theory of graph
\(\Cst\)\nb-algebras (see~\cite{Raeburn:Graph_algebras}).

\subsection{Gauge-invariant ideals}
\label{sec:gauge-invariant}

We now describe which ideals in~\(\Cst(H)\) come from open
\(H\)\nb-invariant subsets of~\(H^0\) if~\(H\) is not effective.
Recall that the groupoid~\(H\) is graded by the group completion~\(G\)
of~\(P\), \(H=\bigsqcup_{g\in G} H_g\), and that this corresponds to
the Fell bundle structure on the Cuntz--Pimsner algebra~\(\CP\), that
is, \(\CP_1 = \Cst(H_1)\) is the groupoid \(\Cst\)\nb-algebra of the
subgroupoid~\(H_1\).  The canonical projection to~\(\CP_1\) is a
conditional expectation \(E\colon \CP\to \CP_1\).

\begin{definition}
  \label{def:gauge-invariant}
  We call an ideal \(I\subseteq \CP\)
  \emph{gauge-invariant} if~\(I\)
  is equal to the ideal in~\(\CP\)
  generated by \(E(I)\subseteq\CP_1\).
\end{definition}

If~\(G\)
is an Abelian group, so that there is a dual action of~\(\hat{G}\)
on~\(\CP\),
then standard arguments show that an ideal is ``gauge-invariant'' in
the sense of Definition~\ref{def:gauge-invariant} if and only if it is
invariant under the dual action of the compact group~\(\hat{G}\)
on~\(\CP\)
in the usual sense because
\(E(x) = \int_{\hat{G}} \alpha_\gamma(x) \,\diff \gamma\).

\begin{theorem}
  \label{the:gauge-invariant_ideal}
  The gauge-invariant ideals in~\(\CP\) are in bijection with open
  invariant subsets of~\(H^0\) or, equivalently, with open, hereditary
  and saturated subsets of~\(X\).
\end{theorem}

\begin{proof}
  The bijection between invariant subsets of~\(H^0\) and open,
  hereditary and saturated subsets of~\(X\) is
  Theorem~\ref{the:invariant_open_in_H}.  An open and invariant subset
  \(U\subseteq H^0\) corresponds to the ideal \(\Cst(H|_U)\subseteq
  \Cst(H)\), where~\(H|_U\) denotes the subgroupoid of~\(H\) with
  object space \(U\subseteq H^0\) and arrow space \(\rg^{-1}(U)=
  \s^{-1}(U) \subseteq H^1\).  This ideal is gauge-invariant with
  \(E(\Cst(H|_U)) = \Cst(H_1|_U)\).

  Conversely, let \(I\subseteq \Cst(H)\) be a gauge-invariant ideal.
  Then \(E(I)\) is an ideal in \(\Cst(H_1)\).  The groupoid~\(H_1\) is
  an approximately proper equivalence relation or a hyperfinite
  relation.  This implies that it is amenable and that any restriction
  of~\(H_1\) to an invariant closed subset remains effective.  Now
  \cite{Brown-Clark-Farthing-Sims:Simplicity}*{Corollary 5.9} implies
  that ideals in~\(\Cst(H_1)\) are in bijection with
  \(H_1\)\nb-invariant open subsets of~\(H^0\), where \(U\subseteq
  H^0\) corresponds to the ideal \(\Cst(H_1|_U)\).  This can only be
  of the form~\(E(I)\) for an ideal \(I\subseteq \Cst(H)\) if~\(U\) is
  invariant under the whole groupoid~\(H\), and then the ideal~\(I\)
  generated by \(\Cst(H_1|_U)\) is \(\Cst(H|_U)\).  Thus any
  gauge-invariant ideal~\(I\) is of the form \(\Cst(H|_U)\) for an
  open \(H\)\nb-invariant subset~\(U\) of~\(H^0\).
\end{proof}

\subsection{Invariant measures}
\label{sec:invariant_measures}

In the following, a ``measure'' on a locally compact space~\(X\)
means a Radon measure or, equivalently, a positive linear functional
on~\(\Contc(X)\).
We assume~\(X\)
to be second countable and~\(P\)
to be countable, so that~\(H\)
is second countable.  Let \(c\colon P\to(0,\infty)\)
be a homomorphism to the multiplicative group of positive real
numbers.  This extends to the group completion~\(G\)
and then to~\(H\),
by letting \(c|_{H^1_{p,q}} = c(p)/c(q)\).
This cocycle on~\(H\)
generates a \(1\)\nb-parameter
group of automorphisms of~\(\Cst(H)\),
see~\cite{Renault:Groupoid_Cstar}*{Section 5}.  The one-parameter
automorphism groups of~\(\Cst(H)\)
described above are not as special as it may seem:

\begin{proposition}
  \label{pro:automorphisms_fixing_diagonal}
  Let~\(H\) be effective.  An automorphism of~\(\Cst(H)\) that acts
  trivially on the \(\Cst\)\nb-subalgebra~\(\Cst(H_1)\) is given by
  pointwise multiplication with a homomorphism \(G\to\T\).  A
  one-parameter group of such automorphisms comes from a
  homomorphism \(c\colon P\to(0,\infty)\).
\end{proposition}

\begin{proof}
  Renault's construction of a groupoid model for \(\Cst\)\nb-algebras
  with a Cartan subalgebra in~\cite{Renault:Cartan.Subalgebras} is
  natural.  More precisely, any automorphism of a
  twisted groupoid \(\Cst\)\nb-algebra
  \(\Cst(H,L)\)
  for an \'etale, effective, Hausdorff, locally compact groupoid~\(H\)
  and a Fell line bundle~\(L\)
  over~\(H\)
  that maps the subalgebra \(\Cont_0(H^0)\)
  into itself must come from an automorphism of the pair~\((H,L)\).

  The automorphism fixes \(\Cont_0(H^0)\subseteq \Cst(H_1)\) if and
  only if the induced automorphism of~\(H\) acts trivially on
  objects.  For an effective groupoid, this implies that the
  automorphism acts identically on the inverse semigroup of
  bisections.  Hence it is the identity automorphism.  Thus the only
  source of such automorphisms of~\(\Cst(H,L)\) are automorphisms of
  the Fell line bundle~\(L\); this is the trivial Fell line bundle
  in our case.

  Any automorphism of a Fell line bundle over a groupoid~\(H\) acts
  by pointwise multiplication with a continuous groupoid
  homomorphism \(H\to\T\).  Since we want the automorphism to act
  identically on~\(\Cst(H_1)\), this homomorphism must be constant
  on \(H_1\subseteq H\).  Hence it is constant on the
  subspaces~\(H_g\) and thus comes from a group homomorphism
  on~\(G\).  Since~\(G\) is the group completion of~\(P\),
  homomorphisms \(G\to\T\) are in bijection with homomorphisms
  \(P\to\T\).

  A one-parameter automorphism group is equivalent to a continuous
  homomorphism \(P\times\R\to\T\); this is equivalent to a
  continuous homomorphism \(P\to \operatorname{Hom}(\R,\T) \cong \R
  \cong (\R_{>0},\cdot)\).  Hence all one-parameter automorphism
  groups of~\(\Cst(H)\) that fix~\(\Cst(H_1)\) come from a
  homomorphism \(P\to(0,\infty)\).
\end{proof}

\begin{definition}
  \label{def:c-invariant_measure}
  A measure~\(\mu\) on~\(H^0\) is \emph{\(c\)\nb-invariant} if
  \(\mu(\rg(B)) = c(g) \mu(\s(B))\) for any bisection \(B\subseteq
  H^1_g\).  If \(c\equiv 1\), we speak simply of invariant measures.
\end{definition}

A \(c\)\nb-invariant
measure on~\(H^0\)
gives a KMS-weight on~\(\Cst(H)\)
for the corresponding automorphism group (with temperature~\(1\));
conversely, if~\(H\)
has trivial isotropy groups, then any KMS-weight on~\(\Cst(H)\)
for this \(1\)\nb-parameter
group of automorphisms is of this form for a unique \(c\)\nb-invariant
measure on~\(H^0\)
(see \cite{Renault:Groupoid_Cstar}*{Proposition 5.4}); this result is
generalised by Neshveyev~\cite{Neshveyev:KMS_non-principal} to KMS
states on groupoids with non-trivial isotropy, and by
Thomsen~\cite{Thomsen:KMS_weights} to KMS weights.  In particular,
invariant measures on~\(H^0\)
give tracial weights on~\(\Cst(H)\);
if the set of objects with non-trivial isotropy has measure zero for
all invariant measures on~\(H^0\), then all KMS weights are of this form.

We are going to describe invariant measures on~\(H^0\) in terms of
measures on~\(X\).  This requires two operations on measures:
push-forwards along continuous maps and pull-backs along local
homeomorphisms.  The first is standard: if \(f\colon X\to Y\) is a
continuous map and~\(\mu\) is a measure on~\(X\), then \(f_*\mu\) is
the measure on~\(Y\) defined by \(f_*\mu(B)= \mu(f^{-1}(B))\) for
Borel subsets \(B\subseteq Y\).  If \(f\colon X\to Y\) is a
local homeomorphism and~\(\lambda\) is a measure on~\(Y\), then
\(f^*\lambda\) is the measure on~\(X\) defined by
\[
f^*\lambda(B) = \int_Y \abs{f^{-1}(y) \cap B} \,\diff\lambda(y).
\]
If \(h\colon X\to\C\) is Borel measurable with compact support, then
\[
\int_X h(x) \,\diff (f^*\lambda)(x)
= \int_Y \sum_{\{x\in X\mid f(x)= y\}} h(x) \,\diff\lambda(y)
\]
because this holds for characteristic functions of Borel
subsets.

\begin{definition}
  \label{def:c-invariant_measure_tc}
  A measure~\(\lambda\) on~\(X\) is \emph{\(c\)\nb-invariant} if
  \(\lambda= c(p) (\rg_p)_*\s_p^*(\lambda)\) for all \(p\in P\).
\end{definition}

\begin{theorem}
  \label{the:invariant_measures}
  The map~\((\pi_1)_*\) induced by the projection \(\pi_1\colon
  H^0\to X\) is a bijection between \(c\)\nb-invariant measures
  on~\(H^0\) and \(c\)\nb-invariant measures on~\(X\).
\end{theorem}

\begin{proof}
  A measure~\(\mu\) on~\(H^0\) gives measures \(\mu_p \defeq
  (\pi_p)_*(\mu)\) on~\(M_p\) for all \(p\in P\).  These are linked
  by \((\rg_{p,q})_*\mu_{pq} = \mu_p\) for all \(p,q\in P\) because
  \(\rg_{p,q}\circ\pi_{pq} = \pi_p\).  Conversely, we claim that any
  family of measures~\((\mu_p)_{p\in P}\) with
  \((\rg_{p,q})_*\mu_{pq} = \mu_p\) for all \(p,q\in P\) comes from
  a unique measure~\(\mu\) on~\(H^0\).  This is because
  \(\bigcup_{p\in P} \pi_p^*(\Contc(M_p))\) is a dense subspace
  in~\(\Contc(H^0)\).  The consistency condition
  \(\rg_{p,q}\circ\pi_{pq} = \pi_p\) implies that the positive
  linear maps \(\Contc(M_p) \to \C\), \(f\mapsto \int_{M_p}
  f(x)\,\diff\mu_p(x)\), well-define a positive linear map on
  \(\bigcup_{p\in P} \pi_p^*(\Contc(M_p))\).  This extends uniquely to
  a positive linear map
  on~\(\Contc(H^0)\).  Thus we may replace a measure~\(\mu\)
  on~\(H^0\) by a family of measures~\((\mu_p)_{p\in P}\) on the
  spaces~\(M_p\) whenever this is convenient.

  Next we claim that~\(\mu\) is \(c\)\nb-invariant if and only if
  \(\mu_p = c(p)\s_p^*\mu_1\) for all \(p\in P\).

  Assume first that~\(\mu\)
  is \(c\)\nb-invariant.
  Let \(p\in P\)
  and let \(U\subseteq M_p\)
  be an open, relatively compact subset such that
  \(\s_p|_U\colon U\to X\)
  is injective.  Then
  \[
  \{(x,p,\tilde{\s}_p x) \mid x\in \pi_p^{-1}(U)\}\subseteq H^1_{p,1}
  \]
  is a bisection with range~\(\pi_p^{-1}(U)\)
  and source~\(\pi_1^{-1}(\s_p(U))\)
  by Lemma~\ref{lem:H_base}.  Since~\(\mu\) is \(c\)\nb-invariant,
  \[
  \mu_p(U)
  = \mu(\pi_p^{-1}(U))
  = c(p) \mu(\pi_1^{-1}(\s_p(U)))
  = c(p) \mu_1(\s_p(U)).
  \]
  This equality also holds for all open subsets of~\(U\)
  because~\(\s_p\)
  is still injective on them.  This implies
  \(\mu_p(B) = c(p) \mu_1(\s_p(B))\)
  for all Borel subsets~\(B\) of~\(U\).

  If \(B\subseteq M_p\)
  is an arbitrary Borel subset, then we may cover~\(B\)
  by open, relatively compact subsets on which~\(\s_p\)
  is injective because~\(\s_p\)
  is a local homeomorphism.  Then we may decompose~\(B\)
  as a countable disjoint union \(B=\bigsqcup_{i\in\N} B_i\)
  of Borel subsets with \(B_i\subseteq U_i\)
  for open, relatively compact subsets~\(U_i\)
  such that~\(\s_p|_{U_i}\)
  is injective for all~\(i\).
  Applying the formula above for each~\(i\) gives
  \[
  \mu_p(B)
  = c(p) \int_X \abs{\s_p^{-1}(x) \cap B}\,\diff\mu_1(x)
  = c(p) \s_p^*\mu_1(B).
  \]
  Thus \(\mu_p = c(p)\s_p^*\mu_1\) if~\(\mu\) is \(c\)\nb-invariant.

  Conversely, assume \(\mu_p = c(p)\s_p^*\mu_1\).  Let
  \(g=pq^{-1}\in G\) and let \(V\subseteq H^1_g\) be a bisection.
  We may assume without loss of generality that \(V\subseteq
  H^1_{p,q}\), replacing \((p,q)\) by \((ph,qh)\) for some \(h\in
  P\) if necessary.  Decomposing~\(V\) into disjoint Borel subsets,
  we may further reduce to the case where \(\s(V)\subseteq H^0\) is
  one of the base neighbourhoods in Lemma~\ref{lem:H_base}; say,
  \(\s(V)= \pi_t^{-1}(U)\) for an open subset \(U\subseteq M_t\).
  Replacing \((p,q,t)\) by \((p a,q a,t b)\) for suitable \(a,b\in
  P\), we may arrange \(t = p\) by the Ore condition.  We assume
  this from now on.  Decomposing~\(V\) even further, we may arrange
  that~\(\s_p|_U\) is injective.  Then~\(V\) is the product of two
  bisections, one of the form \(\{(x,p,\tilde{\s}_{p}(x))\mid x\in
  \pi_p^{-1}(U)\}\), the other of the form
  \(\{(y,q,\tilde{\s}_{q}(y))^{-1} \mid y\in \rg(V)\}\); here
  \(\tilde{\s}_{q} \rg(V) = \tilde{\s}_{p} \pi_p^{-1}(U) =
  \pi_1^{-1}(\s_p(U))\).  So \(\rg(V)\) must be of the form
  \(\pi_q^{-1}(W)\) for some \(W\subseteq M_q\) for
  which~\(\s_q|_W\) is a homeomorphism onto~\(\s_p(U)\).

  The upshot of these reductions is that~\(\mu\)
  is \(c\)\nb-invariant
  once the \(c\)\nb-invariance
  condition holds for bisections of the form
  \(\{(x,p,\tilde{\s}_{p}(x))\mid x\in \pi_p^{-1}(U)\}\)
  with \(p\in P\)
  and an open subset \(U\subseteq M_p\).
  But this is exactly the condition \(\mu_p = c(p)\s_p^*\mu_1\).
  This finishes the proof of the claim.

  The claim shows that the family of measures~\((\mu_p)_{p\in P}\)
  and hence the measure~\(\mu\) is determined uniquely by the
  measure~\(\mu_1\) on \(M_1=X\) provided~\(\mu\) is
  \(c\)\nb-invariant.  If we are given a measure~\(\lambda\)
  on~\(X\), then \(\mu_p \defeq c(p)\s_p^*(\lambda)\) is the only
  possibility for a \(c\)\nb-invariant measure on~\(H^0\) with
  \(\mu_1=\lambda\).  This family of measures gives a measure
  on~\(H^0\) if and only if \((\rg_{p,q})_*\mu_{pq} = \mu_p\) for
  all \(p,q\in P\).  In particular, for \(p=1\) and \(q\in P\), this
  gives the condition \(c(q) (\rg_q)_*\s_q^*\lambda = (\rg_q)_*\mu_q
  = \lambda\), that is, \(\lambda\) has to be \(c\)\nb-invariant.

  The proof of the theorem will be finished by checking that
  \[
  c(p)\s_p^*\lambda = c(pq) (\rg_{p,q})_*\s_{pq}^* \lambda
  \]
  holds for all \(p,q\in P\)
  if~\(\lambda\)
  is a \(c\)\nb-invariant
  measure on~\(X\).
  Since \(c(pq)=c(p)c(q)\),
  we have to prove
  \(c(q) (\rg_{p,q})_*\s_{pq}^* \lambda = \s_p^*\lambda\).
  Let \(U\subseteq M_p\)
  be an open, relatively compact subset.  On the one hand,
  \((\s_p^*\lambda)(U) \defeq \int_X \abs{\s_p^{-1}(x)\cap U}
  \,\diff\lambda(x)\).
  Substituting \(c(q) (\rg_q)_*\s_q^*\lambda\)
  for~\(\lambda\), this becomes
  \begin{multline}
    \label{eq:c-invariance_1}
    c(q) \int_X \abs{\s_p^{-1}(x) \cap U}
    \,\diff(\rg_q)_*\s_q^*\lambda(x)
    = c(q) \int_{M_q} \abs{\s_p^{-1}(\rg_q(z)) \cap U}
    \,\diff\s_q^*\lambda(z)
    \\= c(q) \int_X \sum_{\{z\in M_q \mid \s_q(z)=x\}}
    \abs{\s_p^{-1}(\rg_q(z)) \cap U} \,\diff\lambda(x).
  \end{multline}
  On the other hand,
  \begin{equation}
    \label{eq:c-invariance_2}
    c(q)((\rg_{p,q})_*\s_{pq}^* \lambda)(U)
    = c(q) \s_{pq}^* \lambda(\rg_{p,q}^{-1}(U))
    = c(q) \int_X \abs{\s_{pq}^{-1}(x) \cap
      \rg_{p,q}^{-1}(U)}\,\diff\lambda(x).
  \end{equation}
  We may identify \(M_{pq} \cong M_p\times_{\s_p,X,\rg_q} M_q\) and
  \(\rg_{p,q}\) with the projection to the first factor.  Hence
  \(\s_{pq}^{-1}(x) \cap \rg_{p,q}^{-1}(U)\) is the set of pairs
  \((y,z)\) with \(y\in U\subseteq M_p\), \(z\in M_q\),
  \(\s_p(y)=\rg_q(z)\) and \(\s_q(z)=x\).  The cardinality of this set
  is the sum over all \(z\in \s_q^{-1}(x)\) of the cardinalities of
  \(\s_p^{-1}(\rg_q(z))\cap U\).  Hence the right hand sides in
  \eqref{eq:c-invariance_1} and~\eqref{eq:c-invariance_2} are equal,
  as desired.
\end{proof}

Similar arguments as in the end of the proof of
Theorem~\ref{the:invariant_measures} show the following:

\begin{remark}
  The map \(\alpha_p \defeq (\rg_p)_*(\s_p)^*\)
  on the space of measures on~\(X\)
  is an action of~\(P\),
  that is, \(\alpha_{pq} = \alpha_p \circ\alpha_q\)
  for all \(p,q\in P\).
  Therefore, if \(S\subseteq P\)
  generates~\(P\),
  then it suffices to check whether a measure is invariant for
  \(p\in S\).
\end{remark}

\begin{example}
  \label{exa:KMS_discrete_X}
  Assume that~\(X\)
  is discrete.  Then a (positive) measure on~\(X\)
  is simply a function \(\lambda\colon X\to\R_+\).
  We have
  \((\rg_p)_*(\s_p)^*(\lambda)(x) = \sum_{\rg_p(m)=x}
  \lambda(\s_p(m))\);
  a \(c\)-invariant
  measure is a non-negative joint eigenvector of these maps with
  eigenvalue \(p\mapsto c(p)\).
  If \(P=\N\),
  then it suffices to look at the generator \(1\in\N\),
  and the resulting map on the measures of~\(X\)
  is matrix-vector multiplication with the adjacency matrix of the
  graph described by~\(X_1\).
  For \(P=(\N^k,+)\),
  it suffices to look at the \(k\)~generators
  of~\(\N^k\).
  Thus the \(c\)\nb-invariant
  weights on~\(H^0\)
  are exactly those joint eigenvalues of the adjacency matrices that
  have a non-negative eigenvector.

  Any \(c\)\nb-invariant
  weight on~\(H^0\)
  gives a KMS state on \(\Cst(H)\)
  by first applying the conditional expectation
  \(\Cst(H)\to\Cont_0(H^0)\)
  that restricts functions to~\(H^0\)
  and then integrating against the measure on~\(H^0\).
  If \emph{all} isotropy groups are trivial, then these are all KMS
  states.  Neshveyev~\cite{Neshveyev:KMS_non-principal} shows how to
  describe KMS states if there is non-trivial isotropy.  A
  particularly simple case is if the set of points in~\(H^0\)
  with non-trivial isotropy is a null set for all \(c\)\nb-invariant
  measures on~\(H^0\).
  In that case, all KMS states for the \(\R\)\nb-action
  generated by the cocycle~\(c\)
  factor through the conditional expectation, so we get the same
  answer as if there were no isotropy groups.

  If the set of non-trivial isotropy has positive measure, then
  Neshveyev's description of KMS states uses essentially invariant
  measurable families of states on the \(\Cst\)\nb-algebras
  of the isotropy groups.  If there is lots of isotropy, these may be
  hard to classify.  The situation is tractable, however, if the set
  of points in~\(H^0\)
  with non-trivial isotropy is countable.

  This always happens for (separable) graph algebras.  The set of
  orbits of points in~\(H^0\)
  with non-trivial isotropy is in bijection with simple loops in the
  graph.  The orbit of a point in~\(H^0\)
  and the set of simple loops are both countable.  Thus for a graph
  \(\Cst\)\nb-algebra,
  we may get all KMS states in two steps.  First we find the
  non-negative eigenvectors of the adjacency matrix.  Then we examine
  which loops in the graph give atoms for the induced measure
  on~\(H^0\).
  For each such loop, we take a state on~\(\Cst(\Z)\),
  that is, a probability measure on the circle.  The KMS states for a
  non-negative eigenvector of the adjacency matrix are in bijection
  with the set of all such families of probability measures on the
  circle.

  In the case of (higher-rank) graphs, the measures on~\(H^0\)
  that we get are essentially product measures coming from a measure
  on the discrete set of vertices~\(V\)
  of the graph.  If the higher-rank graph is effective, then the set
  of points with non-trivial isotropy in~\(H^0\)
  is often also a null set for such product type measures on~\(H^0\).
  This explains why the descriptions of the KMS states or weights on
  (higher-rank) graph \(\Cst\)\nb-algebras
  in, for example, \cites{anHuef-Laca-Raeburn-Sims:KMS_finite,
    anHuef-Laca-Raeburn-Sims:KMS_higher, Thomsen:KMS_weights}, are
  often equivalent to what we found above, namely, those joint
  eigenvalues of the \(k\)~adjacency
  matrices that have a non-negative eigenvector.  As usual, our
  analysis only applies to regular (higher-rank) graphs.
\end{example}

\begin{remark}
  \label{rem:stably_finite}
  It is shown in~\cite{Haagerup-Thorbjornsen:Random_K} that any stably
  finite, exact, unital \(\Cst\)\nb-algebra has a tracial state.
  Conversely, a \(\Cst\)\nb-algebra with a tracial state is stably
  finite.  Thus a unital, exact \(\Cst\)\nb-algebra has a tracial
  state if and only if it is stably finite.  We have already remarked
  that~\(\Cst(H)\) is unital if and only if~\(X'\) is compact, and
  exactness follows if~\(H\) is amenable, compare
  Theorem~\ref{the:amenable_groupoid_model}.  Hence
  Theorem~\ref{the:invariant_measures} gives a necessary and
  sufficient criterion for~\(\Cst(H)\) to be stably finite in the case
  where~\(X'\) is compact metrisable, \(P\) is countable, and~\(H\) is
  amenable.
\end{remark}

\subsection{Local contractivity}
\label{sec:locally_contracting}

\begin{definition}[\cite{Anantharaman-Delaroche:Purely_infinite}]
  \label{def:locally_contracting}
  A locally compact, Hausdorff, \'etale groupoid~\(\Gr\) is
  \emph{locally contracting} if, for every non-empty open subset
  \(U\subset \Gr^0\), there are an open subset \(V\subseteq U\) and
  a bisection~\(B\) of~\(\Gr\) such that \(\overline{V}\subseteq
  \s(B)\) and \(\rg(B\overline{V})\subsetneq V\).
\end{definition}

Since~\(H\) is an étale, locally compact groupoid,
\cite{Anantharaman-Delaroche:Purely_infinite}*{Proposition 2.4}
shows that \(\Cred(H)\) is purely infinite (that is, every
hereditary \(\Cst\)\nb-subalgebra contains an infinite projection)
if~\(H\) is essentially free and locally contracting.  Actually, we
only need~\(H\) to be effective here by
\cite{Brown-Clark-Farthing-Sims:Simplicity}*{Lemma 3.1.(4)}: this is
exactly the condition in
\cite{Anantharaman-Delaroche:Purely_infinite}*{Lemma 2.3} that is
used in the proof of
\cite{Anantharaman-Delaroche:Purely_infinite}*{Proposition 2.4}.
We would, however, not expect local contractivity to be necessary
for~\(\Cred(H)\) to be purely infinite.

The following definition characterises when the groupoid~\(H\) is
locally contracting in terms of the original action by topological
correspondences:

\begin{definition}
  \label{def:locally_contracting_tc}
  An action~\((M_p,\sigma_{p,q})\)
  of an Ore monoid~\(P\)
  on a locally compact Hausdorff space~\(X\)
  by proper topological correspondences is \emph{locally contracting}
  if for any relatively compact, open subset \(S\subseteq X'\),
  there are \(n\in \N\)
  and \(p_i,q_i,a_i,b_i\in P\)
  and \(W_i\subseteq M'_{p_i}\times_{\s_{p_i},X',\s_{q_i}} M'_{q_i}\)
  for \(1\le i \le n\) such that
  \begin{enumerate}[label=\textup{(LC\arabic*)}]
  \item \label{en:contracting-sections5} \(p_1 a_1 = p_2 a_2 = \dotsb
    = p_n a_n = q_1 b_1 = \dotsb = q_n b_n\);

  \item \label{en:contracting-sections1} the restricted coordinate
    projections \(\pr_1\colon W_i\to M'_{p_i}\) and \(\pr_2\colon
    W_i\to M'_{q_i}\) are injective and open;

  \item \label{en:contracting-sections3} the subsets
    \[
    \pr_1(W_i) M'_{a_i}
    \defeq \{\sigma^{-1}_{p_i,a_i}(x_1,x_2)\mid
    x_1\in \pr_1(W_i),\ x_2\in M'_{a_i},
    \s_{p_i}(x_1)=\rg_{a_i}(x_2)\}
    \]
    of~\(M'_{p_i a_i}\) are disjoint, and so are the subsets
    \(\pr_2(W_i) M'_{b_i}\);

  \item \label{en:contracting-sections4}
    \(\overline{\bigsqcup_{i=1}^n \pr_1(W_i) M'_{a_i}} \subsetneq
    \bigsqcup_{i=1}^n \pr_2(W_i) M'_{b_i}\);

  \item \label{en:contracting-sections7}
    \(\rg_{q_i}\pr_2(W_i)\subseteq S\);

  \item \label{en:contracting-sections6} \(p_i q_i^{-1} \neq p_j
    q_j^{-1}\) in~\(G\) for \(i\neq j\).
  \end{enumerate}
  Here \(X'\subseteq X\) is the closed invariant subspace of
  possible situations, which is different from~\(X\) if some of the
  range maps are not surjective.
\end{definition}

The choice of \(a_i,b_i\) does not really matter: if the conditions
hold for one choice satisfying~\ref{en:contracting-sections5}, then
also for all others.  This follows from
Lemma~\ref{lem:unique_truncation} and the surjectivity of the
maps~\(\rg_{p,q}\) on the~\(M'_{pq}\).

Giving up some symmetry, we may use the Ore conditions to simplify the
data above slightly: we may assume either \(p_1=p_2=\dotsb = p_n\) and
\(a_1=a_2=\dotsb=a_n\) or \(q_1=q_2=\dotsb = q_n\) and
\(b_1=b_2=\dotsb = b_n\).

Condition~\ref{en:contracting-sections6} says, roughly speaking, that
we cannot make~\(n\) smaller by combining the data for any \(i\neq
j\).  This is its only role, and it could be left out.

\begin{theorem}
  \label{the:locally_contracting}
  The groupoid~\(H\) is locally contracting if and only if the
  action~\((M_p,\sigma_{p,q})\) is locally contracting.
\end{theorem}

\begin{proof}
  To simplify notation, we replace~\(X\)
  by~\(X'\) throughout, so that the maps~\(\rg_p\) are all surjective.

  Call a subset~\(U\) of~\(H^0\) good if there are \(V,B\) as in
  Definition~\ref{def:locally_contracting}.  If \(U_1\subseteq U_2\)
  and~\(U_1\) is good, then so is~\(U_2\).  If \(T\subseteq
  H^1\) is a bisection, then \(\rg(T)\) is good if and only if
  \(\s(T)\) is good: given \(V,B\) for \(\s(T)\), then \(TV\) and
  \(TBT^{-1}\) will work for~\(\rg(T)\).  If \(U\subseteq H^0\) is
  an arbitrary open subset, then there are \(p\in P\) and
  \(U_p\subseteq M_p\) such that \(\pi_p^{-1}(U_p)\subseteq U\),
  \(\s_p|_{U_p}\) is injective, and~\(U_p\) is relatively compact
  and open; here we use that subsets of the form \(\pi_p^{-1}(U_p)\)
  for \(U_p\subseteq M_p\) open form a base
  (Lemma~\ref{lem:H_base}), that~\(M_p\) is locally compact, and
  that~\(\s_p\) is a local homeomorphism for each~\(p\)
  (Lemma~\ref{lem:tilde_s_local_homeo}).  Hence there is a bisection
  with range~\(\pi_p^{-1}(U_p)\) and source~\(\pi_1^{-1}(\s_p(U_p))\).
  Thus~\(U\) is good when~\(\pi_1^{-1}(\s_p(U_p))\) is good.  Summing
  up, all non-empty open subsets of~\(H^0\) are good once all
  non-empty subsets of the form~\(\pi_1^{-1}(S)\) with \(S\subseteq X\)
  relatively compact and open are good.

  Assume now that the action \((M_p,\sigma_{p,q})\) is locally
  contracting.  Let \(S\subset X\) be a non-empty, open, relatively
  compact set.  Pick \(n\in \N\) and \(p_i,q_i, a_i,b_i\in P\) and
  subsets \(W_i\subseteq M_{p_i}\times_{\s_{p_i},X,\s_{q_i}} M_{q_i}\) as in
  Definition~\ref{def:locally_contracting_tc}.  Let
  \[
  B_i\defeq \{(x_1 y, p_i q_i^{-1},x_2 y)\mid
  (x_1,x_2)\in W_i,\ y\in H^0,\ \s_{p_i}(x_1)=\pi_1(y)\};
  \]
  here \(x_1 y\) and~\(x_2 y\) are well-defined because
  \(\s_{p_i}(x_1)=\s_{q_i}(x_2) = \pi_1(y)\), which uses that
  \(\s_{p_i}(x_1)=\s_{q_i}(x_2)\) for \((x_1,x_2)\in W_i\).  Since
  \(\tilde{\s}_{p_i} (x_1 y) = y = \tilde{\s}_{q_i} (x_2 y)\), we
  have \(B_i\subseteq H^1_{p_i,q_i}\).
  Condition~\ref{en:contracting-sections1} ensures that the range
  and source maps are open and injective on each~\(B_i\), so these
  are bisections.  For different~\(i\), they have disjoint sources
  and ranges by~\ref{en:contracting-sections3}.  Hence \(B\defeq
  B_1\sqcup B_2\sqcup\dotsb\sqcup B_n\) is a bisection as well.
  Condition~\ref{en:contracting-sections4} gives \(\overline{\rg(B)}
  \subsetneq \s(B)\), and~\ref{en:contracting-sections7}
  gives \(\s(B) \subseteq \pi_1^{-1}(S)\).  Since we assume~\(S\)
  to be relatively compact, the closed subset~\(\overline{\rg(B)}\)
  is compact.  So there is an open subset~\(V\) with
  \(\overline{\rg(B)} \subsetneq V \subseteq \overline{V} \subseteq
  \s(B)\).  Hence \(\pi_1^{-1}(S)\) is good.  This implies
  that~\(H\) is locally contracting.

  Conversely, assume that~\(H\) is locally contracting.  Let
  \(S\subseteq X\) be a relatively compact, non-empty, open subset
  and let \(U\defeq \pi_1^{-1}(S)\).  Then~\(U\) is relatively compact,
  non-empty and open because~\(\pi_1\colon H^0\to X\) is surjective,
  continuous, and proper.  Since~\(H\) is locally contracting, there
  is a bisection \(B\subseteq H^1\) with \(\overline{\rg(B)}\subsetneq
  \s(B)\subseteq U\).  Next we have to analyse this bisection~\(B\)
  locally.  This does not yet use the special feature of~\(B\) and
  gives slightly more, namely, a ``base'' for the inverse semigroup of
  bisections of~\(H\).  By this we mean an inverse subsemigroup closed
  under finite intersections that covers~\(H^1\).  This can be used to
  study actions of~\(H\) on \(\Cst\)\nb-algebras as
  in~\cite{Buss-Meyer:Actions_groupoids}.

  \begin{lemma}
    \label{lem:H_bisections_base}
    Let \(B\subseteq H^1\) be a bisection and \(\eta\in B\).  Then
    there is an open bisection~\(B_\eta\) with \(\eta \in B_\eta
    \subseteq B\) that has the following form:
    \[
    B_\eta\defeq \{(x_1 y, p q^{-1},x_2 y)\mid
    (x_1,x_2)\in W,\ y\in H^0,\ \s_p(x_1)=\pi_1(y)\}
    \]
    for some \(p,q\in P\) and a subset \(W\subseteq
    M_p\times_{\s_p,X,\s_q} M_q\) such that \(\pr_1\colon W\to M_p\)
    and \(\pr_2\colon W\to M_q\) are injective and open.

    If \(p t = p_1\), \(q t = q_1\) and
    \[
    W_1 \defeq \{(x_1 y, x_2 y)\mid
    (x_1,x_2)\in W,\ y\in W_t,\ \s_p(x_1)=\rg_t(y)\}
    \]
    for some \(t\in P\), then the data \((p,q,W)\) and
    \((p_1,q_1,W_1)\) define the same bisection~\(B_\eta\).
  \end{lemma}

  \begin{proof}
    We write~\(B_*\)
    for the homeomorphism \(\s(B)\to\rg(B)\)
    induced by~\(B\);
    this is determined uniquely by \(B_*(\s(h))= \rg(h)\)
    for all \(h\in B\).  Let \(\xi\defeq \s(\eta)\).

    There are \(p,q\in P\)
    with \(\eta\in H^1_{p,q}\)
    because these subsets cover~\(H^1\).
    Since the subsets \(B\)
    and~\(H^1_{p,q}\)
    and the map \(\s\colon H^1\to H^0\)
    are open, \(\s(B\cap H^1_{p,q})\)
    is an open neighbourhood of~\(\xi\).
    Let \(h\in B\cap H^1_{p,q}\).
    Lemma~\ref{lem:tilde_s_local_homeo} shows that there are unique
    \(x_2\in M_{q}\)
    and \(y\in H^0\)
    with \(\s_{q}(x_2) = \pi_1(y)\)
    and \(\s(h) = x_2 y\),
    namely, \(x_2 = \pi_{q}(\s(h))\)
    and \(y= \tilde{\s}_{q} (\s(h))\).
    Similarly, there are unique \(x_1\in M_{p}\)
    and \(y'\in H^0\)
    with \(\s_{p}(x_1) = \pi_1(y')\)
    and \(\rg(h) = x_1 y'\),
    namely, \(x_1 = \pi_{p}(\rg(h))\)
    and \(y' = \tilde{\s}_{p}(\rg(h))\).
    The assumption \(h\in H^1_{p,q}\)
    means exactly that \(y = y'\).
    Thus \(h=(x_1 y,p q^{-1}, x_2 y)\)
    for \(x_1\in M_{p}\),
    \(x_2\in M_{q}\),
    \(y\in H^0\) with \(\s_{p}(x_1) = \s_{q}(x_2) = \pi_1(y)\).

    Since~\(\s_{p}\) is a local homeomorphism, there is an open
    neighbourhood~\(V\) around \(\pi_{p}(B_*\xi) \in M_{p}\) so that
    \(\s_{p}|_{V}\colon V \to X\) is a homeomorphism onto an open
    subset of~\(X\).  The subset
    \[
    V' \defeq \{x \in \s(B\cap H^1_{p,q}) \mid
    \pi_{p}(B_* x) \in V\}
    \]
    is still an open neighbourhood of~\(\xi\).  It contains a
    neighbourhood of~\(\xi\) that belongs to the base in
    Lemma~\ref{lem:H_base}.  This gives us \(q_2\in P\) and
    \(V''\subset M_{q_2}\) with \(\xi\in \pi_{q_2}^{-1}(V'')\subseteq
    V'\).  Since~\(\s_{q_2}\) is a local homeomorphism as well, we may
    further shrink~\(V''\) so that~\(\s_{q_2}|_{V''}\) becomes
    injective; we assume this.  The first Ore condition gives us
    \(a,b\in P\) with \(q a = q_2 b\).  Let \(p'\defeq p a\) and
    \(q'\defeq q a = q_2 b\).  Let
    \[
    W\defeq \{ (\pi_{p'}(\rg(h)),\pi_{q'}(\s(h))) \mid
    h\in B,\ \pi_{q_2}(\s(h))\in V''\}.
    \]
    We claim that \(p',q',W\) have the asserted properties.

    Let \(h\in B\) satisfy \(\pi_{q_2}(\s(h))\in V''\).  Write
    \(\s(h)= x_2 y\), \(\rg(h)=x_1 y\) with \(x_1\in M_p\), \(x_2\in
    M_q\), \(y\in H^0\) and \(\s_p(x_1)=\s_q(x_2)=\pi_1(y)\) as above.
    Then \(x_1\in V\), so~\(x_1\) is the unique point in~\(V\) with
    \(\s_p(x_1) = \s_q(x_2)\).  Now write \(y= y_1 y_2\) with \(y_1\in
    M_a\), \(y_2\in H^0\), \(\s_a(y_1)=\pi_1(y_2)\).  Then
    \(\s(h) = x_1 y_1 y_2\) and \(\rg(h) = x_2 y_1 y_2\).  The point
    \(x_1 y_1\in M_{p a} = M_{p'}\) is the unique one in
    \(\rg_{p,a}^{-1}(V)\) with \(\s_{p,a}(x_1 y_1) = y_1\).  This
    shows that \(\pi_{p'}(\rg(h))\) is determined by
    \(\pi_{q'}(\s(h))\) and that the map that takes
    \(\pi_{q'}(\s(h))\) to \(\pi_{q'}(\rg(h))\) is continuous.  Hence
    the second coordinate projection \(\pr_2\colon W\to M_{q'}\) is
    injective and open.  Since we assumed~\(\s_{q_2}\) to be injective
    on~\(V''\), the same holds for \(\pr_1\colon W\to M_{p'}\).

    If \((x_1 y, p q^{-1},x_2 y)\in B_\eta\), then \(\tilde{\s}_{p}
    (x_1 y) = y = \tilde{\s}_{q}(x_2 y)\), so \((x_1 y, p q^{-1},x_2
    y) \in H^1_{p,q}\) and \(B_\eta\subseteq H^1\).  Since
    \(\pr_1\colon W\to M_{p'}\) and \(\pr_2\colon W\to M_{q'}\) are
    injective and open, so are the maps \(\s,\rg\colon B_\eta\to
    H^0\) because~\(y_2\) is common to both source and range and
    \(x_1 y_1\) and \(x_2 y_1\) determine each other uniquely and
    continuously.  Thus~\(B_\eta\) is a bisection.  It is clear from
    the construction that \(\eta\in B_\eta\subseteq B\).  The last
    statement about different data giving the same~\(B_\eta\) is
    implicitly shown above.
  \end{proof}

  Fix \(\xi_0\in \s(B)\setminus \overline{\rg(B)}\).  The subset
  \(\overline{\rg(B)}\cup\{\xi_0\} \subseteq \s(B) \subseteq U =
  \pi_1^{-1}(S)\) is closed and contained in the relatively compact
  subset~\(\pi_1^{-1}(S)\), so it is compact.  For each \(x\in
  \overline{\rg(B)}\cup\{\xi_0\} \subseteq \s(B)\), there is a unique
  \(\eta\in B\) with \(\s(\eta)=x\).  Let \(B_\eta\subseteq B\) be
  some bisection as in Lemma~\ref{lem:H_bisections_base}.  The open
  subsets~\(\s(B_\eta)\) for these chosen bisections cover
  \(\overline{\rg(B)}\cup\{\xi_0\}\).  By compactness, this only needs
  finitely many of them, say, \(B_1,\dotsc,B_\ell\).  Let \(B'\defeq
  B_1\cup B_2\dotsc \cup B_\ell\subseteq B\).  This is still an open
  bisection, and it satisfies \(\overline{\rg(B')}\subseteq
  \overline{\rg(B)} \subseteq \s(B')\subseteq \s(B) \subseteq U\).
  Since \(\xi_0\in \s(B')\setminus \overline{\rg(B)}\), we get
  \(\overline{\rg(B')} \subsetneq \s(B')\).  Hence we may
  replace~\(B\) by~\(B'\).

  Each~\(B_i\) is constructed from certain \(p_i,q_i\in P\) and
  \(W_i\subseteq M_{p_i}\times M_{q_i}\) as in
  Lemma~\ref{lem:H_bisections_base}.  The
  Ore conditions also provide \(e,a_i,b_i\in P\) with \(p_i a_i = e =
  q_i b_i\) for all~\(i\).  Our construction so far already achieves
  the most crucial properties \ref{en:contracting-sections5},
  \ref{en:contracting-sections1}, \ref{en:contracting-sections4} and
  \ref{en:contracting-sections7}; the last one follows from \(\s(B')
  \subseteq \s(B)\subseteq \pi_1^{-1}(S)\).  So far, however, the subsets
  \(\s(B_i)\) and \(\rg(B_i)\) may still overlap, and some among the
  group elements~\(p_i q_i^{-1}\) may well be equal.  We now rectify
  this.

  Let \(p_i q_i^{-1} = p_j q_j^{-1}\) for some \(i\neq j\).  We may
  replace \((p_i,q_i)\) by \((p_i c,q_i c)\) for \(c\in P\) using the
  last statement in Lemma~\ref{lem:H_bisections_base}.  By this, we
  can arrange that \(p_i = p_j\) and \(q_i = q_j\), which we now
  assume.  Then
  \[
  B_i\cup B_j =
  \{(x_1 y, p q^{-1},x_2 y)\mid
  (x_1,x_2)\in W_i\cup W_j,\ y\in H^0,\ \s_p(x_1)=\pi_1(y)\}.
  \]
  We know that \(B_i\cup B_j\) is a bisection because it is contained
  in the bisection~\(B\).  Since \(X=X'\), this implies that the
  coordinate projections remain injective on \(W_i\cup W_j\).  They
  are open because this holds locally on \(W_i\) and~\(W_j\).
  Hence we may simply merge the data \((p_i,q_i,W_i)\) and
  \((p_j,q_j,W_j)\) into one piece, without changing the
  bisection~\(B'\).
  We go on merging part of our data until all group elements
  \(g_i \defeq p_i q_i^{-1}\)
  are different.  This achieves~\ref{en:contracting-sections6}.  By
  construction, \(B_i\subseteq H^1_{p_i,q_i}\subseteq H^1_{g_i}\),
  and these subsets are disjoint for different~\(g_i\).
  Hence the \(B_i\)
  are disjoint bisections.  Since their union remains a bisection,
  their ranges and sources must be disjoint, which
  is~\ref{en:contracting-sections3}.  Hence all the conditions in
  Definition~\ref{def:locally_contracting_tc} hold.
\end{proof}

Our condition for local contractivity is rather complicated because it
is necessary and sufficient.  If we restricted to \(n=1\)
in Definition~\ref{def:locally_contracting_tc}, the condition would no
longer be necessary, and it would depend on the \(G\)\nb-grading
on~\(H\),
so it would not be a property of the groupoid~\(H\)
alone.  Nevertheless, the case \(n=1\)
of our criterion gives a useful sufficient condition:

\begin{corollary}
  \label{cor:locally_contracting_1}
  The groupoid~\(H\)
  is locally contracting if, for any relatively compact, open subset
  \(S\subseteq X'\),
  there are \(p,q,a,b\in P\)
  with \(p a = q b\)
  and a subset \(W\subseteq M'_p\times_{\s_p,X',\s_q} M'_q\)
  such that the projections \(\pr_1\colon W\to M'_p\)
  and \(\pr_2\colon W\to M'_q\)
  are injective and open, \(\rg_q(\pr_2(W))\subseteq S\),
  and \(\overline{\pr_1(W)}\cdot M'_a\subsetneq \pr_2(W)\cdot M'_b\)
  as subsets of \(M'_{p a} = M'_{q b}\).
\end{corollary}

\begin{proof}
  Besides restricting to \(n=1\), we also have rewritten
  \[
  \overline{\pr_1(W)\cdot M'_a} = \overline{\pr_1(W)}\cdot M'_a,
  \]
  using the closure of~\(\pr_1(W)\)
  in~\(M'_p\).
  This is because
  \(\pr_1(W)\cdot M'_a = \pr_1(W)\times_{\s_p,X',\s_q} M'_a \subseteq
  M'_p\times_{\s_p,X',\s_q} M'_a\),
  and for such a subset the closure operation clearly works on the
  first entry only.
\end{proof}

Specialising further, we may assume~\(W\)
to have the form \(W= W'_p \times_{\s_p,X',\s_q} W'_q\)
for open subsets \(W_p\subseteq M'_p\)
and \(W_q\subseteq M'_q\);
if~\(W\)
has this form, then we may choose \(W_p\)
and~\(W_q\)
minimal given~\(W\)
by taking \(W_p=\pr_1(W)\)
and \(W_q=\pr_2(W)\).
Then \(\s_p(W_p) = \s_q(W_q)\)
because \(\s_p\circ\pr_1=\s_q\circ\pr_2\) on~\(W\).

\begin{lemma}
  The maps \(\pr_1|_W\)
  and~\(\pr_2|_W\)
  are injective and open if and only if \(W_p\subseteq M'_p\)
  and \(W_q\subseteq M'_q\)
  are open and the restrictions \(\s_p|_{W_p}\)
  and~\(\s_q|_{W_q}\) are injective.
\end{lemma}

\begin{proof}
  Assume first that \(\pr_1|_W\)
  and~\(\pr_2|_W\)
  are injective and open.  Then \(W_p= \pr_1(W)\)
  and \(W_q = \pr_2(W)\)
  are open.  If \(x\in W_p\),
  then any \(y\in W_q\)
  with \(\s_q(y) = \s_p(x)\)
  gives a point \((x,y)\in W\)
  with \(\pr_1(x,y)=x\).
  Hence~\(\s_q|_{W_q}\)
  is injective if~\(\pr_1|_W\)
  is injective.  Conversely, assume that \(W_p\subseteq M'_p\)
  and \(W_q\subseteq M'_q\)
  are open subsets and the restrictions \(\s_p|_{W_p}\)
  and~\(\s_q|_{W_q}\)
  are injective.  Since \(\s_p\)
  and~\(\s_q\)
  are local homeomorphisms, so are their restrictions to the open
  subsets \(W_p\)
  and~\(W_q\).
  Being also injective and continuous, they are homeomorphisms onto
  \(\s_p(W_p) = \s_q(W_q)\).
  Hence the coordinate projections on~\(W\)
  are homeomorphisms onto \(W_p\)
  and~\(W_q\),
  respectively.  Since these subsets are open, the coordinate
  projections are injective and open.
\end{proof}

\begin{corollary}
  \label{cor:locally_contracting_2}
  The groupoid~\(H\)
  is locally contracting if, for any relatively compact, open subset
  \(S\subseteq X'\),
  there are \(p,q,a,b\in P\)
  with \(p a = q b\)
  and open subsets \(W_p\subseteq M'_p\)
  and \(W_q\subseteq M'_q\)
  such that \(\s_p(W_p) = \s_q(W_q)\),
  the restrictions \(\s_p|_{M'_p}\)
  and \(\s_q|_{M'_q}\)
  are injective, \(\rg_q(W_q)\subseteq S\),
  and \(\overline{W_p} \cdot M'_a \subsetneq W_q \cdot M'_b\)
  as subsets of \(M'_{p a} = M'_{q b}\).
\end{corollary}

\begin{proof}
  Specialise Corollary~\ref{cor:locally_contracting_1} to the case
  where \(W=W_p\times_{\s_p,X',\s_q} W_q\).
\end{proof}

The role of \(a,b\in P\)
is only as padding to be able to compare \(\overline{W_p}\)
and~\(W_q\).
Since we restricted to possible histories, so that all range maps are
surjective, we have
\[
\overline{W_p} \cdot M'_a \subsetneq W_q \cdot M'_b
\quad\iff\quad
\overline{W_p} \cdot M'_{a t} \subsetneq W_q \cdot M'_{b t}
\]
for any \(t\in P\).
This shows that the choice of \(a,b\)
does not matter: if the criterion holds for one choice, it holds for
all choices.  The same is true, of course, for
Corollary~\ref{cor:locally_contracting_2}, and an analogous statement
holds for Theorem~\ref{the:locally_contracting}.

If~\(P\)
is a lattice-ordered Ore monoid, we therefore get equivalent criteria
if we take \(a= p^{-1}(p\vee q)\)
and \(b= q^{-1}(p\vee q)\)
in Corollary~\ref{cor:locally_contracting_1} or
Corollary~\ref{cor:locally_contracting_2}.  We write down the variant
of Corollary~\ref{cor:locally_contracting_2}:

\begin{corollary}
  \label{cor:locally_contracting_3}
  Assume that~\(P\)
  is Ore and lattice-ordered.  The groupoid~\(H\)
  is locally contracting if, for any relatively compact, open subset
  \(S\subseteq X'\),
  there are \(p,q\in P\)
  and open subsets \(W_p\subseteq M'_p\)
  and \(W_q\subseteq M'_q\)
  such that \(\s_p(W_p) = \s_q(W_q)\),
  the restrictions \(\s_p|_{M'_p}\)
  and \(\s_q|_{M'_q}\)
  are injective, \(\rg_q(W_q)\subseteq S\),
  and \(\overline{W_p} \cdot M'_a \subsetneq W_q \cdot M'_b\)
  as subsets of \(M'_{p a} = M'_{q b}\),
  where \(a= p^{-1}(p\vee q)\) and \(b= q^{-1}(p\vee q)\).
\end{corollary}

The inclusion \(\overline{W_p} \cdot M'_a \subseteq W_q \cdot M'_b\)
for \(a= p^{-1}(p\vee q)\)
and \(b= q^{-1}(p\vee q)\)
means that for any \(x_p\in \overline{W_p}\)
there is \(x_q\in W_q\)
so that \(x_p\)
and~\(x_q\)
have a common extension; the minimal such extension lives in
\(M'_{p a} = M'_{q b}\).
The meaning of \(\overline{W_p} \cdot M'_a \neq W_q \cdot M'_b\)
is that some \(y\in W_q\)
has no common extension with any element of~\(\overline{M_p}\).

Corollary~\ref{cor:locally_contracting_3} involves the same
ingredients as the sufficient condition for the boundary path groupoid
of a higher-rank topological graph to be locally contracting in
\cite{Renault-Sims-Williams-Yeend:Uniqueness}*{Proposition 5.8}.  The
proof of \cite{Renault-Sims-Williams-Yeend:Uniqueness}*{Lemma 5.9}
shows that our sufficient criterion is more general than the one
in~\cite{Renault-Sims-Williams-Yeend:Uniqueness}.  We are, however,
not aware of any applications of the criterion
in~\cite{Renault-Sims-Williams-Yeend:Uniqueness}.

Our criteria have the advantage that we understand very well which
assumptions we have imposed on the contracting bisections for~\(H\).
In Corollary~\ref{cor:locally_contracting_1}, the only assumption is
that the bisection is contained in~\(H^1_g\)
for some \(g\in G\).
In Corollaries \ref{cor:locally_contracting_2}
and~\ref{cor:locally_contracting_3}, we assume this and that the
bisection has a product form.


\begin{bibdiv}
  \begin{biblist}
\bib{Abadie:Tensor}{article}{
  author={Abadie, Fernando},
  title={Tensor products of Fell bundles over discrete groups},
  status={eprint},
  note={\arxiv {funct-an/9712006}},
  date={1997},
}

\bib{Albandik-Meyer:Colimits}{article}{
  author={Albandik, Suliman},
  author={Meyer, Ralf},
  title={Colimits in the correspondence bicategory},
  date={2015},
  status={accepted},
  journal={M\"unster J. Math.},
  issn={1867-5778},
  note={\arxiv {1502.07771}},
}

\bib{Anantharaman-Delaroche:Purely_infinite}{article}{
  author={Anantharaman-Delaroche, Claire},
  title={Purely infinite $C^*$\nobreakdash -algebras arising from dynamical systems},
  journal={Bull. Soc. Math. France},
  volume={125},
  date={1997},
  number={2},
  pages={199--225},
  issn={0037-9484},
  review={\MRref {1478030}{99i:46051}},
  eprint={http://www.numdam.org/item?id=BSMF_1997__125_2_199_0},
}

\bib{Renault_AnantharamanDelaroche:Amenable_groupoids}{book}{
  author={Anantharaman-Delaroche, Claire},
  author={Renault, Jean},
  title={Amenable groupoids},
  series={Monographies de L'Enseignement Math\'ematique},
  volume={36},
  publisher={L'Enseignement Math\'ematique, Geneva},
  year={2000},
  pages={196},
  isbn={2-940264-01-5},
  review={\MRref {1799683}{2001m:22005}},
}

\bib{Andreka-Nemethi:Limits}{article}{
  author={Andr\'eka, H.},
  author={N\'emeti, I.},
  title={Direct limits and filtered colimits are strongly equivalent in all categories},
  conference={ title={Universal algebra and applications}, address={Warsaw}, date={1978}, },
  book={ series={Banach Center Publ.}, volume={9}, publisher={PWN, Warsaw}, },
  date={1982},
  pages={75--88},
  review={\MRref {738804}{85d:18003}},
  eprint={http://matwbn.icm.edu.pl/ksiazki/bcp/bcp9/bcp919.pdf},
}

\bib{Arveson:Continuous_Fock}{article}{
  author={Arveson, William},
  title={Continuous analogues of Fock space},
  journal={Mem. Amer. Math. Soc.},
  volume={80},
  date={1989},
  number={409},
  pages={iv+66},
  issn={0065-9266},
  review={\MRref {987590}{90f:47061}},
  doi={10.1090/memo/0409},
}

\bib{Arzumanian-Renault:Pseudogroups}{article}{
  author={Arzumanian, Victor},
  author={Renault, Jean},
  title={Examples of pseudogroups and their $C^*$\nobreakdash -algebras},
  conference={ title={Operator algebras and quantum field theory}, address={Rome}, date={1996}, },
  book={ publisher={Int. Press, Cambridge, MA}, },
  date={1997},
  pages={93--104},
  review={\MRref {1491110}{99a:46101}},
}

\bib{Bost-Connes:Hecke}{article}{
  author={Bost, Jean-Beno\^\i t},
  author={Connes, Alain},
  title={Hecke algebras, type III factors and phase transitions with spontaneous symmetry breaking in number theory},
  journal={Selecta Math. (N.S.)},
  volume={1},
  date={1995},
  number={3},
  pages={411--457},
  issn={1022-1824},
  review={\MRref {1366621}{96m:46112}},
  doi={10.1007/BF01589495},
}

\bib{Brown-Clark-Farthing-Sims:Simplicity}{article}{
  author={Brown, Jonathan Henry},
  author={Clark, Lisa Orloff},
  author={Farthing, Cynthia},
  author={Sims, Aidan},
  title={Simplicity of algebras associated to \'etale groupoids},
  journal={Semigroup Forum},
  volume={88},
  date={2014},
  number={2},
  pages={433--452},
  issn={0037-1912},
  review={\MRref {3189105}{}},
  doi={10.1007/s00233-013-9546-z},
}

\bib{Brownlowe-Raeburn:Exel-Larsen}{article}{
  author={Brownlowe, Nathan},
  author={Raeburn, Iain},
  title={Two families of Exel--Larsen crossed products},
  journal={J. Math. Anal. Appl.},
  volume={398},
  date={2013},
  number={1},
  pages={68--79},
  issn={0022-247X},
  review={\MRref {2984316}{}},
  doi={10.1016/j.jmaa.2012.08.026},
}

\bib{Buss-Meyer:Actions_groupoids}{article}{
  author={Buss, Alcides},
  author={Meyer, Ralf},
  title={Inverse semigroup actions on groupoids},
  status={accepted},
  note={\arxiv {1410.2051}},
  date={2014},
  issn={0035-7596},
  journal={Rocky Mountain J. Math.},
}

\bib{Buss-Meyer-Zhu:Higher_twisted}{article}{
  author={Buss, Alcides},
  author={Meyer, Ralf},
  author={Zhu, {Ch}enchang},
  title={A higher category approach to twisted actions on \(\textup C^*\)\nobreakdash -algebras},
  journal={Proc. Edinb. Math. Soc. (2)},
  date={2013},
  volume={56},
  number={2},
  pages={387--426},
  issn={0013-0915},
  doi={10.1017/S0013091512000259},
  review={\MRref {3056650}{}},
}

\bib{Clifford-Preston:Semigroups_I}{book}{
  author={Clifford, A. H.},
  author={Preston, G. B.},
  title={The algebraic theory of semigroups. Vol. I},
  series={Mathematical Surveys, No. 7},
  publisher={Amer. Math. Soc.},
  place={Providence, RI},
  date={1961},
  pages={xv+224},
  review={\MRref {0132791}{24 \#A2627}},
  doi={10.1090/surv/007.1},
}

\bib{Connes-Marcolli-Ramachadran:KMS}{article}{
  author={Connes, Alain},
  author={Marcolli, Matilde},
  author={Ramachandran, Niranjan},
  title={KMS states and complex multiplication},
  journal={Selecta Math. (N.S.)},
  volume={11},
  date={2005},
  number={3-4},
  pages={325--347},
  issn={1022-1824},
  review={\MRref {2215258}{2007a:11078}},
  doi={10.1007/s00029-005-0013-x},
}

\bib{Cuntz:axb_N}{article}{
  author={Cuntz, Joachim},
  title={$C^*$\nobreakdash -algebras associated with the $ax+b$-semigroup over $\mathbb {N}$},
  conference={ title={$K$\nobreakdash -theory and noncommutative geometry}, },
  book={ series={EMS Ser. Congr. Rep.}, publisher={Eur. Math. Soc., Z\"urich}, },
  date={2008},
  pages={201--215},
  review={\MRref {2513338}{2010i:46086}},
  doi={10.4171/060-1/8},
}

\bib{Cuntz-Echterhoff-Li:K-theory}{article}{
  author={Cuntz, Joachim},
  author={Echterhoff, Siegfried},
  author={Li, Xin},
  title={On the K-theory of the C*-algebra generated by the left regular representation of an Ore semigroup},
  journal={J. Eur. Math. Soc. (JEMS)},
  volume={17},
  date={2015},
  number={3},
  pages={645--687},
  issn={1435-9855},
  review={\MRref {3323201}{}},
  doi={10.4171/JEMS/513},
}

\bib{Cuntz-Vershik:Endomorphisms}{article}{
  author={Cuntz, Joachim},
  author={Vershik, Anatoly},
  title={$C^*$\nobreakdash -algebras associated with endomorphisms and polymorphisms of compact abelian groups},
  journal={Comm. Math. Phys.},
  volume={321},
  date={2013},
  number={1},
  pages={157--179},
  issn={0010-3616},
  review={\MRref {3089668}{}},
  doi={10.1007/s00220-012-1647-0},
}

\bib{Dadarlat:Cstar_vector}{article}{
  author={D\u {a}d\u {a}rlat, Marius},
  title={The $C^*$\nobreakdash -algebra of a vector bundle},
  journal={J. Reine Angew. Math.},
  volume={670},
  date={2012},
  pages={121--143},
  issn={0075-4102},
  review={\MRref {2982694}{}},
  doi={10.1515/CRELLE.2011.145},
}

\bib{Deaconu:Continuous_graphs}{article}{
  author={Deaconu, Valentin},
  title={Continuous graphs and $C^*$\nobreakdash -algebras},
  conference={ title={Operator theoretical methods}, address={Timi\c soara}, date={1998}, },
  book={ publisher={Theta Found., Bucharest}, },
  date={2000},
  pages={137--149},
  review={\MRref {1770320}{2001g:46123}},
}

\bib{Deaconu:Iterating}{article}{
  author={Deaconu, Valentin},
  title={Iterating the Pimsner construction},
  journal={New York J. Math.},
  volume={13},
  date={2007},
  pages={199--213},
  issn={1076-9803},
  review={\MRref {2336239}{2008g:46093}},
  eprint={http://nyjm.albany.edu/j/2007/13_199.html},
}

\bib{Dinh:Discrete_product}{article}{
  author={Dinh, Hung T.},
  title={Discrete product systems and their $C^*$\nobreakdash -algebras},
  journal={J. Funct. Anal.},
  volume={102},
  date={1991},
  number={1},
  pages={1--34},
  issn={0022-1236},
  review={\MRref {1138835}{93d:46097}},
  doi={10.1016/0022-1236(91)90133-P},
}

\bib{Doplicher-Roberts:Duals}{article}{
  author={Doplicher, Sergio},
  author={Roberts, John E.},
  title={Duals of compact Lie groups realized in the Cuntz algebras and their actions on $C^*$\nobreakdash -algebras},
  journal={J. Funct. Anal.},
  volume={74},
  date={1987},
  number={1},
  pages={96--120},
  issn={0022-1236},
  review={\MRref {901232}{89a:22011}},
  doi={10.1016/0022-1236(87)90040-1},
}

\bib{Exel:Partial_actions}{article}{
  author={Exel, Ruy},
  title={Partial actions of groups and actions of inverse semigroups},
  journal={Proc. Amer. Math. Soc.},
  volume={126},
  date={1998},
  number={12},
  pages={3481--3494},
  issn={0002-9939},
  review={\MRref {1469405}{99b:46102}},
  doi={10.1090/S0002-9939-98-04575-4},
}

\bib{Exel:Amenability}{article}{
  author={Exel, Ruy},
  title={Amenability for Fell bundles},
  journal={J. Reine Angew. Math.},
  volume={492},
  date={1997},
  pages={41--73},
  issn={0075-4102},
  review={\MRref {1488064}{99a:46131}},
  doi={10.1515/crll.1997.492.41},
}

\bib{Exel:Exact_groups_Fell}{article}{
  author={Exel, Ruy},
  title={Exact groups and Fell bundles},
  journal={Math. Ann.},
  volume={323},
  date={2002},
  number={2},
  pages={259--266},
  issn={0025-5831},
  review={\MRref {1913042}{2003e:46095}},
  doi={10.1007/s002080200295},
}

\bib{Exel:Look_endomorphism}{article}{
  author={Exel, Ruy},
  title={A new look at the crossed-product of a $C^*$\nobreakdash -algebra by an endomorphism},
  journal={Ergodic Theory Dynam. Systems},
  volume={23},
  date={2003},
  number={6},
  pages={1733--1750},
  issn={0143-3857},
  review={\MRref {2032486}{2004k:46119}},
  doi={10.1017/S0143385702001797},
}

\bib{Exel:Interactions}{article}{
  author={Exel, Ruy},
  title={Interactions},
  journal={J. Funct. Anal.},
  volume={244},
  date={2007},
  number={1},
  pages={26--62},
  issn={0022-1236},
  review={\MRref {2294474}{2007k:46092}},
  doi={10.1016/j.jfa.2006.03.023},
}

\bib{Exel:Inverse_combinatorial}{article}{
  author={Exel, Ruy},
  title={Inverse semigroups and combinatorial $C^*$\nobreakdash -algebras},
  journal={Bull. Braz. Math. Soc. (N.S.)},
  volume={39},
  date={2008},
  number={2},
  pages={191--313},
  issn={1678-7544},
  review={\MRref {2419901}{2009b:46115}},
  doi={10.1007/s00574-008-0080-7},
}

\bib{Exel:New_look}{article}{
  author={Exel, Ruy},
  title={A new look at the crossed product of a $C^*$\nobreakdash -algebra by a semigroup of endomorphisms},
  journal={Ergodic Theory Dynam. Systems},
  volume={28},
  date={2008},
  number={3},
  pages={749--789},
  issn={0143-3857},
  review={\MRref {2422015}{2010b:46144}},
  doi={10.1017/S0143385707000302},
}

\bib{Exel:Blend_Alloys}{article}{
  author={Exel, Ruy},
  title={Blends and alloys},
  journal={C. R. Math. Acad. Sci. Soc. R. Can.},
  volume={35},
  date={2013},
  number={3},
  pages={77--113},
  issn={0706-1994},
  review={\MRref {3136103}{}},
}

\bib{Exel-Renault:Semigroups_interaction}{article}{
  author={Exel, Ruy},
  author={Renault, Jean},
  title={Semigroups of local homeomorphisms and interaction groups},
  journal={Ergodic Theory Dynam. Systems},
  volume={27},
  date={2007},
  number={6},
  pages={1737--1771},
  issn={0143-3857},
  review={\MRref {2371594}{2009k:46124}},
  doi={10.1017/S0143385707000193},
}

\bib{Farthing-Patani-Willis:Crossed-product}{article}{
  author={Farthing, Cynthia},
  author={Patani, Nura},
  author={Willis, Paulette N.},
  title={The crossed-product structure of $C^*$\nobreakdash -algebras arising from topological dynamical systems},
  journal={J. Operator Theory},
  volume={70},
  date={2013},
  number={1},
  pages={191--210},
  issn={0379-4024},
  review={\MRref {3085824}{}},
  doi={10.7900/jot.2011may31.1924},
}

\bib{Fowler:Product_systems}{article}{
  author={Fowler, Neal J.},
  title={Discrete product systems of Hilbert bimodules},
  journal={Pacific J. Math.},
  volume={204},
  date={2002},
  number={2},
  pages={335--375},
  issn={0030-8730},
  review={\MRref {1907896}{2003g:46070}},
  doi={10.2140/pjm.2002.204.335},
}

\bib{Fowler-Sims:Product_Artin}{article}{
  author={Fowler, Neal J.},
  author={Sims, Aidan},
  title={Product systems over right-angled Artin semigroups},
  journal={Trans. Amer. Math. Soc.},
  volume={354},
  date={2002},
  number={4},
  pages={1487--1509},
  issn={0002-9947},
  review={\MRref {1873016}{2002j:18006}},
  doi={10.1090/S0002-9947-01-02911-7},
}

\bib{Haagerup-Thorbjornsen:Random_K}{article}{
  author={Haagerup, Uffe},
  author={Thorbj\o rnsen, Steen},
  title={Random matrices and $K$\nobreakdash -theory for exact $C^*$\nobreakdash -algebras},
  journal={Doc. Math.},
  volume={4},
  date={1999},
  pages={341--450},
  issn={1431-0635},
  review={\MRref {1710376}{2000g:46092}},
  eprint={http://www.math.uni-bielefeld.de/documenta/vol-04/12.html},
}

\bib{Higson-Kasparov:E_and_KK}{article}{
  author={Higson, Nigel},
  author={Kasparov, Gennadi},
  title={$E$\nobreakdash -theory and $KK$\nobreakdash -theory for groups which act properly and isometrically on Hilbert space},
  journal={Invent. Math.},
  volume={144},
  date={2001},
  number={1},
  pages={23--74},
  issn={0020-9910},
  review={\MRref {1821144}{2002k:19005}},
  doi={10.1007/s002220000118},
}

\bib{anHuef-Laca-Raeburn-Sims:KMS_finite}{article}{
  author={an Huef, Astrid},
  author={Laca, Marcelo},
  author={Raeburn, Iain},
  author={Sims, Aidan},
  title={KMS states on the $C^*$\nobreakdash -algebras of finite graphs},
  journal={J. Math. Anal. Appl.},
  volume={405},
  date={2013},
  number={2},
  pages={388--399},
  issn={0022-247X},
  review={\MRref {3061018}{}},
  doi={10.1016/j.jmaa.2013.03.055},
}

\bib{anHuef-Laca-Raeburn-Sims:KMS_higher}{article}{
  author={an Huef, Astrid},
  author={Laca, Marcelo},
  author={Raeburn, Iain},
  author={Sims, Aidan},
  title={KMS states on $C^*$\nobreakdash -algebras associated to higher-rank graphs},
  journal={J. Funct. Anal.},
  volume={266},
  date={2014},
  number={1},
  pages={265--283},
  issn={0022-1236},
  review={\MRref {3121730}{}},
  doi={10.1016/j.jfa.2013.09.016},
}

\bib{Ionescu-Muhly-Vega:Markov}{article}{
  author={Ionescu, Marius},
  author={Muhly, Paul S.},
  author={Vega, Victor},
  title={Markov operators and $C^*$\nobreakdash -algebras},
  journal={Houston J. Math.},
  volume={38},
  date={2012},
  number={3},
  pages={775--798},
  issn={0362-1588},
  review={\MRref {2970658}{}},
  eprint={http://www.math.uh.edu/~hjm/restricted/pdf38(3)/08ionescu.pdf},
}

\bib{Kang-Pask:Aperiodicity}{article}{
  author={Kang, Sooran},
  author={Pask, David},
  title={Aperiodicity and primitive ideals of row-finite $k$\nobreakdash -graphs},
  journal={Internat. J. Math.},
  volume={25},
  date={2014},
  number={3},
  pages={1450022, 25},
  issn={0129-167X},
  review={\MRref {3189779}{}},
  doi={10.1142/S0129167X14500220},
}

\bib{Katsura:from_correspondences}{article}{
  author={Katsura, Takeshi},
  title={A construction of $C^*$\nobreakdash -algebras from $C^*$\nobreakdash -correspondences},
  conference={ title={Advances in quantum dynamics}, address={South Hadley, MA}, date={2002}, },
  book={ series={Contemp. Math.}, volume={335}, publisher={Amer. Math. Soc., Providence, RI}, },
  date={2003},
  pages={173--182},
  review={\MRref {2029622}{2005k:46131}},
  doi={10.1090/conm/335},
}

\bib{Katsura:Cstar_correspondences}{article}{
  author={Katsura, Takeshi},
  title={On $C^*$\nobreakdash -algebras associated with $C^*$\nobreakdash -correspondences},
  journal={J. Funct. Anal.},
  volume={217},
  date={2004},
  number={2},
  pages={366--401},
  issn={0022-1236},
  review={\MRref {2102572}{2005e:46099}},
  doi={10.1016/j.jfa.2004.03.010},
}

\bib{Katsura:class_I}{article}{
  author={Katsura, Takeshi},
  title={A class of $C^*$\nobreakdash -algebras generalizing both graph algebras and homeomorphism $C^*$\nobreakdash -algebras. I. Fundamental results},
  journal={Trans. Amer. Math. Soc.},
  volume={356},
  date={2004},
  number={11},
  pages={4287--4322},
  issn={0002-9947},
  review={\MRref {2067120}{2005b:46119}},
  doi={10.1090/S0002-9947-04-03636-0},
}

\bib{Katsura:class_II}{article}{
  author={Katsura, Takeshi},
  title={A class of $C^*$\nobreakdash -algebras generalizing both graph algebras and homeomorphism $C^*$\nobreakdash -algebras. II. Examples},
  journal={Internat. J. Math.},
  volume={17},
  date={2006},
  number={7},
  pages={791--833},
  issn={0129-167X},
  review={\MRref {2253144}{2007e:46051}},
  doi={10.1142/S0129167X06003722},
}

\bib{Katsura:class_III}{article}{
  author={Katsura, Takeshi},
  title={A class of $C^*$\nobreakdash -algebras generalizing both graph algebras and homeomorphism $C^*$\nobreakdash -algebras. III. Ideal structures},
  journal={Ergodic Theory Dynam. Systems},
  volume={26},
  date={2006},
  number={6},
  pages={1805--1854},
  issn={0143-3857},
  review={\MRref {2279267}{2007j:46090}},
  doi={10.1017/S0143385706000320},
}

\bib{Kishimoto-Kumjian:Crossed_Cuntz}{article}{
  author={Kishimoto, Akitaka},
  author={Kumjian, Alexander},
  title={Crossed products of Cuntz algebras by quasi-free automorphisms},
  conference={ title={Operator algebras and their applications}, place={Waterloo, ON}, date={1994/1995}, },
  book={ series={Fields Inst. Commun.}, volume={13}, publisher={Amer. Math. Soc.}, place={Providence, RI}, },
  date={1997},
  pages={173--192},
  review={\MRref {1424962}{98h:46076}},
  eprint={http://wolfweb.unr.edu/~alex/pub/oncross.pdf},
}

\bib{Kumjian:Diagonals}{article}{
  author={Kumjian, Alexander},
  title={On $C^*$\nobreakdash -diagonals},
  journal={Canad. J. Math.},
  volume={38},
  date={1986},
  number={4},
  pages={969--1008},
  issn={0008-414X},
  review={\MRref {854149}{88a:46060}},
  doi={10.4153/CJM-1986-048-0},
}

\bib{Kumjian-Pask-Raeburn-Renault:Graphs}{article}{
  author={Kumjian, Alex},
  author={Pask, David},
  author={Raeburn, Iain},
  author={Renault, Jean},
  title={Graphs, groupoids, and Cuntz--Krieger algebras},
  journal={J. Funct. Anal.},
  volume={144},
  date={1997},
  number={2},
  pages={505--541},
  issn={0022-1236},
  review={\MRref {1432596}{98g:46083}},
  doi={10.1006/jfan.1996.3001},
}

\bib{Kwasniewski-Szymanski:Ore}{article}{
  author={Kwa\'sniewski, Bartosz Kosma},
  author={Szyma\'nski, Wojciech},
  title={Topological aperiodicity for product systems over semigroups of Ore type},
  date={2013},
  status={eprint},
  note={\arxiv {1312.7472}},
}

\bib{Laca:Endomorphisms_back}{article}{
  author={Laca, Marcelo},
  title={From endomorphisms to automorphisms and back: dilations and full corners},
  journal={J. London Math. Soc. (2)},
  volume={61},
  date={2000},
  number={3},
  pages={893--904},
  issn={0024-6107},
  review={\MRref {1766113}{2002a:46094}},
  doi={10.1112/S0024610799008492},
}

\bib{Larsen:Crossed_abelian}{article}{
  author={Larsen, Nadia S.},
  title={Crossed products by abelian semigroups via transfer operators},
  journal={Ergodic Theory Dynam. Systems},
  volume={30},
  date={2010},
  number={4},
  pages={1147--1164},
  issn={0143-3857},
  review={\MRref {2669415}{2011j:46114}},
  doi={10.1017/S0143385709000509},
}

\bib{Leinster:Basic_Bicategories}{article}{
  author={Leinster, Tom},
  title={Basic Bicategories},
  date={1998},
  status={eprint},
  note={\arxiv {math/9810017}},
}

\bib{Lewin-Sims:Aperiodicity}{article}{
  author={Lewin, Peter},
  author={Sims, Aidan},
  title={Aperiodicity and cofinality for finitely aligned higher-rank graphs},
  journal={Math. Proc. Cambridge Philos. Soc.},
  volume={149},
  date={2010},
  number={2},
  pages={333--350},
  issn={0305-0041},
  review={\MRref {2670219}{2012b:46106}},
  doi={10.1017/S0305004110000034},
}

\bib{Li:Ring_Cstar}{article}{
  author={Li, Xin},
  title={Ring $C^*$\nobreakdash -algebras},
  journal={Math. Ann.},
  volume={348},
  date={2010},
  number={4},
  pages={859--898},
  issn={0025-5831},
  review={\MRref {2721644}{2012a:46100}},
  doi={10.1007/s00208-010-0502-x},
}

\bib{Li:Semigroup_amenability}{article}{
  author={Li, Xin},
  title={Semigroup $\textup C^*$\nobreakdash -algebras and amenability of semigroups},
  journal={J. Funct. Anal.},
  volume={262},
  date={2012},
  number={10},
  pages={4302--4340},
  issn={0022-1236},
  review={\MRref {2900468}{}},
  doi={10.1016/j.jfa.2012.02.020},
}

\bib{Li:Nuclearity_semigroup}{article}{
  author={Li, Xin},
  title={Nuclearity of semigroup $C^*$\nobreakdash -algebras and the connection to amenability},
  journal={Adv. Math.},
  volume={244},
  date={2013},
  pages={626--662},
  issn={0001-8708},
  review={\MRref {3077884}{}},
  doi={10.1016/j.aim.2013.05.016},
}

\bib{MacLane:Categories}{book}{
  author={MacLane, Saunders},
  title={Categories for the working mathematician},
  note={Graduate Texts in Mathematics, Vol. 5},
  publisher={Springer},
  place={New York},
  date={1971},
  pages={ix+262},
  review={\MRref {0354798}{50\,\#7275}},
  doi={10.1007/978-1-4757-4721-8},
}

\bib{Meyer:Generalized_Fixed}{article}{
  author={Meyer, Ralf},
  title={Generalized fixed point algebras and square-integrable groups actions},
  journal={J. Funct. Anal.},
  volume={186},
  date={2001},
  number={1},
  pages={167--195},
  issn={0022-1236},
  review={\MRref {1863296}{2002j:46086}},
  doi={10.1006/jfan.2001.3795},
}

\bib{Muhly-Solel:Tensor}{article}{
  author={Muhly, Paul S.},
  author={Solel, Baruch},
  title={Tensor algebras over $C^*$\nobreakdash -correspondences: representations, dilations, and $C^*$\nobreakdash -envelopes},
  journal={J. Funct. Anal.},
  volume={158},
  date={1998},
  number={2},
  pages={389--457},
  issn={0022-1236},
  review={\MRref {1648483}{99j:46066}},
  doi={10.1006/jfan.1998.3294},
}

\bib{Muhly-Tomforde:Quivers}{article}{
  author={Muhly, Paul S.},
  author={Tomforde, Mark},
  title={Topological quivers},
  journal={Internat. J. Math.},
  volume={16},
  date={2005},
  number={7},
  pages={693--755},
  issn={0129-167X},
  review={\MRref {2158956}{2006i:46099}},
  doi={10.1142/S0129167X05003077},
}

\bib{Muhly-Renault-Williams:Equivalence}{article}{
  author={Muhly, Paul S.},
  author={Renault, Jean N.},
  author={Williams, Dana P.},
  title={Equivalence and isomorphism for groupoid \(C^*\)\nobreakdash -algebras},
  journal={J. Operator Theory},
  volume={17},
  date={1987},
  number={1},
  pages={3--22},
  issn={0379-4024},
  review={\MRref {873460}{88h:46123}},
  eprint={http://www.theta.ro/jot/archive/1987-017-001/1987-017-001-001.html},
}

\bib{Murphy:Crossed_semigroups}{article}{
  author={Murphy, Gerard J.},
  title={Crossed products of $C^*$\nobreakdash -algebras by semigroups of automorphisms},
  journal={Proc. London Math. Soc. (3)},
  volume={68},
  date={1994},
  number={2},
  pages={423--448},
  issn={0024-6115},
  review={\MRref {1253510}{94m:46105}},
  doi={10.1112/plms/s3-68.2.423},
}

\bib{Neshveyev:KMS_non-principal}{article}{
  author={Neshveyev, Sergey V.},
  title={KMS states on the $C^*$\nobreakdash -algebras of non-principal groupoids},
  journal={J. Operator Theory},
  volume={70},
  date={2013},
  number={2},
  pages={513--530},
  issn={0379-4024},
  review={\MRref {3138368}{}},
  doi={10.7900/jot.2011sep20.1915},
}

\bib{Paterson:Groupoids}{book}{
  author={Paterson, Alan L. T.},
  title={Groupoids, inverse semigroups, and their operator algebras},
  series={Progress in Mathematics},
  volume={170},
  publisher={Birkh\"auser Boston Inc.},
  place={Boston, MA},
  date={1999},
  pages={xvi+274},
  isbn={0-8176-4051-7},
  review={\MRref {1724106}{2001a:22003}},
  doi={10.1007/978-1-4612-1774-9},
}

\bib{Pimsner:Generalizing_Cuntz-Krieger}{article}{
  author={Pimsner, Mihai V.},
  title={A class of $C^*$\nobreakdash -algebras generalizing both Cuntz--Krieger algebras and crossed products by~$\mathbf Z$},
  conference={ title={Free probability theory}, address={Waterloo, ON}, date={1995}, },
  book={ series={Fields Inst. Commun.}, volume={12}, publisher={Amer. Math. Soc.}, place={Providence, RI}, },
  date={1997},
  pages={189--212},
  review={\MRref {1426840}{97k:46069}},
}

\bib{Raeburn:Graph_algebras}{book}{
  author={Raeburn, Iain},
  title={Graph algebras},
  series={CBMS Regional Conference Series in Mathematics},
  volume={103},
  publisher={Amer. Math. Soc.},
  place={Providence, RI},
  date={2005},
  pages={vi+113},
  isbn={0-8218-3660-9},
  review={\MRref {2135030}{2005k:46141}},
}

\bib{Raeburn-Sims:Product_graphs}{article}{
  author={Raeburn, Iain},
  author={Sims, Aidan},
  title={Product systems of graphs and the Toeplitz algebras of higher-rank graphs},
  journal={J. Operator Theory},
  volume={53},
  date={2005},
  number={2},
  pages={399--429},
  issn={0379-4024},
  review={\MRref {2153156}{2006d:46073}},
  eprint={http://www.theta.ro/jot/archive/2005-053-002/2005-053-002-010.html},
}

\bib{Raeburn-Sims-Yeend:Higher_graph}{article}{
  author={Raeburn, Iain},
  author={Sims, Aidan},
  author={Yeend, Trent},
  title={Higher-rank graphs and their $C^*$\nobreakdash -algebras},
  journal={Proc. Edinb. Math. Soc. (2)},
  volume={46},
  date={2003},
  number={1},
  pages={99--115},
  issn={0013-0915},
  review={\MRref {1961175}{2004f:46068}},
  doi={10.1017/S0013091501000645},
}

\bib{Renault:Groupoid_Cstar}{book}{
  author={Renault, Jean},
  title={A groupoid approach to $\textup C^*$\nobreakdash -algebras},
  series={Lecture Notes in Mathematics},
  volume={793},
  publisher={Springer},
  place={Berlin},
  date={1980},
  pages={ii+160},
  isbn={3-540-09977-8},
  review={\MRref {584266}{82h:46075}},
  doi={10.1007/BFb0091072},
}

\bib{Renault:Representations}{article}{
  author={Renault, Jean},
  title={Repr\'esentation des produits crois\'es d'alg\`ebres de groupo\"\i des},
  journal={J. Operator Theory},
  volume={18},
  date={1987},
  number={1},
  pages={67--97},
  issn={0379-4024},
  review={\MRref {912813}{89g:46108}},
  eprint={http://www.theta.ro/jot/archive/1987-018-001/1987-018-001-005.html},
}

\bib{Renault:Radon-Nikodym_AP}{article}{
  author={Renault, Jean},
  title={The Radon--Nikod\'ym problem for approximately proper equivalence relations},
  journal={Ergodic Theory Dynam. Systems},
  volume={25},
  date={2005},
  number={5},
  pages={1643--1672},
  issn={0143-3857},
  review={\MRref {2173437}{2006h:46065}},
  doi={10.1017/S0143385705000131},
}

\bib{Renault:Cartan.Subalgebras}{article}{
  author={Renault, Jean},
  title={Cartan subalgebras in $C^*$\nobreakdash -algebras},
  journal={Irish Math. Soc. Bull.},
  number={61},
  date={2008},
  pages={29--63},
  issn={0791-5578},
  review={\MRref {2460017}{2009k:46135}},
  eprint={http://www.maths.tcd.ie/pub/ims/bull61/S6101.pdf},
}

\bib{Renault-Sims-Williams-Yeend:Uniqueness}{article}{
  author={Renault, Jean N.},
  author={Sims, Aidan},
  author={Williams, Dana P.},
  author={Yeend, Trent},
  title={Uniqueness theorems for topological higher-rank graph C*\nobreakdash -algebras},
  status={eprint},
  note={\arxiv {0906.0829v3}},
  date={2012},
}

\bib{Shotwell:Simplicity}{article}{
  author={Shotwell, Jacob},
  title={Simplicity of finitely aligned $k$\nobreakdash -graph $C^*$\nobreakdash -algebras},
  journal={J. Operator Theory},
  volume={67},
  date={2012},
  number={2},
  pages={335--347},
  issn={0379-4024},
  review={\MRref {2928319}{}},
  eprint={http://www.mathjournals.org/jot/2012-067-002/2012-067-002-004.pdf},
}

\bib{Thomsen:KMS_weights}{article}{
  author={Thomsen, Klaus},
  title={KMS weights on groupoid and graph $C^*$\nobreakdash -algebras},
  journal={J. Funct. Anal.},
  volume={266},
  date={2014},
  number={5},
  pages={2959--2988},
  issn={0022-1236},
  review={\MRref {3158715}{}},
  doi={10.1016/j.jfa.2013.10.008},
}

\bib{Tu:BC_moyennable}{article}{
  author={Tu, Jean-Louis},
  title={La conjecture de Baum--Connes pour les feuilletages moyennables},
  journal={\(K\)\nobreakdash -Theory},
  volume={17},
  date={1999},
  number={3},
  pages={215--264},
  issn={0920-3036},
  review={\MRref {1703305}{2000g:19004}},
  doi={10.1023/A:1007744304422},
}

\bib{Wright:Aperiodicity-conditions-topological-graphs}{article}{
  author={Wright, Sarah Elizabeth},
  title={Aperiodicity conditions in topological $k$\nobreakdash -graphs},
  journal={J. Operator Theory},
  volume={72},
  date={2014},
  number={1},
  pages={3--14},
  issn={0379-4024},
  review={\MRref {3246978}{}},
  doi={10.7900/jot.2012aug20.196},
}

\bib{Yeend:Topological-higher-rank-graphs}{article}{
  author={Yeend, Trent},
  title={Topological higher-rank graphs and the $C^*$\nobreakdash -algebras of topological 1\nobreakdash -graphs},
  book={ title={Operator theory, operator algebras, and applications}, series={Contemp. Math.}, volume={414}, publisher={Amer. Math. Soc.}, place={Providence, RI}, date={2006}, },
  pages={231--244},
  review={\MRref {2277214}{2007j:46100}},
  doi={10.1090/conm/414},
}

\bib{Yeend:Groupoid_models}{article}{
  author={Yeend, Trent},
  title={Groupoid models for the $C^*$\nobreakdash -algebras of topological higher-rank graphs},
  journal={J. Operator Theory},
  volume={57},
  date={2007},
  number={1},
  pages={95--120},
  issn={0379-4024},
  review={\MRref {2301938}{2008f:46074}},
  eprint={http://www.theta.ro/jot/archive/2007-057-001/2007-057-001-005.html},
}
  \end{biblist}
\end{bibdiv}
\end{document}